\documentclass[12pt,twoside,english]{article}
\usepackage{babel,amssymb,amsthm}
\usepackage{graphicx}
\usepackage{epstopdf}
\usepackage{amsmath}
\usepackage{bm}
\usepackage{float, framed}
\usepackage[utf8]{inputenc}
\usepackage{wrapfig}
\usepackage{pgffor} 
\usepackage[font=small,labelfont=bf]{caption} 

\usepackage[show]{ed}

\newcommand{\triple}[2]{|\!|\!|#1|\!|\!|_{#2}}

\newcommand{\bs}{\boldsymbol}

\newcommand{\mb}{\mathbf}
\newcommand{\mc}{\mathcal}
\newcommand{\mr}{\mathrm}

\newcommand{\R}{\mathbb R}
\newcommand{\C}{\mathbb C}
\newcommand{\CC}{{\mathrm C}}
\newcommand{\MM}{{\mathrm M}}
\newcommand{\II}{{\mathrm I}}
\newcommand{\eps}[1]{\bs\varepsilon(#1)}
\newcommand{\uu}{{\mathbf u}}
\newcommand{\ww}{{\mathbf w}}

\newcommand{\HH}{{\mathbf H}}

\newcommand{\matC}{\mathbb C^{d\times d}}
\newcommand{\matR}{\mathbb R^{d\times d}}
\newcommand{\symm}{{\mathrm{sym}}}
\newcommand{\skw}{{\mathrm{skw}}}
\newcommand{\diff}{{\mathrm{diff}}}
\newcommand{\Om}{\Omega}

\newcommand{\Cdiff}{{\mathrm C_\diff}}
\newcommand{\Ltwo}{\mathbb L^2(\Omega)}

\renewcommand{\Re}{\mathrm{Re}\,}

\newtheorem{proposition}{Proposition}[section]
\newtheorem{corollary}[proposition]{Corollary}
\newtheorem{lemma}[proposition]{Lemma}
\newtheorem{theorem}[proposition]{Theorem}

\newtheorem{hypothesis}{Hypothesis}

\newcommand{\abssec}[1]{\noindent\normalsize {\bfseries #1\quad }\ignorespaces}
\newenvironment{keywords}{\abssec{Key Words}}{\par\vspace{.1in}}
\newenvironment{AMS}{\abssec{AMS subject
		classification}}{\par\vspace{.1in}}

\numberwithin{equation}{section}

\title{Analysis of models for viscoelastic wave propagation}

\date{\today}

\author{Thomas Brown\footnote{Department of Mathematical Sciences, University of Delaware. {\tt tsbrown@udel.edu}}, Shukai Du\footnote{Department of Mathematical Sciences, University of Delaware. {\tt shukaidu@udel.edu}}, Hasan Eruslu\footnote{Department of Mathematical Sciences, University of Delaware. {\tt heruslu@udel.edu}}, Francisco-Javier Sayas\footnote{Department of Mathematical Sciences, University of Delaware. {\tt fjsayas@udel.edu}}}

\begin{document}

\maketitle

\begin{abstract}
We consider the problem of waves propagating in a viscoelastic solid.  For the material properties of the solid we consider both classical and fractional differentiation in time versions of the Zener, Maxwell, and Voigt models, where the coupling of different models within the same solid are covered as well.  Stability of each model is investigated in the Laplace domain, and these are then translated to time-domain estimates.   With the use of semigroup theory, some time-domain results are also given which avoid using the Laplace transform and give sharper estimates.   We take the time to develop and explain the theory necessary to understand the relation between the equations we solve in the Laplace domain and those in the time-domain which are written using the language of causal tempered distributions.  Finally we offer some numerical experiments that highlight some of the differences between the models and how different parameters effect the results.  
\end{abstract}



\section{Introduction}

This paper offers a thorough introduction to mathematical tools to describe wave propagation in solids modeled with a wide collection of viscoelastic laws. Before we even attempt a general description of the models we will be addressing, let us emphasize what our goals are and what has not been tackled in the present paper. We aim for a unified mathematical description of a wide collection of known viscoelastic models, including basic well-posedness results. The models will include all classical viscoelastic wave models, fractional versions thereof, and couplings of different models in different subregions. The techniques that we will employ in the first part of the article (Sections \ref{sec:3} through \ref{sec:7}) are those of Laplace transforms, understanding the influence of volume forces, normal stresses, and given displacements, as well as the strain-stress relation as transfer functions describing linear distributional processes in the time domain. These techniques are classical, although we will use them in a language that is borrowed from the recent literature of time domain integral equations. In a second part of the paper (Sections \ref{sec:8} to \ref{sec:10}), we will introduce and use tools from the theory of strongly continuous semigroups to analyze the three classical models and some transitional situations where, for instance, classical Zener-style viscoelasticity coexists with pure linear elasticity in different subdomains, with smooth or abrupt transition regions. Our goals for the current piece of work are not in the realm of the modeling: we will analyze but not discuss known models, and we will not deal with physical justifications thereof. A particular issue where we will be very restrictive is the fact that we will only deal with solids moving from equilibrium (no displacement, strain, or stress) at time zero. There are practical reasons for this choice (since stress remembers past strain, it is not entirely justifiable to start the clock with known displacement and stress), but we are also restricted because of our analysis goals. In both parts (transfer function analysis and semigroup analysis) we will give a hint at how to deal with initial conditions.

Barring initial and boundary conditions that are needed to fully describe the model, our goal is the study of a linear elasticity equation
\[
\rho \ddot{\mathbf u}=\mathrm{div}\,\boldsymbol\sigma+\mathbf f
\]
in a bounded domain of $d$-dimensional space. Here $\mathbf u$ is the displacement field, upper dots denote time differentiation, $\boldsymbol\sigma$ is the stress tensor and $\mathbf f$ represents the volumetric forces. Linear strain $\boldsymbol\varepsilon=\frac12(\nabla \mathbf u + ( \nabla \mathbf u)^\top)$ determines stress through a generic convolutional law (we only display the time variable in the following formulas)
\[
\boldsymbol\sigma(t)=\int_0^t \mathcal D(t-\tau){\dot{\boldsymbol\varepsilon}(\tau)}\mathrm d\tau,
\]
where $\mathcal D$ is a time-dependent, possibly distributional, tensor-valued kernel. Following the careful description given by Francesco Mainardi in his monograph \cite{Mainardi2010}, the four classical models of viscoelasticity are named after Zener, Voigt, Maxwell, and Newton. The strain-to-stress relationship is given by a differential equation as follows:
\begin{alignat*}{6}
\boldsymbol\sigma+a\,\dot{\boldsymbol\sigma}
&=\mathrm C_0 \boldsymbol\varepsilon+\mathrm C_1\dot{\boldsymbol\varepsilon}, & \qquad &
\mbox{(Zener)} \\
\boldsymbol\sigma
&=\mathrm C_0 \boldsymbol\varepsilon+\mathrm C_1\dot{\boldsymbol\varepsilon}, & \qquad &
\mbox{(Voigt)}\\
\boldsymbol\sigma+a\,\dot{\boldsymbol\sigma}
&=\mathrm C_1\dot{\boldsymbol\varepsilon}, & \qquad &
\mbox{(Maxwell)}\\
\boldsymbol\sigma
&=\mathrm C_1\dot{\boldsymbol\varepsilon}, & \qquad &
\mbox{(Newton)}\\
\boldsymbol\sigma
&=\mathrm C_0\boldsymbol\varepsilon. & \qquad &
\mbox{(Linear elasticity)}
\end{alignat*}
Here $a$ is a non-negative function and $\mathrm C_0$ and $\mathrm C_1$ are four-index tensors satisfying hypotheses similar to the tensor that is used to describe linear elasticity (we will be covering heterogeneous anisotropic solids), with some additional conditions that will be introduced when we explain the models in detail.
Formally speaking, the last three models can be considered as particular examples of Zener's model. However, they have very different properties. Newton's model is equivalent to a parabolic equation, as can be seen by substituting the formula for $\boldsymbol\sigma=\mathrm C_1\dot{\boldsymbol\varepsilon}$ in the equation of conservation of momentum and integrating once in the time variable. This model therefore does not produce waves and will be ignored in the sequel. Voigt's model also gives an explicit strain-to-stress relation: if we substitute $\boldsymbol\sigma=\mathrm C_0 \boldsymbol\varepsilon+\mathrm C_1\dot{\boldsymbol\varepsilon}$ in the dynamical equation we can see that we end up with a PDE of order three (there are terms with two derivatives in space and one in time). The models of Zener, Maxwell, and Voigt dissipate energy.

Let us now give a quick literature review, including some relevant work from the modeling and mathematical communities. Some of the monographs \cite{Mainardi1982, Mainardi2010, AtPiStZo2014a} contain a large collection of references that can be used for a deeper introduction to this fascinating area.  For a very early introduction to linear viscoelasticity, see \cite{GuSt1962}. A generalized model that encompasses all classical models describe above four is examined in \cite[Chapter 3]{AtPiStZo2014a}.  An overview of the physics of viscoelasticity and its relation to rheology (fluid and soft solid flow) can be found in \cite{Phan2013} and \cite{MaCr2012}, where the latter also develops numerical methods for solving problems associated to viscoelasticity, while an overview of the mathematical theories and techniques, including the problem of waves propagating in viscoelastic media can be found in \cite{FaMo1992}.  Viscoelastic models have also been formulated in the language of integral equations (see \cite{ShNe2014, FaMo1992, KeHe2014}), and electrostatic models like the Cole-Cole model in \cite{LuHa2004}.  

We wish to consider waves propagating in viscoelastic solids using both the classical models and those using fractional derivatives.  The relation between fractional derivatives and viscoelastic models (including introductions to the Mittag-Leffler functions) can be found in \cite{Mainardi2010, MaGo2000}, and \cite{Mainardi2012} offers a short survey of the history of development of this theory. Mainardi, who it can easily be seen is at the center of much of the development and communication of waves in viscoelasticity, shows in \cite{Mainardi2018} that the fractional relaxation process with constant coefficients is equivalent to a similar process governed by variable coefficient ODE.  This relationship has also been explored in \cite{LuHa2004}, where fractional derivatives are bypassed by using a method of Yuan and Agrawal to convert the equations to a system of first order ODE.  The Zener model has been extensively studied in the context of waves with both classical \cite{AtPiStZo2014b, BeEzJo2004, Lee2017} and fractional \cite{NaHo2013} time derivatives, as well as in the quasistatic case \cite{RiShWhWh2003, PeKa2014}.  Other authors have explored the Maxwell and fractional Maxwell models \cite{CoGiMa2017b}, different variations of Maxwell models \cite{CoGiMa2017a, GiCo2018}, and the Voigt model \cite{BeEzJo2004,KeHe2014,LiZhLu2016}. 

For some of our stability results, we make rigorous use of Laplace transforms for vector-valued distrivutions.  Laplace transforms also appear in \cite{AtPiStZo2014b, FaMo1992, GiCo2018, GuSt1962, NaHo2013, PeKa2014}, used in the context of showing existence and uniqueness of solutions, justifications of models, exploration of the constitutive relations, or to simplify numerical implementation.   In the context of stability analysis, a Green's function representation of the solution to the three dimensional wave problem is used in \cite{FaMo1992}.  While we obtain some estimates in the time domain by making use of an inversion theorem for the Laplace transformation (in the spirit of the Payley-Wiener theorems), in some cases can make use of semigroup theory to obtain estimates directly.  Similar analysis has been performed for waves in unbounded domains \cite{BeEzJo2004} and in bounded domains \cite{Lee2017}, where  semigroup analysis is used to show the existence and uniqueness of solutions for waves in a Zener model.   

While not explored in detail in the present work, we are also interested in the analysis and implementation of numerical schemes for the simulation of waves in viscoelastic materials.  As mentioned earlier, \cite{MaCr2012} contains an overview of numerical methods for problems in viscoelasticity including finite element, boundary element, and finite volume formulations.  Numerical implementation with finite elements has also been explored in \cite{KeHe2014} and \cite{YuPeKa2016} for the simulation and comparison to real world data, specifically blood flow in \cite{YuPeKa2016}.  Also in the context of blood flow, \cite{PeKa2014} uses discontinuous Galerkin methods to simulate a quasistatic nonlinear 1D fractional Zener model. A DG method for a general linear quasistatic viscoelastic model is proposed in \cite{RiShWhWh2003} and a priori error estimates are derived.  Convergence of finite element methods for viscoelasticity is explored \cite{LaRaSa2015} and \cite{LiZhLu2016, Lee2017}, where the first reference focuses on convergence in time while the latter two are concerned with optimal order of convergence in space.  Coupling of elastic and viscoelastic subdomains is examined in \cite{MeCo2003} where boundary elements for the viscoelastic subdomain are coupled with finite elements for the elastic components, whereas in \cite{TrPe2014} a scheme involving only finite elements are used for the same problem and the two schemes are compared. 

Our paper is structured as follows. After introducing the general model  (Section \ref{sec:2}), we give a general framework for the viscoelastic material law as a transfer function (Section \ref{sec:3}) and then move on to prove that the main classical models (Zener, Maxwell, Voigt), fractional versions of them, and combinations of different models in different subdomains, fit in our general framework (Section \ref{sec:4}). Sections \ref{sec:5}-\ref{sec:7} contain the Laplace domain analysis of the model carried out as follows: first we do a transfer function analysis, then we give the general theory of how to understand transfer functions as Laplace transforms of distributional convolutions in the time variable, and finally we give estimates for the case of smooth (in time) data. In Sections \ref{sec:8}-\ref{sec:10}, we start with a semigroup analysis of the classical models. Because we are striving for generality, we make an effort to include models where the classical viscoelastic models can degenerate into classical elasticity, which motivates a careful discussion of the closure process of a normed space with respect to a certain seminorm and how this affects the action of some operators. Section \ref{sec:SEMI} gives a detailed treatment of Zener's model, using the tools of the previous section and well-known results of the theory of strongly continuous semigroups in Hilbert spaces. Section \ref{sec:10} sketches the main changes that need to be made to the preceding analysis to study Voigt's and Maxwell's models. Finally, and just for the sake of illustration, we show some simulations for one and three dimensional models. 

Before we proceed with the work at hand, let us give here some quick notational pointers.
While the transient models we will be describing and analysing in this paper take values on spaces of real valued functions, the transfer function analysis will require the introduction of complex variables and complex-valued functions. To be on the safe side, {\em all brackets and forms considered in this paper will be linear or bilinear}, never conjugate linear or sesquilinear.
We will write
\[
\mathrm A:\mathrm B=\sum_{i,j=1}^d a_{ij}b_{ij},
\qquad \mathrm A,\mathrm B\in \matC,
\]
with no conjugation involved. 
The upperscript $\top$ will be used for transposition of matrices, without conjugation. 
Note that $\mathrm A:\mathrm B=0$ for all $\mathrm A\in \matC_\symm:=\{ \mathrm A\in \matC\,:\, \mathrm A^\top=\mathrm A\}$, and $\mathrm B\in \matC_\skw:=\{\mathrm B\in \matC\,:\, \mathrm B^\top=-\mathrm B\}$. We will write $\|\MM\|^2=\MM:\overline\MM$. 

Given two Banach spaces $X$ and $Y$, we will consider the space $\mathcal B(X,Y)$ of bounded linear maps from $X$ to $Y$ with the operator norm. We will shorten $\mathcal B(X):=\mathcal B(X,X)$.

\section{An introduction to the model problem}\label{sec:2}

The wave propagation problem will be given in an open bounded domain $\Omega\subset \mathbb R^d$, whose boundary is denoted by $\Gamma$. In order to have a well defined trace operator in classical Sobolev spaces, we will assume that $\Omega$ is locally a Lipschitz hypograph, although this hypothesis can be relaxed as long as we have a trace operator.
We will assume that $\Gamma$ is decomposed into Dirichlet and Neumann parts, $\Gamma_D$ and $\Gamma_N$, satisfying 
\[
\overline{\Gamma_D}\cup \overline{\Gamma_N}=\Gamma, \qquad \Gamma_D\cap\Gamma_N=\emptyset.
\]

The inner products in the Lebesgue spaces
\[
L^2(\Omega), \quad \mathbf L^2(\Omega):=L^2(\Omega;\R^d), \quad 
\mathbb L^2(\Omega):=L^2(\Omega;\R^{d\times d}_\symm)
\]
(note that the latter is a space of symmetric-matrix-valued functions) will be respectively denoted by
\[
(u,v)_\Omega:=\int_\Omega u\,v,\qquad
(\uu,\mathbf v)_\Omega:=\int_\Omega \uu\cdot\mathbf v,\qquad
(\mathrm S,\mathrm T)_\Omega:=\int_\Omega \mathrm S\,:\,\mathrm T,
\]
and $\|\cdot\|_\Omega$ will denote the associated norm in all three cases. We will also consider the Sobolev space
\[
\HH^1(\Omega)=H^1(\Omega;\R^d) 
:=\{ \uu :\Omega \to \R^d\,:\, \uu\in \mathbf L^2(\Omega),\, 
\nabla \uu \in L^2(\Omega;\R^{d\times d})\},
\]
endowed with the norm 
\[
\| \uu\|_{1,\Omega}^2:=\| \uu\|_\Omega^2+\|\nabla \uu\|_\Omega^2,
\]
and the symmetric gradient operator
\[
\HH^1(\Omega)\ni \uu\longmapsto 
\bs\varepsilon(\uu):=\tfrac12(\nabla \uu+(\nabla\uu)^\top)
\in \mathbb L^2(\Omega).
\]

We can define a bounded and surjective trace operator $\gamma:\HH^1(\Omega)\to \HH^{1/2}(\Gamma)$.
For simplicity, we will write $\gamma_D\uu:=\gamma\uu|_{\Gamma_D}$.
We then consider the Sobolev spaces:
\begin{alignat*}{6}
\HH^1_D(\Omega) &:= \{ \ww\in \HH^1(\Omega)\,:\, \gamma_D \ww=0 \}=\ker\gamma_D,\\
\HH^{1/2}(\Gamma_D) &:= \{ \gamma_D\uu\,:\, \uu \in \HH^1(\Omega)\}
	=\mathrm{range}\,\gamma_D,\\
\widetilde\HH^{1/2}(\Gamma_N) &:= \{\gamma\ww|_{\Gamma_N}\,:\, \ww\in \HH^1_D(\Omega)\},\\
\HH^{-1/2}(\Gamma_N) &:= \widetilde\HH^{1/2}(\Gamma_N)'.
\end{alignat*}
Our exposition will include the cases where either $\Gamma_D$ or $\Gamma_N$ is empty. To be completely precise, $\HH^{-1/2}(\Gamma_N)$ is the representation of the dual of $\widetilde\HH^{1/2}(\Gamma_N)$ making
\[
\widetilde\HH^{1/2}(\Gamma_N)
\subset L^2(\Gamma_N;\R^d)
\subset \HH^{-1/2}(\Gamma_N)
\]
a well-defined Gelfand triple. The reciprocal duality product of the two fractional spaces on $\Gamma_N$ will be denoted with the angled bracket
$\langle\cdot,\cdot\rangle_{\Gamma_N}$. 

We will consider the space for symmetric-tensor-valued functions
\[
\mathbb H(\mathrm{div},\Omega):=
\{ \mathrm S\in \mathbb L^2(\Omega)\,:\,
\mathrm{div}\,\mathrm S \in \mathbf L^2(\Omega)\}.
\]
In the above definition, the divergence operator is applied to the rows of the matrix valued function $\mathrm S$. Following well-known results on Sobolev spaces, we can define a bounded linear and surjective operator $\gamma_N:\mathbb H(\mathrm{div},\Omega)\to \HH^{-1/2}(\Gamma_N)$ so that the following weak formulation of Betti's formula
\begin{alignat*}{6}
\langle \gamma_N \mathrm S,\gamma\uu\rangle_{\Gamma_N}
& =(\mathrm S,\nabla\uu)_\Omega+(\mathrm{div}\,\mathrm S,\uu)_\Omega\\
&=(\mathrm S,\bs\varepsilon(\uu))_\Omega+(\mathrm{div}\,\mathrm S,\uu)_\Omega \qquad
\forall \mathrm S\in \mathbb H(\mathrm{div},\Omega)
\quad\forall \uu\in \HH^1_D(\Omega),
\end{alignat*}
holds. 

Pending a precise introduction of the material law, which we will give in the Laplace domain in Section \ref{sec:3}, we are now ready to give a functional form for the {\bf viscoelastic wave propagation problem}. We look for $\uu:[0,\infty)\to \HH^1(\Omega)$ and $\bs\sigma:[0,\infty)\to \mathbb H(\mathrm{div},\Omega)$ satisfying for all $t\ge 0$
\begin{subequations}\label{eq:2.1}
\begin{alignat}{6}
\rho\, \ddot\uu(t) &=\mathrm{div}\,\bs\sigma(t)+\mathbf f(t),\\
\gamma_D\uu(t) &=\boldsymbol\alpha(t), \\
\gamma_N\bs\sigma(t) &=\boldsymbol\beta(t).
\end{alignat}
Here $\rho\in L^\infty(\Omega)$, with $\rho\ge \rho_0$ almost everywhere for some positive constant $\rho_0$, models the mass density in the solid, which at the initial time $t=0$ is at rest on the reference configuration $\Omega$
\begin{equation}
\uu(0)=\mathbf 0, \qquad\dot\uu(0)=\mathbf 0.
\end{equation}
Upper dots are used to denote time derivatives.
The data are functions
\[
\mathbf f:[0,\infty)\to \mathbf L^2(\Omega),
\qquad
\bs\alpha:[0,\infty)\to \HH^{1/2}(\Gamma_D),\\
\qquad
\bs\beta:[0,\infty)\to \HH^{-1/2}(\Gamma_N).
\]
What characterizes the viscoelastic model (for small deformations where the strain at time $t$  can be described by $\bs\varepsilon(\uu(t))$) is the existence of a  strain-stress relation of the form
\begin{equation}\label{eq:2.1e}
\bs\sigma(t)=\int_0^t \mathcal D(t-\tau)\,\bs\varepsilon(\dot\uu(\tau))
\mathrm d\tau.
\end{equation}
\end{subequations}
The viscoelastic law \eqref{eq:2.1e} is {\em formally written} in terms of a convolutional kernel $\mathcal D$. In a first approximation, this kernel can be considered as a fourth order tensor (with some symmetric properties) depending on the time variable. As we will see later on, the most interesting examples arise when the causal convolution \eqref{eq:2.1e} is a distributional one and $\mathcal D$ is described as a causal tensor-valued distribution of the real variable. 

%% SECTION 3

\section{Viscoelastic laws in the Laplace domain}\label{sec:3}

In this section we are going to give a precise meaning to the general viscoelastic law \eqref{eq:2.1e}. Instead of writing the convolutional law \eqref{eq:2.1e} in the time domain, we are going to introduce a Laplace transformed model which we will analyze in detail. This model will then be used to justify a family of distributional models in the time domain. At this point, we consider the formal Laplace transform of the convolutional process \eqref{eq:2.1e} and introduce
\[
\mathrm C(s):=s\, \mathcal L\{ \mathcal D\}(s).
\]
(Note that multiplication by $s$ takes care of time differentiation.) Laplace transforms will be defined in the complex half-plane
\[
\C_+:=\{ s\in \C\,:\, \mathrm{Re}\,s>0\}.
\]

The {\bf viscoelastic material tensor} can be described as a transfer function through a holomorphic map
\[
\CC: \mathbb C_+\to \mathcal B(\matC;
L^\infty(\Omega;\matC))\equiv L^\infty(\Omega;\C^{d\times d\times d\times d}).
\]
Given $\MM\in \matC$ we thus have $\mathbb C_+\ni s\mapsto \CC(s)\MM\in L^\infty(\Omega;\matC)$. 

\begin{framed}
\begin{hypothesis}[Symmetry] \label{hypo1}
Almost everywhere in $\Omega$:
\begin{subequations}\label{eq:C.1}
\begin{alignat}{6}
\label{eq:C.1a}
\CC(\overline s)\overline\MM=\overline{\CC(s)\MM} 
	& \qquad & \forall s\in \C_+
	& \quad & \forall \MM\in \matC,\\
\label{eq:C.1b}
\CC(s)\MM \in \matC_\symm 
	& & \forall s\in \C_+ & & \forall \MM\in \matC,\\
\label{eq:C.1c}
\CC(s)\MM = \CC(s) (\tfrac12(\MM+\MM^\top)) 
	& & \forall s\in \C_+& & \forall \MM \in \matC,\\
\label{eq:C.1d}
\CC(s)\MM:\mathrm N=\CC(s)\mathrm N: \mathrm M
	& & \forall s\in \C_+ & & \forall \MM,\mathrm N\in \matC.
\end{alignat}
\end{subequations}
\end{hypothesis}
\end{framed}

Some easy observations: conditions \eqref{eq:C.1b} and \eqref{eq:C.1d} imply \eqref{eq:C.1c}; if $\MM\in \matC_\skw$, then $\CC(s)\MM=0$ for all $s$; and if $\MM\in \matR$, then $\CC(s)\MM\in \matR$ for all $s\in(0,\infty)$. 

\begin{framed}
\begin{hypothesis}[Positivity] \label{hypo2}
There exists a non-decreasing function $\psi:(0,\infty)\to (0,\infty)$ satisfying 
\begin{subequations}\label{eq:C.2}
\begin{equation}
\inf_{0<x<1} x^{-\ell} \psi(x) > 0, \qquad \ell>0,
\end{equation}
and such that almost everywhere in $\Omega$
\begin{equation}
\Re(\overline s\CC(s)\MM:\overline\MM)\ge \psi(\Re s)\, \|\MM\|^2
	\qquad \forall \MM \in \matC_\symm, \quad \forall s\in \C_+.
\end{equation}
\end{subequations}
\end{hypothesis}
\end{framed}

For each $s\in \C_+$, we take $\|\CC(s)\|$ to be the smallest real number so that, almost everywhere in $\Omega$,
\[
\| \CC(s)\MM\|\le \|\CC(s)\|\,\|\MM\| \qquad \forall \MM\in \matC_\symm.
\]

\begin{framed}
\begin{hypothesis}[Boundedness]\label{hypo3}
There exists an integer $r\ge 0$ and non-increasing function $\phi:(0,\infty)\to (0,\infty) $ such that
\begin{subequations}\label{eq:phiRe}
\begin{equation}
\sup_{0<x<1} x^k\phi(x) <\infty, \qquad k\ge 0,
\end{equation}
and  almost everywhere in $\Omega$
\begin{equation}
\| \CC(s)\|\le |s|^r \phi(\Re s) \qquad \forall s\in \C_+.
\end{equation}
\end{subequations}
\end{hypothesis}
\end{framed}

We can equivalently  introduce this material tensor with a collection of holomorphic functions (the {\bf material coefficients})
\[
C_{ijkl}:\C_+\to L^\infty(\Omega;\C), \qquad i,j,k,l=1,\ldots,d,
\]
satisfying
\begin{subequations}\label{eq:C.3}
\begin{alignat}{6}
C_{ijkl}(\overline s)=\overline{C_{ijkl}(s)}
	& \qquad & \forall s\in \C_+ & \quad & i,j,k,l=1,\ldots,d,\\
C_{ijkl}(s)=C_{jikl}(s)=C_{ijlk}(s)=C_{klij}(s)
	& \qquad & \forall s\in \C_+ & \quad & i,j,k,l=1,\ldots,d.
\end{alignat}
\end{subequations}
In this case, we just define
\[
(\CC(s)\MM)_{ij}=\sum_{k,l=1}^d C_{ijkl}(s)m_{kl} \qquad i,j=1,\ldots, d,
\]
and notice that \eqref{eq:C.1} is equivalent to \eqref{eq:C.3}. When we write the material tensor in terms of coefficients, we can take
\[
\| \CC(s)\|_{\max}:=d^2 \,\max_{i,j,k,l} \|C_{ijkl}(s)\|_{L^\infty(\Omega)}
\]
as an upper bound of $\|\CC(s)\|$ almost everywhere. With this point of view $\mathrm C(s)$ can be considered as an element of $L^\infty(\Omega;\C^{d\times d\times d\times d})$. 

The general {\bf viscoelastic material law} in the Laplace domain is
\begin{equation}\label{eq:C10}
\bs\sigma(\uu)=\CC(s)\eps\uu.
\end{equation}
Here and in the sequel, we will identify $\CC(s)$ with the associated `multiplication' operator $\CC(s)\in \mathcal B(L^2(\Omega;\matC_\symm))$, noticing that
\begin{equation}\label{eq:C20}
\|\CC(s)\|_{L^2\to L^2} \le \mathrm{ess}\sup \|\CC(s)\|\le \|\CC(s)\|_{\max} \qquad \forall s\in \C_+.
\end{equation}
The associated bilinear form is
\[
a(\uu,\ww;s):=(\CC(s)\eps\uu,\eps\ww)_\Omega.
\]
It is clear that $a:\HH^1(\Omega;\C)\times \HH^1(\Omega;\C)\to \C$ is bilinear, bounded, symmetric, and satisfies
\begin{alignat}{6}
\label{eq:C.4}
|a(\uu,\ww;s)| & \le |s|^r \phi(\Re s) \,\|\eps\uu\|_{\Omega}\|\eps\ww\|_{\Omega}
	& \qquad &\forall \uu,\ww\in \HH^1(\Omega;\C), \quad s\in \C_+,
	\\
a(\uu,\ww;s) &=(\CC(s)\eps\uu,\nabla \ww)_\Omega
	 & \qquad & \forall \uu,\ww\in \HH^1(\Omega;\C), \quad s\in \C_+,
	 \\
\label{eq:3.9}
\Re a(\uu,\overline{s\,\uu};s)
 & \ge \psi( \Re s)\|\eps\uu\|_\Omega^2 
 	& \qquad & \forall \uu\in \HH^1(\Omega;\C)\quad s\in \C_+.
\end{alignat}
A precise time domain description of the strain-stress material law \eqref{eq:C10} will require the introduction of some tools of the theory of operator valued distributions. We will do this in Section \ref{sec:DIST}.

\section{Examples}\label{sec:4}

Before detailing the main examples covered with our theory, let us introduce a definition that will make our exposition simpler. Let
\begin{subequations}\label{eq:1.8}
\begin{equation}
\CC \in \mathcal B(\matR;L^\infty(\Omega;\matR))\equiv L^\infty(\Omega;\mathbb R^{d\times d\times d\times d})
\end{equation} 
satisfy almost everywhere in $\Omega$
\begin{alignat}{6}
\CC \MM \in \matR_\symm & \qquad & & \forall \MM\in \matR,\\
\CC\MM:\mathrm N=\CC \mathrm N:\MM & & & \forall \MM,\mathrm N\in \matR,\\
\CC\MM:\MM\ge c\|\MM\|^2 & & & \forall \MM \in \matR_\symm,
\end{alignat}
\end{subequations}
where $c>0$ is a constant. For simplicity, in the future, we will write $\CC  \ge c $ to refer to the last inequality and we will say that $\CC $ is a {\bf steady Hookean material model}, when conditions \eqref{eq:1.8} are satisfied. The constant $c>0$ will be called a lower bound for the model. Note that we can apply the model to complex-valued matrices
\[
\CC(\MM_{\mathrm{re}}+\imath \MM_{\mathrm{im}})
:=\CC\MM_{\mathrm{re}}+\imath \CC\MM_{\mathrm{im}}.
\]
When hypotheses \eqref{eq:1.8} are satisfied with $c=0$, we will call $\CC$ a non-negative Hookean model.
For a steady Hookean model $\CC$, we will write
\[
\|\CC\|_{\max}:=d^2\max_{i,j,k,l}\|C_{ijkl}\|_{L^\infty(\Omega)}.
\]

\begin{lemma}
Let $\CC$ be a steady Hookean model. Then:
\begin{itemize}
\item[{\rm (a)}] If $\mathrm N\in \matC_\skw$, then $\CC\mathrm N=0$ almost everywhere.
\item[{\rm (b)}] Almost everywhere in $\Omega$
\[
\CC\MM =\CC (\tfrac12(\MM+\MM^\top)) \qquad \forall \MM\in \matC.
\]
\item[{\rm (c)}] Almost everywhere in $\Omega$
\[
\CC\MM:\overline\MM \ge c \,\MM:\overline\MM = 
c(\|\MM_{\mathrm{re}}\|^2+\|\MM_{\mathrm{im}}\|^2)
\qquad\forall \MM\in \matC_\symm.
\]
\item[{\rm (d)}] If $a\in L^\infty(\Omega)$ satisfies $a\ge a_0>0$ almost everywhere, then $a\,\CC$ is a steady Hookean model.
\end{itemize}
\end{lemma}

\begin{proof}
Note first that $\CC\MM\in \matC_\symm$ almost everywhere for all $\MM\in \matC$. Therefore, if $\mathrm N\in \matC_\skw$, then 
\[
0=\CC\MM :\mathrm N=\MM: \CC\mathrm N \qquad \forall \MM\in \matC,
\]
which implies $\CC\mathrm N=0$. Property (b) follows from (a). Properties (c) and (d) are straightforward.
\end{proof}

\subsection{Elastic models}

In the case where we take a Hookean model $\CC_0$ and we consider the constant function $\CC(s)\equiv\CC_0$, it is simple to see that the Hypotheses \ref{hypo1}-\ref{hypo3} are satisfied with $\psi(x):=c_0\,x$ ($c_0$ being the lower bound for the model $\CC_0$) and \eqref{eq:phiRe}, $r=0$, and $\phi(x):= \|\CC_0\|$. The time domain version of this model is the usual linear strain-stress relation
\[
\bs\sigma(t)=\CC_0 \bs\varepsilon(\uu(t)).
\]

\subsection{Zener's classical viscoelastic model}\label{sec:Zener}

We  now consider the material law
\[
\CC(s)=(1+a\,s)^{-1} (\CC_0+ s\,\CC_1),
\]
where $a\in L^\infty(\Omega)$ is strictly positive, $\CC_0$ and $\CC_1$ are steady Hookean material models, with
\[
\CC_\diff:=\CC_1-a\,\CC_0 \ge 0,
\]
that is, almost everywhere
\[
\CC_1\MM:\MM\ge a\,\CC_0\MM:\MM \qquad \forall \MM\in \matR_\symm.
\]
This hypothesis makes $\CC_\diff$ a non-strict Hookean model. As we will see in the direct time-domain analysis of Section \ref{sec:SEMI}, $\CC_\diff$ is the diffusive part of the elastic model, while $\CC_0$ acts as a base or ground elastic model. 
We will make use of the formula
\[
(1+as)^{-1}(\CC_0+s\CC_1)=\CC_0+s(1+as)^{-1}\CC_\diff.
\]
The Laplace domain stress-strain relation can be written in implicit form
\[
\bs\sigma+a\,s\,\bs\sigma=\CC_0\eps\uu+ s \CC_1\eps\uu,
\]
corresponding to the differential relation in the time domain
\[
\bs\sigma(t)+a\,\dot{\bs\sigma}(t)=
\CC_0\eps{\uu(t)}+\CC_1\eps{\dot\uu(t)}, \qquad \bs\sigma(0)=0.
\]

\begin{proposition}[Laplace domain properties of Zener's model]
\label{prop:Zener}
Let $\CC(s)$ be a viscoelastic Zener model.
\begin{itemize}
\item[\rm (a)] If $c_0>0$ is the lower bound for the tensor $\CC_0$, then Hypothesis \ref{hypo2} is satisfied with  $\psi(x):=c_0 x$. 
\item[\rm (b)] Hypothesis \ref{hypo3} is satisfied with $r=0$ and
\begin{equation}\label{eq:PHI}
\phi(x):=
\frac{(1+\|a\|_{L^\infty(\Omega)})}{a_0^2}
(\|\CC_0\|_{\max}+\|\CC_1\|_{\max}) \frac1{\min\{1,x\}^2},
\end{equation}
where $a_0=1/\|a^{-1}\|_{L^\infty(\Omega)}$ is a lower bound for $a$. 
\end{itemize}
\end{proposition}

\begin{proof}
To prove (a), note first that
\[
\overline s (1+as)^{-1} (\CC_0+s\CC_1) 
=\overline s \CC_0+(1+as)^{-1} |s|^2 \CC_\diff,
\]
and therefore
\[
\overline s (1+as)^{-1} (\CC_0+s\CC_1)\MM :\overline\MM
=
(\CC_0\MM:\overline\MM) \, \overline s +
(|s|^2 \CC_\diff\MM:\overline\MM)\,(1+as)^{-1},
\]
where all the bracketed quantities in the right hand side are real. Taking real parts and noticing that
\[
\Re (1+as)^{-1}=\frac{1}{|1+as|^2} (1+a\Re s)\ge 0,
\]
the result follows.

To prove (b), we start with the explicit form of the coefficients
\[
C_{ijkl}(s)=(1+as)^{-1} (C_{ijkl}^0+ sC_{ijkl}^1).
\]
Let then $g_0=C_{ijkl}^0$ and $g_1=C_{ijkl}^1$. An easy computation shows that
\[
\|(1+as)^{-1}\|_{L^\infty(\Omega)} \le \frac{1+|s|\, \|a\|_{L^\infty(\Omega)}}{a_0^2|s|^2},
\]
where $a_0=1/\|a^{-1}\|_{L^\infty(\Omega)}$ so that $a\ge a_0$ almost everywhere. Using
\begin{equation}\label{eq:D.0}
\min\{1,\Re s\}\max\{1,|s|\}\le |s| \qquad \forall s\in \C_+,
\end{equation}
we easily estimate
\begin{alignat*}{6}
\| (1+as)^{-1}(g_0+s g_1)\|_{L^\infty(\Omega)}
	& \le \frac{1+|s|\, \|a\|_{L^\infty(\Omega)}}{a_0^2|s|^2} 
	(\| g_0\|_{L^\infty(\Omega)}+|s|\, \|g_1\|_{L^\infty(\Omega)}) \\
	& \le \frac{1+ \|a\|_{L^\infty(\Omega)}}{a_0^2} 
	(\| g_0\|_{L^\infty(\Omega)}+ \|g_1\|_{L^\infty(\Omega)}) 
	\frac{\max\{1,|s|\}^2}{|s|^2} \\
	& \le \frac{1+ \|a\|_{L^\infty(\Omega)}}{a_0^2} 
	(\| g_0\|_{L^\infty(\Omega)}+ \|g_1\|_{L^\infty(\Omega)}) 
	 \frac1{\min\{1,\Re s\}^2},
\end{alignat*}
which proves the result.
\end{proof}

The above exposition of Zener's model allows for full anisotropy. The {\bf isotropic viscoelastic model} can be easily described with two variable coefficients. To do that we let
$
\lambda,\mu :\C_+\to L^\infty(\Omega)
$
be holomorphic functions with the following properties being satisfied almost everywhere in $\Omega$ and for all $s\in \C_+$:
\begin{subequations}\label{eq:C.7}
\begin{alignat}{6}
\lambda(\overline s)=\overline{\lambda(s)}, 
	& \qquad && 
\mu(\overline s)=\overline{\mu(s)}, \\
\label{eq:C.7b}
\Re(\overline s\, \lambda(s))\ge 0,
	& \qquad &&
\Re(\overline s\,\mu(s))\ge \mu_0 \Re s \qquad (\mu_0>0).
\end{alignat}
\end{subequations}
We then define
\[
\CC(s)\MM:=2\mu(s) (\tfrac12(\MM+\MM^\top))+\lambda(s) \,(\mathrm{tr}\,\MM)\, \II,
\]
where $\II$ is the $d\times d$ identity matrix, so that the material law is
\[
\bs\sigma=2\mu(s) \eps\uu+\lambda(s)(\nabla\cdot\uu)\,\II.
\]
Examples of functions satisfying \eqref{eq:C.7} can be found using a variant of Zener's model for viscoelasticity: let $a,m_\mu,b_\mu,m_\lambda,b_\lambda \in L^\infty(\Omega)$ be strictly positive (bounded below by a positive number) and such that
\[
a\,m_\mu\le  b_\mu, \qquad a\, m_\lambda\le b_\lambda.
\]
Then
\[
\lambda(s):=\frac{m_\lambda+b_\lambda s}{1+a\,s}
	\qquad \text{and} \qquad
\mu(s):=\frac{m_\mu+b_\mu s}{1+a\,s}
\]
satisfy \eqref{eq:C.7}.  To prove the lower bound \eqref{eq:C.7b}, note that
	\begin{align*}
		\Re(\overline s\mu(s)) & = \Re \left( \frac{\overline sm_\mu(1 + as) + \overline s s(b_\mu - am_\mu)}{1+as}  \right)  \\
		& = m_\mu \Re s+ |s|^2 (b_\mu-a m_\mu) \frac{1 + a \Re s}{| 1+ a s|^2}
\ge m_\mu \Re s.
	\end{align*}

\subsection{Fractional Zener models}

In this section we explore models of the form $\CC(s^\nu)$ where $\CC(s)$ is a Zener model and $\nu\in (0,1)$. For fractional powers in the complex plane we will always take the principal determination of the argument, i.e., the one with a branch cut at the negative real axis. A fractional Zener model has the form
\[
(1+a\,s^\nu)^{-1}(\CC_0+s^\nu\CC_1),
\]
where $a$, $\CC_0$, and $\CC_1$ satisfy the same hypotheses as in Section \ref{sec:Zener}. In the time domain, this corresponds to
\[
\bs\sigma(t)+a\,\partial^\nu\bs\sigma(t)
=\CC_0\eps{\uu(t)}+\CC_1\eps{\partial^\nu\uu(t)},
\]
where
\[
(\partial^\nu f)(t)=
\frac1{\Gamma(1-\nu)}\int_0^t \frac{\dot f(\tau)}{(t-\tau)^{\nu}}
\mathrm d\tau
\]
is a Caputo fractional derivative of order $\nu$. Note that this fractional derivative coincides with the Riemann-Liouville fractional derivative of the same order, if we are assuming homogeneous initial conditions for all variables. This fractional derivative can also be defined as a distributional fractional derivative in the entire real line. 

\begin{proposition}[Fractional Zener models]\label{prop:4.3}
Let $\CC(s)$ be a viscoelastic Zener model and let $\nu\in (0,1)$.
\begin{itemize}
\item[\rm (a)] If $c_0>0$ is the lower bound for the tensor $\CC_0$, then $\CC(s^\nu)$ satisfies Hypothesis \ref{hypo2} with $\psi(x):=c_0 x$.
\item[\rm (b)] The model $\CC(s^\nu)$ satisfies Hypothesis \ref{hypo3} with $r=0$ and $\phi$ given by \eqref{eq:PHI}.
\end{itemize}
\end{proposition}

\begin{proof}
To prove (a), using the same idea as in Proposition \ref{prop:Zener}, we write 
\[
\overline s \CC (s^\nu) \MM:\overline \MM
=
(\CC_0\MM:\overline\MM) \, \overline s +
(|s|^2 \CC_\diff\MM:\overline\MM)\,(1+as^\nu)^{-1} s^{\nu-1}.
\]
We thus only need to show 
\[
	\Re (1+as^\nu)^{-1} s^{\nu-1}  \ge 0 \qquad \forall s \in \C_+.
\]
To see that, first observe
\[
	(1+as^\nu)^{-1} s^{\nu-1}  = \frac{1}{s^{1-\nu} + as},
\]
where $1 > 1 - \nu >0$. Since $s \in \C_+$, we have
\[
	\Re (s^{1-\nu} +as)=\Re s^{1-\nu}+ \Re as \geq a_0 \Re s > 0,
\] 
which proves the result. The proof of (b) is a direct consequence of Proposition \ref{prop:Zener}(b) and Lemma \ref{lemma:nu}.
\end{proof}

\begin{lemma}\label{lemma:nu}
The following inequality holds:
\[
\min\{1,\Re s\}\le \Re s^\nu, \qquad \forall \nu\in (0,1), \quad s\in \C_+.
\]
\end{lemma}    

\begin{proof}
Writing $s=r\,e^{\theta}$ with $r>0$ and $\theta\in(-\pi/2,\pi/2)$, it is clear that an equivalent form of the result is the inequality
\begin{equation}\label{eq:thetanu}
\min\{1,r\,\cos\theta\}\le r^\nu \cos(\nu\theta)
\qquad r>0, \qquad \theta\in(-\pi/2,\pi/2).
\end{equation}
It is also clear that we only need to prove \eqref{eq:thetanu} for $\theta\in [0,\pi/2)$. Fix then $\theta\in [0,\pi/2)$ and consider the function
\[
f_\theta(\nu):=\cos(\nu\theta)-(\cos\theta)^\nu.
\]
We have
\[
f_\theta(0)=0, \quad f_\theta(1)=0,
\quad
f_\theta''(\nu)=
	-\theta^2 \cos(\nu\theta)
	-(\cos\theta)^\nu\log^2(\cos\theta)\le 0,
\]
and therefore (by concavity) $f_\theta(\nu)\ge 0$ for $\nu\in (0,1)$ or equivalently
\[
\cos(\nu\theta)\ge (\cos\theta)^\nu \qquad \nu\in (0,1), \quad \theta\in [0,\pi/2).
\]
Finally, this implies
\[
r^\nu\cos(\nu\theta)\ge (r\cos\theta)^\nu \ge
\begin{cases} 1, & \mbox{if $r\cos\theta\ge 1$},\\
r\cos\theta, & \mbox{if $r\cos\theta<1$},
\end{cases}
\]
which proves \eqref{eq:thetanu} and hence the result. 
\end{proof}

\subsection{Maxwell's model}

Maxwell's model is given by
\[
\CC(s)=(1+as)^{-1}s\CC_1,
\]
where $\CC_1$ is a steady Hookean model (with lower bound $c_1>0$) and $a\in L^\infty(\Omega)$ satisfies $a\ge a_0>0$ almost everywhere, for some constant $a_0$. In the time domain this gives again an implicit strain-to-stress relation
\[
\bs\sigma(t)+a\dot{\bs\sigma}(t)=\CC_1\eps{\dot{\mathbf u}(t)}.
\] 

\begin{proposition}\label{prop:Maxwell}
Maxwell's model satisfies Hypotheses \ref{hypo1}-\ref{hypo3} with $r=0$ and
\[
\psi(x):=\frac{c_1\min\{1,a_0^3\}}{2\| a\|_{L^\infty(\Omega)}^2}
\min\{1,x^3\},
	\quad
\phi(x):=\frac{(1+\|a\|_{L^\infty(\Omega)})}{a_0^2} \|\CC_1\|_{\max}
\frac1{\min\{1,x^2\}}.
\]
\end{proposition}

\begin{proof}
Hypothesis \ref{hypo1} is easy to verify. Hypothesis \ref{hypo3} can be verified using the proof of Proposition \ref{prop:Zener} taking $\CC_0=0$. To prove Hypothesis \ref{hypo2} note that almost everywhere
\begin{alignat*}{6}
\Re (\overline s\CC(s)\MM:\overline\MM)
	&=(\CC_1\MM:\overline\MM) |s|^2 \Re (1+as)^{-1}\\
	&\ge c_1 \|\MM\|^2|s|^2 \Re (1+as)^{-1}
	\qquad\forall \MM\in \matC_\symm, \quad s\in \C_+.
\end{alignat*}
Note now that since 
\begin{equation}\label{eq:4.55}
\frac{x^2}{1+x^2}\ge \frac12\min\{1,x^2\} \quad \forall x>0,
\end{equation}
then
\begin{alignat*}{6}
\frac{|as|^2}{|1+as|^2} 
	&\ge \frac12 \frac{|as|^2}{1+|as|^2} 
		\ge\frac14\min\{1,|as|^2\}\\
	& \ge \frac14 \min\{ 1, (a_0\Re s)^2\}
		\ge \frac14\min\{1,a_0^2\} \min\{1, (\Re s)^2\},
\end{alignat*}
and therefore, almost everywhere and for all $s\in \C_+$
\begin{alignat*}{6}
|s|^2 \Re(1+as)^{-1} 
	& = \frac{1+a\Re s}{a^2} \,\frac{|as|^2}{|1+as|^2}\\
	& \ge \frac{\min\{1,a_0\}(1+\Re s)}{\| a\|_{L^\infty(\Omega)}^2}
	\frac14\min\{1,a_0^2\} \min\{1, (\Re s)^2\},
\end{alignat*}
which proves the result. 
\end{proof}

\begin{proposition}[Fractional Maxwell's model]
If $\nu\in (0,1)$ and $\CC(s)=(1+as)^{-1} s\CC_1$ is a Maxwell model, then $\CC(s^\nu)$ satisfies Hypotheses \ref{hypo1}-\ref{hypo3} with $r=0$ and the functions $\psi$ and $\phi$ of Proposition \ref{prop:Maxwell}. 
\end{proposition}

\begin{proof}
Hypothesis \ref{hypo1} is straightforward and Hypothesis \ref{hypo3} follows from the fact that
\[
\frac1{\min\{1,\Re s^\nu\}^2}\le \frac1{\min\{1,\Re s\}^2}
	\qquad s\in \C_+, \quad \nu\in (0,1),
\]
as follows from Lemma \ref{lemma:nu}. To verify Hypothesis \ref{hypo2} we first estimate
\begin{alignat*}{6}
\Re(\overline s\,\CC(s^\nu)\MM:\overline\MM)
	\ge & c_1 \|\MM\|^2 |s|^2 \Re \frac{s^{\nu-1}}{1+as^\nu}\\
	\ge & \frac{c_1}{\| a\|_{L^\infty(\Omega)}^2} \|\MM\|^2
			\frac{|as|^2}{|s^{1-\nu}+as|^2}\Re(s^{1-\nu}+as),
				\qquad \forall s\in \C_+,
\end{alignat*}
almost everywhere. Using \eqref{eq:4.55} and Lemma \ref{lemma:nu}, we can easily bound
\begin{alignat*}{6}
\frac{|as|^2}{|s^{1-\nu}+as|^2} 
	&\ge \frac12\,\frac{|as|^2}{|s|^{2-2\nu}+|as|^2}
	\ge \frac14 \min\{ 1,\frac{|as|^2}{|s|^{2-2\nu}}\}\\
	& \ge \frac14\min\{1,a_0^2\} \min\{ 1,|s^\nu|\}^2
	\ge \frac14\min\{1,a_0^2\} \min\{1,(\Re s)^2\}. 
\end{alignat*}
At the same time
\[
\Re(s^{1-\nu}+as)
	 \ge \min\{1,\Re s\} + a_0 \Re s \ge 2\min\{1,a_0\}\min\{1,\Re s\},
\]
(we have used Lemma \ref{lemma:nu} again) and the proof is finished. 
\end{proof}

\subsection{Voigt's model}

Voigt's model uses 
\[
\CC(s):=\CC_0+s\CC_1
\]
as a viscoelastic parameter model, where $\CC_0$ is a steady Hookean material model and $\CC_1$ is a non-negative Hookean model.  In those parts of the domain where $\CC_1=0$, Voigt's model reduces to classical linear elasticity. In the time domain, this model gives an explicit differential expression for the strain-to-stress relationship
\[
\bs\sigma(t)=\CC_0\eps{\mathbf u(t)}+\CC_1\eps{\dot{\mathbf u}(t)}.
\]
Note that this can be plugged into the momentum equation yielding
\[
\rho\,\ddot{\mathbf u}(t)=
	\mathrm{div}\,(\CC_0\eps{\mathbf u(t)}+\CC_1\eps{\dot{\mathbf u}(t)})+\mathbf f(t),
\]
which shows that this model is a third order differential equation, although we admit the possibility that the third order terms vanish in some regions. 

\begin{proposition}\label{prop:Voigt}
Voigt's model satisfies Hypotheses \ref{hypo1}-\ref{hypo3} with $r=1$,  and
\[
\psi(x):=c_0 x, \qquad 
\phi(x):=\frac{\| \CC_0\|+\|\CC_1\|}{\min\{1,x\}}, 
\]
where $c_0$ is the lower bound of $\CC_0$.
\end{proposition}

\begin{proof}
It is straightforward proof using the type of inequalities of the proof of Proposition \ref{prop:Zener}.
\end{proof}

\begin{proposition}[Fractional Voigt's model]
If $\nu \in (0,1)$ and $\CC(s)=\CC_0+s\CC_1$ is a Voigt model, then $\CC(s^\nu)$ satisfies Hypotheses \ref{hypo1}-\ref{hypo3} with $r=1$,
\[
\psi(x):=c_0 x, \qquad
\phi(x):=\frac{\| \CC_0\|+\|\CC_1\|}{\min\{1,x^2\}}.
\]
\end{proposition}

\begin{proof}
By Proposition \ref{prop:Voigt}
\[
\| \CC(s^\nu)\| \le |s|^\nu \frac{\| \CC_0\|+\|\CC_1\|}{\min\{1,\Re s^\nu\}}
\le \frac{|s|}{\Re s^{1-\nu}}\frac{\| \CC_0\|+\|\CC_1\|}{\min\{1,\Re s\}}, 
\]
where we have used Lemma \ref{lemma:nu}. Using Lemma \ref{lemma:nu} again we obtain the upper bound for $\|\CC(s)\|$ almost everywhere. For positivity (Hypothesis \ref{hypo2}) note that
\[
\Re (\overline s\CC(s^\nu)\MM:\overline\MM)
	=(\Re s) (\CC_0\MM:\overline\MM)+ |s|^2 (\CC_1\MM:\overline\MM) \Re s^{\nu-1}
	\ge c_0 \Re s \|\MM\|^2,
\]
almost everywhere. 
\end{proof}

\subsection{Coupled models}        
               
\begin{proposition}               
Let $\Omega_1,\ldots,\Omega_J$ be non-overlapping subdomains of $\Omega$ such that $\overline\Omega=\cup_{j=1}^J \overline\Omega_j$. Assume that $\mathrm C_j$ is a viscoelastic model in the domain $\Omega_j$, satisfying the Hypotheses \ref{hypo1}-\ref{hypo3}. Then
\begin{equation}\label{eq:4.5}
\mathrm C(s):=\sum_{j=1}^J \chi_{\Omega_j}\mathrm C_j(s)
\end{equation}
defines a viscoelastic model in the full domain $\Omega$. 
\end{proposition}

\begin{proof}
Hypothesis \ref{hypo1} follows readily. Assume now that there exist non-decreasing $\psi_j:(0,\infty)\to(0,\infty)$ satisfying 
\[
\psi_j(x)\ge c_j x^{\ell_j} \quad \forall x\in (0,1], \qquad \ell_j\ge 0, c_j>0,
\]
and
\[
\Re(\overline s \CC_j(s) \MM:\overline\MM)\ge \psi_j(\Re s) \|\MM\|^2 
	\qquad \mbox{a.e. in $\Omega_j$} \quad \forall \MM \in \matC_\symm, \quad s\in \C_+.
\]
Let now $\psi(x):=\min\{\psi_1(x),\ldots,\psi_J(x)\}$. If we take $c:=\min\{c_1,\ldots,c_J\}$ and $\ell:=\max\{\ell_1,\ldots,\ell_J\}$ it follows that $\psi$ is non-decreasing, 
\[
\psi(x)\ge c\,x^\ell \quad\forall x\in (0,1],
\]
and
\[
\Re(\overline s \CC(s) \MM:\overline\MM)\ge \psi(\Re s) \|\MM\|^2 
	\qquad \mbox{a.e. in $\Omega$} \quad \forall \MM \in \matC_\symm.
\]
For the upper bounds, consider integers $r_j\ge 0$ and non-increasing functions $\phi_j:(0,\infty)\to (0,\infty)$ such that
\[
\phi_j(x)\le d_j x^{-k_j} \quad \forall x\in (0,1],\qquad k_j\ge 0, d_j>0,
\]
and
\[
\| \CC_j(s)\MM\| \le |s|^{r_j}\phi_j(\Re s)\|\MM\| 
\qquad \mbox{a.e. in $\Omega_j$} \quad \forall \MM \in \matC_\symm, \quad s\in \C_+.
\]
Let then $r:=\max\{r_1,\ldots,r_J\}$ and the non-increasing function
\[
\phi(x):=\max\{ x^{r_1-r}\phi_1(x),\ldots,x^{r_J-r}\phi_J(x)\}.
\]
If we take $d:=\max\{ d_1,\ldots,d_J\}$ and $k:=\max\{ k_1,\ldots,k_J\}$, we have that
\[
\phi(x)\le d\,x^{-k} \qquad \forall x\in (0,1],
\]
and 
\[
\|\CC(s)\MM\| \le |s|^r \phi(\Re s)\|\MM\| 
\qquad \mbox{a.e. in $\Omega$} \quad \forall \MM \in \matC_\symm, \quad s\in \C_+,
\]
which finishes the proof.
\end{proof}

This means, in particular, that we can combine all the above models (Zener, Maxwell, and Voigt) in their differential or fractional versions, with different fractional orders in different subdomains subdomains. Because all our definitions are distributional, whenever there is an interface (smooth or not) we are implicitly imposing continuity of the displacement (this is done by assuming $\mathbf u$ too take values in $\mathbf H^1(\Omega)$) and of the normal stress (since $\bs\sigma$ takes values in $\mathbb H (\mathrm{div},\Omega)$). 

{\bf Newton's model} is a fourth choice among classical viscoelastic models, given by the law
\[
\CC(s)=s\CC_1,
\]
where $\CC_1$ is a steady Hookean model. Since the viscoelastic law in the time domain becomes
\[
\bs\sigma(t)=\CC_1\eps{\dot{\mathbf u}(t)},
\]
the model can be simplified to
\[
\rho \dot{\mathbf u}(t)=\mathrm{div}\,\CC_1\eps{\mathbf u(t)}
+\mathbf g(t)
\]
(here $\mathbf g$ is an antiderivative of $\mathbf f$ and takes care of non-vanishing initial conditions if needed). Therefore, Newton's model becomes a parabolic equation for linear elasticity and does not produce waves. Since the interest of this paper is wave models, we will not investigate this simple model any further. 
    
\section{Transfer function analysis}\label{sec:5}

We consider the {\bf energy norm}, tagged in a parameter $c>0$:
\[
\triple{\uu}c^2:=c^2 \| \rho^{1/2}\uu\|_\Omega^2+\|\eps\uu\|_\Omega^2.
\]
Recall that the mass density function is $\rho \in L^\infty(\Om)$ is strictly positive.
By \eqref{eq:D.0}, it follows that
\begin{equation}\label{eq:D.00}
\min\{1,\Re s\} \triple\uu1 \le \triple\uu{|s|} \le \frac{|s|}{\min\{1,\Re s\}} \triple\uu1
\qquad \forall \uu \in \HH^1(\Omega;\C) \quad s\in \C_+.
\end{equation}
Finally, consider the bilinear form
\begin{alignat*}{6}
b(\uu,\ww;s)
	:=& a(\uu,\ww;s)+s^2 (\rho\uu,\ww)_\Omega \\
	 =& (\CC(s)\eps\uu,\eps\ww)_\Omega+s^2 (\rho\uu,\ww)_\Omega.
\end{alignat*}

\begin{proposition}\label{prop:5.1}
If $\psi, \phi$ and $r$ are the functions and integer appearing in Hypotheses \ref{hypo2} and \ref{hypo3} and we define
\[
\psi_\star(x):=\min\{ x,\psi(x)\}, 
	\qquad
\phi_\star(x):=\max\{x^{-r},\phi(x)\},
\]
then for all  $s\in \C_+$
\begin{subequations}
\begin{alignat}{6}
	\label{eq:5.2a}
	\Re b(\uu,\overline{s\uu};s) \ge & \psi_\star(\Re s) \triple\uu{|s|}^2
		& \qquad & \forall \uu\in \HH^1(\Omega;\C), \\
	\label{eq:5.2b}
	|b(\uu,\ww;s)| \le & |s|^r \phi_\star(\Re s)\triple\uu{|s|}\triple\ww{|s|}
		& \qquad & \forall \uu,\ww\in \HH^1(\Omega;\C),
\end{alignat}
\end{subequations}
If $\ell\ge 0$ and $k\ge 0$ are the quantities in \eqref{eq:C.2} and \eqref{eq:phiRe}, then
\begin{equation}\label{eq:5.4}
\inf_{0<x<1} x^{-\max\{1,\ell\}}\psi_\star(x)>0, 
	\qquad
\sup_{0<x<1} x^{\max\{r,k\}}\phi_\star(x)<\infty.
\end{equation}
\end{proposition}

\begin{proof}
We have
\[
\Re b(\uu,\overline{s\uu};s)=
	(\Re s) \|\rho^{1/2} s \mathbf u\|_\Omega^2
	+\Re a(\mathbf u,\overline{s\uu};s)
\]
and \eqref{eq:5.2a} follows from \eqref{eq:3.9}. Similarly, by \eqref{eq:C.4}
\begin{alignat*}{6}
|b(\mathbf u,\mathbf w;s)|
	\le & |s|^r\phi(\Re s) \|\eps{\mathbf u}\|_\Omega\|\eps{\mathbf w}\|_\Omega
		+\|\rho^{1/2} s \mathbf u\|_\Omega\|\rho^{1/2} s \mathbf w\|_\Omega\\
	\le & \max\{ |s|^r \phi(\Re s),1\}\triple\uu{|s|}\triple\ww{|s|}
\end{alignat*}
and \eqref{eq:5.2b} follows from the fact that $1\le |s|/\Re s$ for all $s\in \mathbb C_+$. The asymptotic bounds \eqref{eq:5.4} can be proved easily. 
\end{proof}

\begin{lemma}
There exists $C_{\Omega,\rho}>0$ such that for all $c>0$ and $\bs\alpha\in \HH^{1/2}(\Gamma_D)$, the solution of the variational problem
\begin{subequations}\label{eq:D.1}
\begin{alignat}{6}
	& \widehat\uu \in \HH^1(\Omega), \qquad
	 \gamma_D\widehat\uu=\bs\alpha,\\
	& (\eps{\widehat\uu},\eps\ww)_\Omega+ c^2 (\rho\widehat\uu,\ww)_\Omega=0 \quad \forall \ww\in \HH^1_D(\Omega),
\end{alignat}
\end{subequations}
satisfies
\[
	\triple{\widehat\uu}c\le C_{\Omega,\rho} \max\{1,c\}^{1/2} \| \bs\alpha\|_{1/2,\Gamma_D}.
\]
\end{lemma}

\begin{proof}
Let $E:\HH^{1/2}(\Gamma_D) \to \HH^{1/2}(\Gamma)$ be a bounded extension operator and consider the solution of the elliptic boundary value problem
\begin{subequations}\label{eq:D.2}
\begin{alignat}{6}
	& \widetilde\uu \in \HH^1(\Omega),\qquad
	 \gamma\widetilde\uu =E\bs\alpha,\\
	& -\Delta \widetilde\uu+c^2\widetilde\uu =0 \quad\mbox{in $\Omega$}.
\end{alignat}
\end{subequations}
These are $d$ uncoupled scalar problems for each of the components of $\widetilde\uu$. Using the Bamberger-HaDuong lifting lemma (originally stated in \cite{BaHa1986}, see \cite[Proposition 2.5.1]{Sayas2016} for a rephrasing in the current language),
it follows that
\[
c^2 \| \widetilde\uu\|_\Omega^2 + \|\nabla\widetilde\uu\|_\Omega^2 \le  C_\Omega \max\{1,c\} \| E\bs\alpha\|_{1/2,\Gamma}^2
	\le   C_\Omega' \max\{1,c\} \| \bs\alpha\|_{1/2,\Gamma_D}^2.
\]
However, the solution of \eqref{eq:D.1} minimizes $\triple{\widehat\uu}c$ among all $\widehat\uu$ satisfying $\gamma_D\widehat\uu=\bs\alpha$. Therefore
\[
\triple{\widehat\uu}c^2\le \triple{\widetilde\uu}c^2 \le c^2 \|\rho^{1/2}\widetilde\uu\|_\Omega^2 + \|\nabla\widetilde\uu\|_\Omega^2
\le  \max\{\|\rho\|_{L^\infty(\Omega)},1\} \Big(c^2\|\widetilde\uu\|_\Omega^2 + \|\nabla\widetilde\uu\|_\Omega^2\Big).
\]
This finishes the proof.
\end{proof}

The main theorem of this section studies the operator associated to the Laplace transform of problem \eqref{eq:2.1}. In fact, problem \eqref{eq:D.3} below is the Laplace transform of \eqref{eq:2.1} for data that as functions of time are Dirac masses at time equal to zero.

\begin{theorem} \label{th:5.3}
Let $\mathbf f\in \mathbf L^2(\Omega;\C)$, $\bs\alpha\in \HH^{1/2}(\Gamma_D;\C)$, and $\bs\beta\in \HH^{-1/2}(\Gamma_N;\C)$. For $s\in \C_+$, the solution of
\begin{subequations}\label{eq:D.3}
\begin{alignat}{6}
	& \uu \in \HH^1(\Omega;\C),\qquad  \gamma_D \uu =\bs\alpha,\\
	& b(\uu,\ww;s)=(\mathbf f,\ww)_\Omega+\langle \bs\beta,\gamma\ww\rangle_{\Gamma_N} \quad \forall \ww\in \HH^1_D(\Omega;\C),
\end{alignat}
\end{subequations}
satisfies
\[
\triple\uu{|s|}
	\le  \frac{C}{\psi_\star(\Re s)}\Big(    \|\mathbf f\|_\Omega 
	 + \frac{|s|^{3/2+r}\phi_\star(\Re s)}{\min\{1,\Re s\}^{1/2}} 
	 	\| \bs\alpha\|_{1/2,\Gamma_D}
	 	 + \frac{|s|}{\min\{1,\Re s\}} \|\bs\beta\|_{-1/2,\Gamma_N} \Big),
\]
for a certain constant $C$ depending on $\rho$ and the geometry. 
\end{theorem}

\begin{proof}
The solution of \eqref{eq:D.3} for data $(\mathbf f,\bs\beta,\bs\alpha)$ can be decomposed as the sum of the solutions for $(\mathbf f,\mathbf 0,\mathbf 0)$, $(\mathbf 0,\bs\beta,\mathbf 0)$ and $(\mathbf 0,\mathbf 0,\bs\alpha)$. For the first one, we note that $\uu\in \HH^1_D(\Omega;\C)$ satisfies
\[
b(\uu,\ww;s)=(\mathbf f,\ww)_\Omega \quad\forall \ww\in \HH^1_D(\Omega;\C),
\]
and we can use $\ww=\overline{s\uu}$ as test function. Applying \eqref{eq:5.2a}, it follows that
\begin{alignat*}{6}
\psi_\star(\Re s) \triple\uu{|s|}^2
	\le & \Re b(\uu,\overline{s\uu};s)=\Re (\mathbf f,\overline{s\uu})_\Omega \\
	\le & | (\rho^{-1/2}\mathbf f,\rho^{1/2}\overline{s\uu})_\Omega|
	\le  \|\rho^{-1/2}\mathbf f\|_\Omega \triple\uu{|s|}.
\end{alignat*}
For the second one, we use the same argument to bound
\begin{alignat*}{6}
\psi_\star(\Re s)\triple\uu{|s|}^2
	\le & |\langle\bs\beta,\overline s\gamma\overline\uu\rangle_{\Gamma_N}| 
	\le  |s| \|\bs\beta\|_{-1/2,\Gamma_N} \| \gamma\uu\|_{1/2,\Gamma_N} \\
	\le & C_1 |s| \|\bs\beta\|_{-1/2,\Gamma_N}\|\uu\|_{1,\Omega}
	\le  C_2 |s| \|\bs\beta\|_{-1/2,\Gamma_N}\triple\uu1 \\
	\le & C_2 \frac{|s|}{\min\{1,\Re s\}}\|\bs\beta\|_{-1/2,\Gamma_N} \triple\uu{|s|},
\end{alignat*}
where we have used Korn's inequality and \eqref{eq:D.00}.
For the final one, we write $\uu=\widehat\uu+\uu_0$, where $\widehat\uu=\widehat\uu(\bs\alpha,|s|)$ is the solution of \eqref{eq:D.1} with $c=|s|$ and $\uu_0\in \HH^1_D(\Omega;\C)$. Then
\[
b(\uu_0,\ww;s)=-b(\widehat\uu,\ww;s) \qquad\forall \ww\in \HH^1_D(\Omega;\C).
\]
Taking $\ww=\overline{s\uu_0}$ above, and using \eqref{eq:5.2b}, we can bound
\[
\psi_\star(\Re s) \triple{\uu_0}{|s|}^2 
	\le | b(\widehat\uu,\uu_0;s)| 
	\le |s|^r \phi_\star(\Re s) \triple{\widehat\uu}{|s|}\,|s|\, \triple{\uu_0}{|s|}.
\]
Therefore, 
\begin{alignat*}{6}
\triple\uu{|s|}\le \triple{\widehat\uu}{|s|}+\triple{\uu_0}{|s|}
	\le 
	&	\left(1+ 
	\frac{|s|^{r+1}\phi_\star(\Re s)}{\psi_\star(\Re s)}\right) \triple{\widehat\uu}{|s|} \\
	\le &  C_{\Omega,\rho}
	\left(1+ \frac{|s|^{r+1}\phi_\star(\Re s)}{\psi_\star(\Re s)}\right) \max\{1,|s|\}^{1/2}
	 \|\bs\alpha\|_{1/2,\Gamma_D}.
\end{alignat*}
Using \eqref{eq:D.0} and 
\begin{alignat*}{6}
\left(1+ \frac{|s|^{r+1}\phi_\star(\Re s)}{\psi_\star(\Re s)}\right)
\le & \frac{|s|}{\psi_\star(\Re s)} (1+|s|^r\phi_\star(\Re s)) \\
\le & \frac{2|s|^{r+1}}{\psi_\star(\Re s)}\max\{ (\Re s)^{-r},\phi_\star(\Re s)\}
	=\frac{2|s|^{r+1}}{\psi_\star(\Re s)}\phi_\star(\Re s),
\end{alignat*}
the result follows.
\end{proof}

\section{Distributional propagation of viscoelastic waves}\label{sec:DIST}

In this section we show how the transfer function studied in Section \ref{sec:5} (specifically in Theorem \ref{th:5.3}) is the Laplace domain transform of the solution operator for a distributional version of the viscoelastic wave propagation problem \eqref{eq:6.3} (cf.\ Proposition \ref{prop:6.4} below). We start with some language about vector-valued distributions, borrowed from \cite{Sayas2016}. 

\subsection{Background on operator valued distributions}

Given a real Hilbert space $X$, its {\bf complexification} $X_\C:=X+\imath X$ is a complex Hilbert space that is isometric to $X\times X$
\[
\|x_1+\imath x_2\|_{X_\C}^2:=\|x_1\|^2_X+\|x_2\|^2_X, \qquad \forall x_1,x_2\in X,
\]
with the product by complex scalars defined in the natural way. The complexification $X_\C$ has a naturally defined conjugation, which is a conjugate linear isometric involution in $X_\C$. The Lebesgue and Sobolev spaces of complex-valued functions that we have used in Section \ref{sec:5} are complexifications of the corresponding real spaces.
If $X$ and $Y$ are real Hilbert spaces, the space of bounded linear operators $\mathcal B(X_\C,Y_\C)$ can be understood as the subspace of $\mathcal B(X^2,Y^2)$ formed by matrices of operators of the form
\[
\begin{pmatrix} A & - B \\ B & A \end{pmatrix} \qquad A,B\in \mathcal B(X,Y).
\]
This is easily seen to be isomorphic (with an equivalent but not equal norm) to the complexification of the Banach space $\mathcal B(X,Y)$. (For the problem of the many possible equivalent complexifications of real Banach spaces, see \cite{MuSaTo1999}.) If $A\in \mathcal B(X_\C,Y_\C)$ we can define $\overline{A}\in \mathcal B(X_\C,Y_\C)$ by
\[
\overline A x:=\overline{A\overline x},
\]
where in the right hand side we use the natural conjugations of $X_\C$ and $Y_\C$. With this definition $\overline{A x}=\overline A\overline x.$

An $X$-valued {\bf tempered distribution} is a continuous linear map from the Schwartz class $\mathcal S(\R)$ to $X$. We say that the $X$-valued tempered distribution $h$ is causal, when the action of $h$ on any element of $\mathcal S(\R)$ supported in $(-\infty,0)$ is zero.
Causal tempered distributions have a {\bf well defined Laplace transform}, which is a holomorphic function $H=\mathcal L\{h\}:\C_+\to X_\C$ satisfying
\begin{equation}\label{eq:6.0}
\overline{H(s)}=H(\overline s) \qquad \forall s\in \C_+,
\end{equation}
where the conjugation on the left-hand side is the one in $X_\C$.
We will write $h\in \mathrm{TD}(X)$ whenever $h$ is an $X$-valued causal tempered distribution whose Laplace transform satisfies
\begin{equation}\label{eq:6.00}
\| H(s)\|_{X_\C} \le |s|^\mu \psi(\Re s) \qquad \forall s\in \C_+,
\end{equation}
where $\mu \in \R$ and $\psi:(0,\infty)\to (0,\infty)$ is non-increasing and at worst rational at the origin, i.e., $\sup_{0<x<1} x^k \psi(x)<\infty$ for some $k\ge 0$. When, instead of a Hilbert space $X$ and its complexification $X_\C$, we are dealing with bounded linear operators $\mathcal B(X_\C,Y_\C)$, the conjugation in \eqref{eq:6.0}  is the one for operators between complexified spaces.  
Some pertinent observations and results:
\begin{itemize}
\item[(a)] If $h\in \mathrm{TD}(X)$ and $T\in \mathcal B(X,Y)$ is a `steady-state' operator, then $Th\in \mathrm{TD}(Y)$. We can reverse the roles of $T$ and $h$ and show that if an operator valued distribution $T\in \mathrm{TD}(\mathcal B(X,Y))$ acts on a constant $h\in X$, it defines $Th\in \mathrm{TD}(Y)$.
\item[(b)] A simple argument using the formula for the inverse Laplace transform of $s^{-m}H(s)$, where $m$ is an integer chosen so that $\mu-m<-1$, can be used to characterize all these distributions (see \cite[Proposition 3.1.2]{Sayas2016}): $h\in \mathrm{TD}(X)$ if and only if there exists an integer $m\ge 0$ and a causal continuous function $g:\R\to X$ with polynomial growth at infinity such that $h=g^{(m)}$, with differentiation understood in the sense of tempered distributions. Moreover, if $H:\C_+\to X_\C$ is a holomorphic function satisfying \eqref{eq:6.0} and \eqref{eq:6.00} (with the conditions given for $\psi$), then $H=\mathcal L\{h\}$ for some $h\in \mathrm{TD}(X)$. 
\item[(c)] If $h\in \mathrm{TD}(\R)$ and $a\in X$, then the tensor product $a\otimes h$  defines a distribution in $\mathrm{TD}(X)$.
\end{itemize}

If $X$ and $Y$ are Banach spaces and $h\in \mathrm{TD}(\mathcal B(X,Y))$, then the {\bf convolution product} $h*\lambda$ is well defined for any $X$-valued causal distribution $\lambda$, independently on whether it is tempered or not \cite{Treves1967}. In the simpler case where $\lambda\in \mathrm{TD}(X)$, we have the convolution theorem
\[
\mathcal L\{ h*\lambda\}(s)=\mathcal L\{h\}(s) \mathcal L\{\lambda\}(s) \qquad \forall s\in \C_+,
\]
which can be used as an equivalent (and simple) definition of the convolution of $h*\lambda$, and we also have $h*\lambda\in \mathrm{TD}(Y)$, as can be easily proved from the definition. 

\subsection{Viscoelastic material law and wave propagator}

\begin{theorem}
If $\mathrm C:\C_+\to L^\infty(\Omega;\C^{d\times d\times d\times d})$ satisfies hypotheses \eqref{eq:C.1}, then there exists
\[
\mathcal C\in \mathrm{TD}(\mathcal B(\mathbb L^2(\Omega)) )
\]
such that
$\mathcal L\{ \mathcal C \,\mathrm M\}(s)=\CC(s)\mathrm M$ for all $\mathrm M\in \matR_\symm$.
For arbitrary $\uu \in \mathrm{TD}(\HH^1(\Omega))$, the convolution
\[
\mathcal C* \eps\uu \in \mathrm{TD}(\mathbb L^2(\Omega))
\]
is well defined. 
\end{theorem}

\begin{proof}
Let us first recall that for $s\in \C_+$, we have defined the operator
\[
L^2(\Omega;\matC_\symm)\ni \mathrm U \longmapsto
\mathrm C(s)\mathrm U \in L^2(\Omega;\matC_\symm)
\]
and that we have (see \eqref{eq:C20}) 
\[
\| \CC(s)\mathrm U\|_\Omega \le \|\CC(s)\|_{\max} \|\mathrm U\|_\Omega
\qquad \forall \mathrm U\in L^2(\Omega;\matC_\symm).
\]
Also (by \eqref{eq:C.1a})
\[
\overline{\CC(s)}\mathrm U =\overline{\CC(s)\overline{\mathrm U}}
=\CC(\overline s)\mathrm U.
\]
This means that
\[
\CC:\C_+\to \mathcal B(L^2(\Omega;\matC_\symm))
\]
satisfies the conditions \eqref{eq:6.0} and \eqref{eq:6.00} and therefore there exists
\[
\mathcal C\in\mathrm{TD}(\mathcal B(\mathbb L^2(\Omega)))
\] 
such that $\mathcal L\{\mathcal C\}=\CC$. If we fix $\mathrm M\in \matR_\symm$, we can easily show that the Laplace transform of $\mathcal C\mathrm M\in \mathrm{TD}(\mathbb L^2(\Omega))$ is $\mathrm C(s)\mathrm M$. Finally, if $\uu \in \mathrm{TD}(\HH^1(\Omega))$, then $\eps\uu\in \mathrm{TD}(\mathbb L^2(\Omega))$ and the convolution $\mathcal C*\eps\uu$ is well defined. 
\end{proof}

The expression $\bs\sigma=\mathcal C*\eps\uu$ can be equivalently written $\bs\sigma=\mathcal D*\eps{\dot\uu}$ where $\mathcal D\in \mathrm{TD}(\mathcal B(\mathbb L^2(\Omega)))$ is the distribution whose Laplace transform is $s^{-1}\mathrm C(s)$. In the simplest example (the purely elastic case), $\mathrm C(s)=\CC_0$, we can write $\mathcal C=\mathrm C_0\otimes \delta_0$ and $\mathcal D=\mathrm C_0\otimes H$, where $H$ is the Heaviside function. This yields the usual elastic law $\bs\sigma=\CC_0\eps\uu$. 

\begin{proposition}\label{prop:6.2}
For $s\in \C_+$, let us consider the solution map
\[
\mathrm S(s):\mathbf L^2(\Omega;\C)\times \HH^{1/2}(\Gamma_D;\C)\times \HH^{-1/2}(\Gamma_N;\C)
\to \HH^1(\Omega;\C)
\]
defined by $\mathbf u=\mathrm S(s)(\mathbf f,\bs\alpha,\bs\beta)$ being the solution of \eqref{eq:D.3}. The function
\[
\mathrm S:\C_+\to 
\mathcal B(\mathbf L^2(\Omega;\C)\times \HH^{1/2}(\Gamma_D;\C)\times \HH^{-1/2}(\Gamma_N;\C);
\HH^1(\Omega;\C))
\]
is analytic and $\mathrm S(\overline s)=\overline{\mathrm S(s)}$ for all $s\in \C_+$.
\end{proposition}

\begin{proof}
For $s\in \C_+$ and $\uu\in \HH^1(\Omega;\C)$, we define
\[
\mathbb A(s)\uu:=b(\uu,\,\cdot\,;s)
=s^2(\rho\,\uu,\,\cdot\,)_\Omega+(\CC(s)\eps\uu,\eps{\,\cdot\,})_\Omega\in \HH^1_D(\Omega;\C)'.
\]
Using the analyticity of $\CC$ it is easy to see that
\[
\mathbb A:\C_+\to \mathcal B(\HH^1(\Omega;\C),\HH^1_D(\Omega;\C)')
\]
is analytic. Moreover, the operator $\mathbb A_0(s):=\mathbb A(s)|_{\HH^1_D(\Omega;\C)}$ is invertible for all $s\in \C_+$ by the coercivity of the bilinear form $b$ given in \eqref{eq:5.2a}. For any two Banach spaces, the inversion operator
\[
\{ T\in \mathcal B(X,Y)\,:\, \mbox{$T$ is bijective} \} \to \mathcal B(Y,X)
\]
is $\mathcal C^\infty$ and therefore the map $s\to \mathbb A_0(s)^{-1}$ is analytic from $\C_+$ to $\mathcal B(\HH^1_D(\Omega;\C)',\HH^1_D(\Omega;\C))$. If we now consider a bounded operator $L:\HH^{1/2}(\Gamma_D)\to \HH^1(\Omega)$ that is a right-inverse of $\gamma_D$ and its natural extension to complex-valued functions, we can easily write
\[
\mathrm S(s)(\mathbf f,\bs\alpha,\bs\beta)
=L\bs\alpha+
\mathbb A_0(s)^{-1}\left( (\mathbf f,\,\cdot\,)_\Omega+\langle\bs\beta,\gamma\,\cdot\,\rangle_{\Gamma_N}
-\mathbb A(s)L\bs\alpha\right),
\]
which shows that $\mathrm S$ is analytic. Finally \eqref{eq:C.1a} and a simple computation show that $\mathbb A(\overline s)=\overline{\mathbb A(s)}$ and therefore $\mathrm S(\overline s)=\overline{\mathrm S(s)}$.
\end{proof}

\begin{proposition}\label{prop:6.3}
For $s\in \C_+$, let us consider the solution map
\[
\mathrm T(s):\mathbf L^2(\Omega;\C)\times \HH^{1/2}(\Gamma_D;\C)\times \HH^{-1/2}(\Gamma_N;\C)
\to \HH^1(\Omega;\C)\times \mathbb H(\mathrm{div},\Omega;\C)
\]
defined by $\mathrm T(s)(\mathbf f,\bs\alpha,\bs\beta):=(\uu,\mathrm C(s)\eps\uu)$, where $\uu=\mathrm S(s)(\mathbf f,\bs\alpha,\bs\beta)$. The function
\[
\mathrm T:\C_+\to 
\mathcal B(\mathbf L^2(\Omega;\C)\times \HH^{1/2}(\Gamma_D;\C)\times \HH^{-1/2}(\Gamma_N;\C); \,
\HH^1(\Omega;\C)\times \mathbb H(\mathrm{div},\Omega;\C))
\]
is analytic and $\mathrm T(\overline s)=\overline{\mathrm T(s)}$ for all $s\in \C_+$. Finally
\[
\mathrm T(s)^{-1}(\uu,\bs\sigma)=(\rho\,s^2\,\uu-\mathrm{div}\bs\sigma,\gamma_D\uu,\gamma_N\bs\sigma)
\qquad \forall s\in \C_+.
\]
\end{proposition}

\begin{proof}
If $\uu=\mathrm S(s)(\mathbf f,\bs\alpha,\bs\beta)$ and $\bs\sigma=\mathrm C(s)\eps\uu$, then
\begin{subequations}\label{eq:6.66}
\begin{alignat}{6}
& \uu\in \HH^1(\Omega;\C), \quad \bs\sigma\in \mathbb H(\mathrm{div},\Omega;\C),\\
& \rho\,s^2\uu=\mathrm{div}\,\bs\sigma +\mathbf f,\\
& \bs \sigma =\mathrm C(s)\eps{\uu},\\
& \gamma_D \uu=\bs\alpha, \quad \gamma_N\bs\sigma=\bs\beta.
\end{alignat}
\end{subequations}
Since $\mathrm{div}\,\bs\sigma=\rho\,s^2 \mathrm S(s)(\mathbf f,\bs\alpha,\bs\beta)-\mathbf f$ and $\mathrm C$ is analytic, it is clear that $\mathrm T$ is analytic. The conjugation property for $\mathrm T$ and the formula for its inverse are straightforward.
\end{proof}

\begin{proposition}\label{prop:6.4}
There exists a distribution
\[
\mathcal T\in \mathrm{TD}(\mathcal B(\mathbf L^2(\Omega)\times\HH^{1/2}(\Gamma_D)\times \HH^{-1/2}(\Gamma_N),
\HH^1(\Omega)\times \mathbb H(\mathrm{div},\Omega)))
\]
such that for all
\begin{equation}\label{eq:6.data}
\mathbf f\in \mathrm{TD}(\mathbf L^2(\Omega)),
	\qquad
\bs\alpha\in \mathrm{TD}(\HH^{1/2}(\Gamma_D)),
	\qquad
\bs\beta\in \mathrm{TD}(\HH^{-1/2}(\Gamma_N))
\end{equation}
the pair $(\uu,\bs\sigma)=\mathcal T*(\mathbf f,\bs\alpha,\bs\beta)$ is the unique solution to
\begin{subequations}\label{eq:6.3}
\begin{alignat}{6}
\label{eq:6.3a}
&\uu \in \mathrm{TD}(\HH^1(\Omega)), 
	\quad \bs\sigma\in \mathrm{TD}(\mathbb H(\mathrm{div},\Omega)),\\
& \rho\,\ddot\uu =\mathrm{div}\,\bs\sigma+\mathbf f,\\
&\bs\sigma=\mathcal C*\eps\uu,\\
&\gamma_D\uu =\bs\alpha, \quad \gamma_N\bs\sigma=\bs\beta.
\end{alignat}
\end{subequations}
\end{proposition}

\begin{proof}
Let $(\bs U,\bs\Sigma)=\mathrm T(s)(\bs F,\bs A,\bs B)$. 
We first estimate
\[
\|\bs\Sigma \|_\Omega+\|\mathrm{div}\,\bs\Sigma \|_\Omega
\le |s|^r\phi(\Re s) \|\eps{\bs U }\|_\Omega+|s|^2 \|\rho\,\bs U \|_\Omega+\|\bs F\|_\Omega.
\]
By Korn's inequality, \eqref{eq:D.00} and Theorem \ref{th:5.3}, we can also bound
\begin{alignat*}{6}
\| \bs U \|_{1,\Omega}
	\le & C \triple{\bs U }1  
	\le  \frac{C}{\min\{1,\Re s\}}\triple{\bs U }{|s|}   \\
	\le & \phi_1(\Re s) \| \bs F \|_\Omega  
	 + |s|^{3/2+r}\phi_2(\Re s) \|\bs A \|_{1/2,\Gamma_D}
	 + |s|\phi_3(\Re s) \|\bs B \|_{-1/2,\Gamma_N},
\end{alignat*}
where
\[
\phi_1(x):= \frac{C_1}{\psi_\star(x)\min\{1,x\}},
	\quad
\phi_2(x):=\frac{C_2\phi_\star(x)}{\psi_\star(x)\min\{1,x^{3/2}\}},
	\quad
\phi_3(x):=\frac{C_3}{\psi_\star(x) \min\{1,x^2\}}.
\]
The above estimates give an upper bound for the norm of $\|\mathrm T(s)\|$, which together with Proposition \ref{prop:6.3} shows the existence of $\mathcal T$ such that $\mathrm T=\mathcal L\{\mathcal T\}$. To prove that $(\uu,\bs\sigma)=\mathcal T*(\mathbf f,\bs\alpha,\bs\beta)$ solves \eqref{eq:6.3}, we can just take Laplace transforms and use the definition of $\mathrm T(s)$. We next give an alternative proof that will be used for some arguments later on. Note first that $\bs\sigma=\mathcal C*\eps\uu$, as follows from the definition of $\mathrm T(s)$. Note also that there exists
\[
\mathcal E\in \mathrm{TD}(
\mathcal B(\HH^1(\Omega)\times \mathbb H(\mathrm{div},\Omega),
\mathbf L^2(\Omega)\times\HH^{1/2}(\Gamma_D)\times \HH^{-1/2}(\Gamma_N))
)
\]
such that $\mathcal L\{\mathcal E\}(s)=\mathrm T(s)^{-1}$ for all $s\in \C_+$. In fact,
\[
\mathcal E*(\mathbf u,\bs\sigma)=(\rho\ddot\uu-\mathrm{div}\,\bs\sigma,\gamma_D\uu,\gamma_N\bs\sigma)
\]
for all $(\mathbf u,\bs\sigma)\in \mathrm{TD}(\HH^1(\Omega)\times \mathbb H(\mathrm{div},\Omega))$. Since $\mathrm T(s)^{-1}\mathrm T(s)$ is the identity operator for all $s\in \C_+$, it follows that
\[
\mathcal E*\mathcal T*(\mathbf f,\bs\alpha,\bs\beta)=(\mathbf f,\bs\alpha,\bs\beta)
\]
and therefore $(\uu,\bs\sigma)$ solves \eqref{eq:6.3}. Finally, if $(\uu,\bs\sigma)$ solves \eqref{eq:6.3}, then
$\bs U=\mathcal L\{\uu\}$ and $\bs\Sigma=\mathcal L\{ \bs\sigma\}$ satisfy
\begin{subequations}\label{eq:6.6}
\begin{alignat}{6}
& \bs U(s)\in \HH^1(\Omega;\C), \quad \bs\Sigma(s)\in \mathbb H(\mathrm{div},\Omega;\C),\\
& \rho\,s^2\bs U(s)=\mathrm{div}\,\bs\Sigma(s)+\bs F(s),\\
& \bs \Sigma(s)=\mathrm C(s)\eps{\bs U(s)},\\
& \gamma_D \bs U(s)=\bs A(s), \quad \gamma_N\bs\Sigma(s)=\bs B(s),
\end{alignat}
\end{subequations}
for all $s\in \C_+$, where $\bs F=\mathcal L\{\mathbf f\}$, $\bs A=\mathcal L\{\bs\alpha\}$, and $\bs B=\mathcal L\{\bs\beta\}$.  Since \eqref{eq:6.6} is uniquely solvable, this proves uniqueness of \eqref{eq:6.3}.
\end{proof}

This general result about the weak (distributional) version of the viscoelastic wave propagation problem includes the possibility of adding non-homogeneous initial conditions for the displacement and the velocity, since they are just included in the volume forcing function $\mathbf f$. These non-zero conditions will make the distributional solution of \eqref{eq:6.3} non-smooth at time $t=0$ and therefore the estimates of Section \ref{sec:7} will not be valid.

\section{Estimates in time}\label{sec:7}

In this section we translate the estimates for the transfer function given in Theorem \ref{th:5.3} into time-domain estimates.
The following theorem rephrases \cite[Proposition 3.2.2]{Sayas2016}, which is an inversion theorem for the Laplace transform of causal convolution operators. 

\begin{theorem}\label{thm:7.1}
Let $\mathrm F:\C_+\to \mathcal B(X,Y)$ be a holomorphic function valued in the space of bounded linear operators between two Hilbert spaces and assume that
\begin{equation}\label{eq:7.0}
\| \mathrm F(s)\|_{X\to Y}\le |s|^{m+\mu}\varphi(\Re s) \qquad \forall s\in \C_+,
\end{equation}
where:
\begin{itemize}
\item[(a)] $m\ge 0$ is an integer, and $\mu\in [0,1)$,
\item[(b)] $\varphi:(0,\infty)\to (0,\infty)$ is non-increasing and $\sup_{0< x< 1} x^k \varphi(x)<\infty$ for some $k\ge 0$.
\end{itemize}
If $\lambda\in \mathcal C^{m+1}(\R;X)$ is a causal function such that $\lambda^{(m+2)}$ is integrable, then the unique $Y$-valued causal function $u$ such that $\mathcal L\{u\}=\mathrm F \mathcal L\{ \lambda\}$ is continuous and satisfies
\begin{equation}\label{eq:7.1}
\| u(t)\|_Y \le C_\mu \left(\frac{t}{1+t}\right)^\mu \varphi(t^{-1}) 
\sum_{\ell=m}^{m+2}\int_0^t \|\lambda^{(\ell)}(\tau)\|_X \mathrm d\tau
	\quad\forall t\ge 0.
\end{equation}
\end{theorem}

The integrability of the $(m+2)$-th derivative of $\lambda$ can be relaxed to local integrability, but in that case we need to add a hypothesis concerning the growth of $\|\lambda^{(m+2)}(t)\|_X$, since we need to be able to take Laplace transforms in $\C_+$. In any case, since all the processes that we are dealing with are causal (the solution depends on the past values of the data, and never on the future ones), the result holds in finite intervals of time assuming only local integrability of the last derivative. Note also that if $X$ and $Y$ are complexifications of real spaces, the process in the time domain is real-valued for real-valued data when we have $\overline{\mathrm F(s)}=\mathrm F(\overline s)$ for all $s$. 

We now consider Sobolev spaces
\[
W^{m,1}_+(0,\infty;X)
	:=\{ f\in \mathcal C^{m-1}([0,\infty);X)\,:\, f^{(m)}\in L^1(0,\infty;X),
		\, f^{(\ell)}(0)=0 \quad \ell\le m-1\},
\]
where $X$ is any Hilbert space.

\begin{corollary}\label{cor:7.2}
Let $\mathrm F$ be as in Theorem \ref{thm:7.1} and $\lambda\in W^{m+2,1}_+(0,\infty;X)$. If $\widetilde\lambda:\mathbb R\to X$ is the trivial extension of $\lambda$ to $(-\infty,0)$, then there is a unique $u\in \mathcal C([0,\infty);Y)$ such that $u(0)=0$ and $\mathcal L\{ \widetilde u\}=\mathrm F\mathcal L\{\widetilde\lambda\}$. Moreover, the estimates \eqref{eq:7.1} hold. 
\end{corollary}

The main theorem of this section follows. It will use the antidifferentiation-in-time operator
\[
g^{(-1)}(t):=\int_0^t g(\tau)\mathrm d\tau.
\]
Consider now data satisfying
\begin{alignat*}{6}
\mathbf f &\in W^{m(f),1}_+(0,\infty;\mathbf L^2(\Omega)),\\
\bs\alpha & \in W^{m(\alpha),1}_+(0,\infty;\mathbf H^{1/2}(\Gamma_D)),\\
\bs\beta & \in W^{m(\beta),1}_+(0,\infty;\mathbf H^{-1/2}(\Gamma_N)), 
\end{alignat*} 
for some non-negative $m(\alpha)$, $m(\beta)$, and $m(f)$. We denote with the same symbols their tilde-extensions, i.e., their extensions by zero to negative times. We finally associate the solution $(\mathbf u,\bs\sigma)$ of the viscoelastic wave propagation problem. We want to give hypotheses on the data guaranteeing the existence of solutions of the equation that are continuous functions of time. Key quantities to keep in mind are the parameter $r$ and  the functions $\phi$ and $\psi$ in \eqref{eq:C.2}-\eqref{eq:phiRe} (hypotheses on the material model) and the derived functions $\phi_\star$ and $\psi_\star$ in Proposition \ref{prop:5.1}. The key theorem to keep in mind is Theorem \ref{th:5.3} and the bound for $\|\mathrm C(s)\|$. 

\begin{theorem}\label{th:7.3}
If
\[
m(f)=2,\quad m(\alpha)=3+r, \quad m(\beta)=3, 
\]
then
$
\mathbf u\in 
	\mathcal C^1([0,\infty);\mathbf L^2(\Omega))
	\cap
	\mathcal C([0,\infty);\mathbf H^1(\Omega)),
$
and we have the estimates
\begin{subequations}
\begin{alignat}{6}
\nonumber
\| \dot{\mathbf u}(t)\|_\Omega
	+\|\eps{\mathbf u(t)}\|_\Omega 
		\le & C_f(t) \sum_{k=0}^2 \int_0^t \|\mathbf f^{(k)}(\tau)\|_\Omega\mathrm d\tau
			+ C_\alpha(t) \sum_{k=1+r}^{3+r} 
				\int_0^t \|\bs\alpha^{(k)}(\tau)\|_{1/2,\Gamma_D}\mathrm d\tau
				\\
\label{eq:7.5a}
			& +C_\beta(t) \sum_{k=1}^3 
				\int_0^t \|\bs\beta^{(k)}(\tau)\|_{-1/2,\Gamma_N}\mathrm d\tau, \\
\nonumber
\|\mathbf u(t)\|_\Omega 
		\le & C_f(t) \sum_{k=-1}^1 \int_0^t \|\mathbf f^{(k)}(\tau)\|_\Omega\mathrm d\tau
			+ C_\alpha(t) \sum_{k=r}^{2+r} 
				\int_0^t \|\bs\alpha^{(k)}(\tau)\|_{1/2,\Gamma_D}\mathrm d\tau
				\\
\label{eq:7.5b}
			& +C_\beta(t) \sum_{k=0}^2 
				\int_0^t \|\bs\beta^{(k)}(\tau)\|_{-1/2,\Gamma_N}\mathrm d\tau.
\end{alignat}
with
\begin{alignat*}{6}
C_f(t)& :=\frac{C_1}{\psi_\star(t^{-1})}, \qquad
C_\alpha(t)& := \frac{C_3 \max\{1,t^{1/2}\} \phi_\star(t^{-1})}{\psi_\star(t^{-1})}
				\left(\frac{t}{1+t}\right)^{1/2},\\
C_\beta(t) &:= \frac{C_2\max\{1,t\}}{\psi_\star(t^{-1})},
\end{alignat*}
for some constants $C_1, C_2, C_3$.
If additionally
\[
m(f)=2+r, \qquad m(\alpha)=3+2r, \qquad m(\beta)=3+r, 
\]
then
$
\bs\sigma \in\mathcal C([0,\infty);\Ltwo)
$
and we have the bound
\begin{alignat}{6}
\nonumber
\|\bs\sigma(t)\|_\Omega  
	\le &  
		C_f(t) \sum_{k=r}^{2+r} \int_0^t \|\mathbf f^{(k)}(\tau)\|_\Omega\mathrm d\tau
		+ C_\alpha(t) \sum_{k=1+2r}^{3+2r} 
				\int_0^t \|\bs\alpha^{(k)}(\tau)\|_{1/2,\Gamma_D}\mathrm d\tau
			 \\
\label{eq:7.5c}
	& +C_\beta(t) \sum_{k=1+r}^{3+r} 
				\int_0^t \|\bs\beta^{(k)}(\tau)\|_{-1/2,\Gamma_N}\mathrm d\tau.
\end{alignat}		
\end{subequations}		
\end{theorem}

\begin{proof}
The result is a more or less direct consequence of Corollary \ref{cor:7.2}, using Theorem \ref{th:5.3} for the Laplace domain estimates, and going through the language of Propositions \ref{prop:6.2} and \ref{prop:6.3}. Since the problem is linear, it can be decomposed as the sum of problems with data $(\mathbf f,0,0)$, $(0,\bs\alpha,0)$, and $(0,0,\bs\beta)$. We follow the notation of Proposition \ref{prop:6.2} and define nine instances of spaces and operators to apply Corollary \ref{cor:7.2}: we list the spaces $X$ and $Y$ (before complexification), as well as the value of $m$ and $\mu$ and the function $\varphi$ in the estimate \eqref{eq:7.0}.
We separate the operator $\mathrm S(s)$ in Proposition \ref{prop:6.2} as a sum of three operators
\[
\mathrm S(s)=\mathrm S_f(s)+\mathrm S_\alpha(s)+\mathrm S_\beta(s).
\]
We will also use the embedding of $\mathbf H^1(\Omega)$ into $\mathbf L^2(\Omega)$. The following table lists all nine operators:
\[
\begin{array}{c|c|c|c|c|c|c|c}
\mathrm F(s) & X & Y & \varphi(x) & m & \mu & \lambda & u \\
\hline & & & & & & & \\[-2ex]
s\mathrm I_{H^1\to L^2} \mathrm S_f(s)  
	& \mathbf L^2(\Omega) & \mathbf L^2(\Omega) 
	& 1/\psi_\star(x) & 0 & 0 
	& \mathbf f & \dot{\mathbf u} \\[1.5ex]
\bs\varepsilon \circ \mathrm S_f(s) 
	& \mathbf L^2(\Omega) & \Ltwo 
	& 1/\psi_\star(x) & 0 & 0 
	& \mathbf f & \eps{\mathbf u} \\[1.5ex]
\mathrm C(s)(\bs\varepsilon \circ \mathrm S_f(s)) 
	& \mathbf L^2(\Omega) & \Ltwo 
	& \phi(x)/\psi_\star(x) & r & 0 
	& \mathbf f &\bs\sigma \\[1.5ex]
s\mathrm I_{H^1\to L^2} \mathrm S_\alpha(s)  
	& \mathbf H^{1/2}(\Gamma_D) & \mathbf L^2(\Omega) 
	& \frac{\phi_\star(x)}{\psi_\star(x)\min\{1,\sqrt x\}} & 1+r & \frac12 
	& \bs\alpha & \dot{\mathbf u} \\[1.5ex]
\bs\varepsilon \circ \mathrm S_\alpha(s) 
	& \mathbf H^{1/2}(\Gamma_D) & \Ltwo 
	& \frac{\phi_\star(x)}{\psi_\star(x)\min\{1,\sqrt x\}} & 1+r & \frac12 
	& \bs\alpha & \eps{\mathbf u} \\[1.5ex]
\mathrm C(s)(\bs\varepsilon \circ \mathrm S_\alpha(s)) 
	& \mathbf H^{1/2}(\Gamma_D) & \Ltwo 
	& \frac{\phi(x)\phi_\star(x)}{\psi_\star(x)\min\{1,\sqrt x\}} & 1+2r & \frac12 
	& \bs\alpha &\bs\sigma \\[1.5ex]
s\mathrm I_{H^1\to L^2} \mathrm S_\beta(s)  
	& \mathbf H^{-1/2}(\Gamma_N) & \mathbf L^2(\Omega) 
	& \frac1{\min\{1,x\}\psi_\star(x)} & 1 & 0 
	& \bs\beta & \dot{\mathbf u} \\[1.5ex]
\bs\varepsilon \circ \mathrm S_\beta(s) 
	& \mathbf H^{-1/2}(\Gamma_N) & \Ltwo 
	& \frac1{\min\{1,x\}\psi_\star(x)} & 1 & 0 
	& \bs\beta & \eps{\mathbf u} \\[1.5ex]
\mathrm C(s)(\bs\varepsilon \circ \mathrm S_\beta(s)) 
	& \mathbf H^{-1/2}(\Gamma_N) & \Ltwo 
	& \frac{\phi(x)}{\min\{1,x\}\psi_\star(x)} & 1+r & 0 
	& \bs\beta &\bs\sigma 
\end{array}
\]
This proves the continuity properties for $\mathbf u$ and $\bs\sigma$ as well as the estimates \eqref{eq:7.5a} and \eqref{eq:7.5c}. Note that zero initial values hold whenever the output function is continuous (see Corollary \ref{cor:7.2}). To obtain the bounds for $\|\mathbf u(t)\|_\Omega$ we can use a simple shifting argument, since if $(\mathbf f,\bs\alpha,\bs\beta)\mapsto \dot{\mathbf u}$ (i.e., we use the operators multiplied by $s$ and with values in $\mathbf L^2(\Omega)$), then $(\mathbf f^{(-1)},\bs\alpha^{(-1)},\bs\beta^{(-1)})\mapsto \mathbf u$.
\end{proof}

In general (see Proposition \ref{prop:5.1})
\[
\frac1{\psi_\star(t^{-1})}\lesssim t^{\max\{1,\ell\}},
	\qquad
\phi(t^{-1}) \lesssim t^k,
	\qquad
\phi_\star(t^{-1})\lesssim t^{\max\{r,k\}}
\]
so all bounds in Theorem \ref{th:7.3} are polynomial.
In Zener's model and its fractional version (see Proposition \ref{prop:Zener} and \ref{prop:4.3}) we have $r=0$ and
\begin{alignat*}{6}
\psi(x)& =c_0\,x, \qquad&  \psi_\star(x)&=\min\{1,c_0\} x,\\
\phi(x)&=C\max\{1,x^{-2}\},\qquad  & \phi_\star(x)& \le \max\{1,C\} \max\{1,x^{-2}\}.
\end{alignat*}

\section{Technical work towards a time domain analysis}\label{sec:8}

In Section \ref{sec:SEMI}, we give a different analysis, based on the theory of $C_0$-semigroups of operators, of the viscoelastic wave propagation for Zener's model, including situations where part of the domain is described with a purely elastic material. This requires a certain amount of preparatory work, which we will present in full detail. To avoid keeping track of constants, in this section and in the next we will use the symbol $\lesssim$ to absorb constants in inequalities that we do not want to display.

Let $\mathrm C_\diff\in L^\infty(\Omega;\mathbb R^{(d\times d)\times (d\times d)})$ be a  non-negative  Hookean model, which will play the role of the {\bf diffusive part of the viscoelastic model}. We will identify the tensor with the operator
\[
\mathrm C_\diff:\mathbb L^2(\Omega)\longrightarrow \mathbb L^2(\Omega),
\]
which is bounded, selfadjoint, and positive semidefinite.
We also consider a strictly positive function $a\in L^\infty(\Omega)$, and the associated multiplication operator
\[
\Ltwo\ni \mathrm E \longmapsto m_a \mathrm E:= a \mathrm E\in \Ltwo.
\]
Note that $m_a\Cdiff=\Cdiff m_a$. We next consider the following seminorm in $\Ltwo$
\[
|\mathrm S|_c^2:=(a\,\Cdiff \mathrm S,\mathrm S)_\Omega.
\]

\begin{proposition}\label{prop:8.1}
The following properties hold:
\begin{itemize}
\item[\rm (a)] $|\mathrm S|_c\lesssim \| \mathrm S\|_\Omega$  for all $\mathrm S\in \Ltwo$.
\item[\rm (b)] $\|\Cdiff\mathrm S\|_\Omega\lesssim |\mathrm S|_c$  for all $\mathrm S\in \Ltwo$.
\item[\rm (c)] $|\mathrm S|_c=0$ if and only if $\mathrm S\in \ker \Cdiff$.
\item[\rm (d)] If $b\in L^\infty(\Omega)$, then $|b\mathrm S|_c\lesssim |\mathrm S|_c$ for all $\mathrm S\in \Ltwo$.
\end{itemize}
\end{proposition}
 
\begin{proof}
Since $m_a$ and $\Cdiff$ are bounded linear operators in $\Ltwo$,
the proof of (a) is straightforward
\[
|\mathrm S|_c^2 
	=(a\Cdiff\mathrm S,\mathrm S)_\Omega 
	\le \| m_a\Cdiff\mathrm S\|_\Omega \|\mathrm S\|_\Omega
 	\lesssim \|\mathrm S\|_\Omega^2.
\]
To prove (b) note first that since $m_{1/a}$ is bounded, then
\begin{equation}\label{eq:8.1}
\|\Cdiff\mathrm S\|_\Omega \lesssim \| a\Cdiff\mathrm S\|_\Omega 
\end{equation}
and therefore, using (a)
\begin{alignat*}{6}
\| a\Cdiff\mathrm S\|_\Omega^2
	&=(a\Cdiff\mathrm S,a\Cdiff\mathrm S)_\Omega
	=(\mathrm S,a\Cdiff\mathrm S)_c \\
	&\le |\mathrm S|_c |a\Cdiff\mathrm S|_c 
	\lesssim |\mathrm S|_c \|a\Cdiff\mathrm S\|_\Omega.
\end{alignat*}
The result follows then by \eqref{eq:8.1}. The property (c) is a direct consequence of (b). Finally, to prove (d) we write
\begin{alignat*}{6}
|b\,\mathrm S|_c^2 
	&=(a\Cdiff(b\mathrm S),b \mathrm S)_\Omega 
	 =(b^2 a \Cdiff\mathrm S,\mathrm S)_\Omega \\
	&=\int_\Omega b^2 a \underbrace{(\Cdiff \mathrm S):\mathrm S}_{\ge 0} 
	 \lesssim \int_\Omega  a (\Cdiff \mathrm S):\mathrm S =|\mathrm S|_c^2.
\end{alignat*}
This finishes the proof.
\end{proof}  

The dynamics of the viscoelastic model will allow for subdomains where energy is conserved and subdomains where the diffusive part of the model is active. The fact that we allow for transition areas between purely elastic and strictly diffusive models forces us to take the {\bf completion} of the space $\Ltwo$ with respect to the seminorm $|\,\cdot\,|_c.$ This standard process requires first to eliminate the kernel of the seminorm and move to its orthogonal complement.
We thus consider the subspace
\begin{alignat*}{6}
M &:=\{\mathrm S\in \Ltwo\,:\, (\mathrm S,\mathrm T)_\Omega=0
		\qquad \forall \mathrm T\,,\, |\mathrm T|_c=0\} \\
	&=\{\mathrm T\in \Ltwo\,:\, |\mathrm T|_c=0\}^\perp 
	=(\ker\Cdiff)^\perp,
\end{alignat*}
which is clearly closed in $\Ltwo.$ If $\mathrm T\in M$ and $|\mathrm T|_c=0$, then $\mathrm T\in (\ker\Cdiff)^\perp\cap\ker\Cdiff=\{0\}$ (Proposition \ref{prop:8.1}(c)) and therefore $|\,\cdot\,|_c$ is a norm in $M$. Note also that if $\mathrm S\in \Ltwo$, we can decompose
\[
\mathrm S=\mathrm S_M+\mathrm S_C, 
	\qquad \mathrm S_M\in M,   \qquad \mathrm S_C\in \ker\Cdiff,
\]
and $\Cdiff \mathrm S=\Cdiff\mathrm S_M$. We define the Hilbert space $\widehat M$ as the completion of $M$ with respect to the norm $|\,\cdot\,|_c$. Consider next the canonical injection $\mathrm I:M \to \widehat M$ and the operator $\Pi:\Ltwo\to M$ that performs the $\Ltwo$ orthogonal projection onto $M$, where in $M$ we consider the norm $|\,\cdot\,|_c$. (For the sake of precision, it is important to understand that the target space for $\Pi$ is $M$ and therefore $\Pi$ is surjective.) We then define 
\[
\mathrm R:=\mathrm I\, \Pi:\Ltwo\to \widehat M,
\] 
which is continuous by Proposition \ref{prop:8.1}(a). Note next that $\Cdiff|_M:M \to \Ltwo$ is bounded by Proposition \ref{prop:8.1}(b) and, therefore, there exists a unique bounded extension
\[
\widehat\Cdiff:\widehat M \to \Ltwo.
\] 
An easy argument about extensions and the fact that $M=(\ker\Cdiff)^\bot$ prove that
\begin{equation}\label{eq:propCdiff}
\widehat\Cdiff \mathrm R =\widehat\Cdiff\,\mathrm I\, \Pi 
	=\Cdiff|_M \,\Pi =\Cdiff.
\end{equation}

We now consider $b\in L^\infty(\Omega)$ and the associated {\bf multiplication operator} $m_b:\Ltwo\to\Ltwo$ that we recall was given by $m_b\mathrm E:=b\,\mathrm E.$ Since $m_b\Cdiff=\Cdiff m_b$, it follows that $m_b$ maps $\ker\Cdiff$ into $\ker\Cdiff$. Therefore, if $\mathrm S\in M$, then
\[
(b\mathrm S,\mathrm T)_\Omega
	=(\mathrm S,b\mathrm T)_\Omega=0 \qquad \forall \mathrm T\in \ker\Cdiff,
\]
which shows that $b\mathrm S\in M$. The bounded linear map $m_b|_M:M\to M$ (see Proposition \ref{prop:8.1}(d) can then be extended to a bounded linear selfadjoint map
\[
\widehat{m_b}:\widehat M \to \widehat M.
\]
Note that since $a\in L^\infty(\Omega)$ is strictly positive, $\widehat{m_{1/a}}=\widehat{m_a}^{-1}$.

\begin{proposition}\label{prop:8.2}
The following properties hold:
\begin{itemize}
\item[\rm (a)] $\widehat{m_b}\mathrm R=\mathrm R m_b$.
\item[\rm (b)] $(\widehat{m_{1/a}}\mathrm S,\mathrm S)_c\ge 0$ for all $\mathrm S\in \widehat M$.
\item[\rm (c)] $(\mathrm S,\mathrm R\mathrm E)_c=
(\widehat\Cdiff\widehat{m_a}\mathrm S,\mathrm E)_\Omega$ for all $\mathrm S\in \widehat M$ and $\mathrm E\in \Ltwo$.
\item[\rm (d)] $m_a\widehat\Cdiff=\widehat\Cdiff\widehat{m_a}$.
\end{itemize}
\end{proposition}

\begin{proof}
It is simple to see that $\Pi m_b=m_b|_M\Pi$ and therefore
\[
\widehat{m_b}\mathrm R=
\widehat{m_b}\mathrm I\,\Pi=\mathrm I\, m_b|_M\Pi=\mathrm R m_b,
\]
which proves (a). 
If $\mathrm S\in M$ and note that
\[
(\widehat{m_{1/a}}\mathrm S,\mathrm S)_c
=(a^{-1}\mathrm S,\mathrm S)_c=(\Cdiff\mathrm S,\mathrm S)_\Omega \ge 0,
\]
then (b) follows by density. Finally, if $\mathrm S\in M$, we have
\[
(\mathrm S,\mathrm R\mathrm E)_c
	=(\mathrm S,\Pi\mathrm E)_c
	=(\Cdiff m_a\mathrm S,\Pi\mathrm E)_\Omega
	=(\Cdiff m_a\mathrm S,\mathrm E)_\Omega
	\quad \forall \mathrm E\in \Ltwo,
\]
and (c) follows by density. Property (d) is straightforward by a density argument. 
\end{proof}

We end this section with a very simple example of the above construction, where the diffusive tensor $\Cdiff$ is strictly positive in a subdomain and vanishes in the complement.
Assume that there are two open sets $\Omega_1$ and $\Omega_2$ such that
\[
\Omega_1\cap\Omega_2=\emptyset, \qquad
\overline\Omega=\overline{\Omega_1}\cup\overline{\Omega_2}
\]  
and we have
\begin{subequations}\label{eq:8.2}
\begin{alignat}{6}
		\label{eq:8.2a}
	\Cdiff\mathrm M \,:\, \mathrm M \ge c_0 \chi_{\Omega_1} \mathrm M:\mathrm M
	& \qquad & \mbox{a.e.} \quad \forall \mathrm M\in \matR_\symm, \\
	\Cdiff\mathrm M =0 && \mbox{a.e. in $\Omega_2$}.
		\label{eq:8.2b}
\end{alignat}
\end{subequations}
We can always write
\[
\mathrm T=\chi_{\Omega_1} \mathrm T+\chi_{\Omega_2} \mathrm T
\]
and note that by \eqref{eq:8.2b}
\begin{equation}\label{eq:8.3}
\Cdiff \mathrm T
	=\chi_{\Omega_1}\Cdiff\mathrm T 
	=\Cdiff(\chi_{\Omega_1}\mathrm T).
\end{equation}
If $\Cdiff\mathrm T=0$, then $\chi_{\Omega_1}\mathrm T:\mathrm T=0$ by \eqref{eq:8.2a} and therefore $\mathrm T=0$ almost everywhere in $\Omega_1$. Reciprocally, if $\mathrm T=0$ in $\Omega_1$, then by \eqref{eq:8.3}, $\Cdiff\mathrm T=0$. We have thus proved that
\begin{equation}\label{eq:8.4}
\ker\Cdiff=\{ \mathrm T\in \Ltwo\,:\, \mathrm T=0 \mbox{ in $\Omega_1$}\}
\end{equation}
and therefore
\[
M=\{\mathrm T\in \Ltwo\,:\, \mathrm T=0\mbox{ in $\Omega_2$}\}
	\equiv \mathbb L^2(\Omega_1).
\]
However, now, due to \eqref{eq:8.2b}, we have
\[
|\mathrm T|_c \approx \| \mathrm T\|_{\Omega_1} \qquad 
	\forall \mathrm T\in M,
\]
and therefore $\widehat M=M$. In this case $\mathrm R:\Ltwo\to \widehat M$ is just the restriction to $\Omega_1$ of matrix-valued functions defined on $\Omega$, $\widehat\Cdiff$ is the restriction of the action of $\Cdiff$ to functions defined on $L^2(\Omega_1;\matR_\symm)$, and the same happens to multiplication operators $\widehat{m_b}$. This simple situation arises when the domain is subdivided into two parts: in one part we will deal with a purely elastic medium ($\Cdiff=0$), while in the other part we will handle a viscoelastic medium that is `strictly' diffusive, as expressed in \eqref{eq:8.2a}.

\section{Semigroup analysis of a general Zener model}
\label{sec:SEMI}

In the coming two sections we will use some basic results on $C_0$-semigroups in Hilbert spaces, namely the Lumer-Philips theorem (characterizing the generators of contractive semigroups in Hilbert spaces as maximal dissipative operators) and existence theorems for strong solutions of equations of the form
\[
\dot U(t)=\mathcal A U(t)+F(t),
\]
where $\mathcal A:D(\mathcal A)\subset \mathcal H\to\mathcal H$ is maximal dissipative and $F:[0,\infty)\to \mathcal H$ is a sufficiently smooth function. These results are well known and the reader is referred to any classical book dealing with semigroups (for instance, Pazy's popular monograph \cite{Pazy1983}) or to the chapters on semigroups in many textbooks on functional analysis.

Consider now the space
\[
\mathcal H:=\mathbf L^2(\Omega)\times \Ltwo \times \widehat M,
\]
endowed with the norm
\[
\|(\mathbf u,\mathrm E,\mathrm S)\|_{\mathcal H}^2
	:=(\rho\,\mathbf u,\mathbf u)_\Omega	
		+(\mathrm C_0\mathrm E,\mathrm E)_\Omega 
		+|\mathrm S|_c^2.
\]
We then define the operator
\[
\mathcal A(\mathbf u,\mathrm E,\mathrm S)
	:=(\rho^{-1}\mathrm{div}\,(\mathrm C_0\mathrm E+\widehat\Cdiff\mathrm S),
		\eps{\mathbf u}, 
		\widehat{m_{1/a}} (\mathrm R\eps{\mathbf u}-\mathrm S)),
\]
with domain
\[
D(\mathcal A) :=
	 \mathbf H^1_D(\Omega) \times 
	 \left\{ (\mathrm E,\mathrm S)\in \Ltwo\times \widehat M\,:\,
		\begin{array}{l}
			 \mathrm{div}(\mathrm C_0\mathrm E+\widehat\Cdiff\mathrm S)
					\in \mathbf L^2(\Omega)\\
		 \gamma_N (\mathrm C_0\mathrm E+\widehat\Cdiff\mathrm S)=0
		 \end{array}\right\}.
\]

\begin{proposition}\label{prop:9.1}
The following properties hold:
\begin{itemize}
\item[\rm (a)]  $(\mathcal AU,U)_{\mathcal H}\le 0$ for all $U\in D(\mathcal A)$.
\item[\rm (b)] The operator $D(\mathcal A)\ni U \mapsto U-\mathcal AU\in \mathcal H$ is surjective.
\end{itemize}
Therefore $\mathcal A$ is maximal dissipative and generates a strongly continuous contractive semigroup in $\mathcal H$. 
\end{proposition}

\begin{proof}
If $U=(\mathbf u,\mathrm E,\mathrm S)\in D(\mathcal A)$, then
\begin{alignat*}{6}
(\mathcal A\,U,U)_{\mathcal H}
	=& (\mathrm{div}\,(\mathrm C_0\mathrm E+\widehat\Cdiff\mathrm S),\mathbf u)_\Omega
		+(\mathrm C_0\eps{\mathbf u},\mathrm E)_\Omega
		+(\widehat{m_a}^{-1}(\mathrm R\eps{\mathbf u}-\mathrm S),\mathrm S)_c\\
	=&-(\widehat\Cdiff\mathrm S,\eps{\mathbf u})_\Omega
		+(\widehat{m_a}^{-1}\mathrm R\eps{\mathbf u},\mathrm S)_c
		-(\widehat{m_a}^{-1}\mathrm S,\mathrm S)_c
	=-(\widehat{m_{1/a}}\mathrm S,\mathrm S)_c\le 0,		
\end{alignat*}	
where we have progressively applied the boundary condition $\gamma_N (\mathrm C_0\mathrm E+\widehat\Cdiff\mathrm S)=0$, Proposition \ref{prop:8.2}(b), and Proposition \ref{prop:8.2}(c). This proves the dissipativity of $\mathcal A$.

Take now $(\mathbf f,\mathrm F,\mathrm G)\in \mathcal H$, solve the coercive variational problem
\begin{subequations}\label{eq:9.1}
\begin{alignat}{6}
& \mathbf u\in \mathbf H^1_D(\Omega),\\ 
& (\rho\,\mathbf u,\mathbf v)_\Omega
+((\mathrm C_0+\Cdiff m_{1+a}^{-1})\eps{\mathbf u},\eps{\mathbf v})_\Omega
=&& (\rho\,\mathbf f,\mathbf v)_\Omega -(\mathrm C_0\mathrm F,\eps{\mathbf v})_\Omega\\
	&&& -(\widehat\Cdiff\widehat{m_{1+a}}^{-1}\widehat{m_a}\mathrm G,
			\eps{\mathbf v})_\Omega \quad \forall\mathbf v\in \mathbf H^1_D(\Omega),
\nonumber
\end{alignat}
\end{subequations}
and define (see Proposition \ref{prop:8.2}(a))
\begin{alignat*}{6}
\mathrm E:= 
	&\eps{\mathbf u}+\mathrm F,\\
\mathrm S:= 
	& \widehat{m_{1+a}}^{-1}(\mathrm R\eps{\mathbf u}+\widehat{m_a}\mathrm G)
	=\mathrm R m_{1+a}^{-1} \eps{\mathbf u}
		+\widehat{m_{1+a}}^{-1}\widehat{m_a}\mathrm G,
\end{alignat*}
so that $\mathrm S+\widehat{m_a}\mathrm S=\mathrm R\eps{\mathbf u}+\widehat{m_a}\mathrm G$ and therefore
\[
\mathrm S=\widehat{m_a}^{-1}(\mathrm R\eps{\mathbf u}-\mathrm S)+\mathrm G.
\]
At the same time, by \eqref{eq:propCdiff} and Proposition \ref{prop:8.2}(a), we have
\[
\Cdiff m_{1+a}^{-1}\eps{\mathbf u}
	+\widehat\Cdiff \widehat{m_{1+a}}^{-1}\widehat{m_a}\mathrm G
	=\widehat\Cdiff (\mathrm R m_{1+a}^{-1}\eps{\mathbf u}
		+\widehat{m_{1+a}}^{-1}\widehat{m_a}\mathrm G)
	=\widehat\Cdiff\mathrm S,
\]
so that \eqref{eq:9.1} implies
\[
 (\rho\,\mathbf u,\mathbf v)_\Omega
+(\mathrm C_0\mathrm E+\widehat\Cdiff\mathrm S,\eps{\mathbf v})_\Omega
= (\rho\,\mathbf f,\mathbf v)_\Omega \qquad\forall\mathbf v\in \mathbf H^1_D(\Omega),
\]
which is equivalent to
\begin{alignat*}{6}
\rho\,\mathbf u
	-\mathrm{div}(\mathrm C_0\mathrm E+\widehat\Cdiff\mathrm S) &= \rho\,\mathbf  f,\\
\gamma_N(\mathrm C_0\mathrm E+\widehat\Cdiff\mathrm S) &=0.
\end{alignat*}
Summing up, we have $(\mathbf u,\mathrm E,\mathrm S)\in D(\mathcal A)$ and $(\mathbf u,\mathrm E,\mathrm S)=\mathcal A(\mathbf u,\mathrm E,\mathrm S)+(\mathbf f,\mathrm F,\mathrm G)$.
\end{proof}

\begin{theorem}\label{th:9.2}
For $\bs\alpha\in W^{2,1}_+(0,\infty;\mathbf H^{1/2}(\Gamma_D))$, $\bs\beta\in W^{1,1}_+(0,\infty;\mathbf H^{-1/2}(\Gamma_N))$, and $\mathbf f\in L^1(0,\infty;\mathbf L^2(\Omega))$, there exists a unique
\[
(\mathbf u,\mathrm E,\mathrm S)\in\mathcal C^1([0,\infty); \mathcal H),
\]
such that
\begin{subequations}\label{eq:9.2}
\begin{alignat}{6}
\label{eq:9.2a}
\rho \dot{\mathbf u}(t)
	&=\mathrm{div}(\mathrm C_0\mathrm E(t)+\widehat\Cdiff\mathrm S(t))
		+\mathbf f^{(-1)}(t) & \qquad & t\ge 0,\\
\label{eq:9.2b}
\dot{\mathrm E}(t)
	&=\eps{\mathbf u(t)} &&t\ge0,\\
\label{eq:9.2c}
\mathrm S(t)+\widehat{m_a}\dot{\mathrm S}(t)
	&=\mathrm R\eps{\mathbf u(t)} && t\ge 0,\\
\label{eq:9.2d}
\gamma_D \mathbf u(t)
	&=\bs\alpha(t) &&t\ge 0,\\
\label{eq:9.2e}
\gamma_N(\mathrm C_0\mathrm E(t)+\widehat\Cdiff\mathrm S(t))
	&=\bs\beta^{(-1)}(t) &&t\ge 0,
\end{alignat}
and
\begin{equation}
\mathbf u(0)=0, \qquad \mathrm E(0)=0, \qquad \mathrm S(0)=0.
\end{equation}
\end{subequations}
\end{theorem}

\begin{proof}
For each $t\ge 0$ we solve the elliptic problem 
\begin{subequations}\label{eq:9.3}
\begin{alignat}{6}
 \mathbf u_{\mathrm{nh}}(t)&\in \mathbf H^1(\Omega), \\
\rho\,\mathbf u_{\mathrm{nh}}(t)
	&=\mathrm{div}(\mathrm C_0\eps{\mathbf u_{\mathrm{nh}}(t)},\\ 
\gamma_D\mathbf u_{\mathrm{nh}}(t)&=\bs\alpha(t),\\
\gamma_N\mathrm C_0\eps{\mathbf u_{\mathrm{nh}}(t)}& =\bs\beta^{(-1)}(t),
\end{alignat}
\end{subequations}
and note that
\begin{equation}\label{eq:9.4}
\| \mathbf u_{\mathrm{nh}}^{(\ell)}(t)\|_\Omega
+\|\eps{\mathbf u_{\mathrm{nh}}^{(\ell)}(t)}\|_\Omega
\lesssim 
\|\bs\alpha^{(\ell)}(t)\|_{1/2,\Gamma_D}
+\|\bs\beta^{(\ell-1)}(t)\|_{-1/2,\Gamma_N}
	\qquad\ell=0,1,2.
\end{equation}
In a second step, we define the function $F:[0,\infty)\to \mathcal H$
\[
F(t):=(\mathbf u_{\mathrm{nh}}(t)-\dot{\mathbf u}_{\mathrm{nh}}(t)
		+\rho^{-1}\mathbf f^{(-1)}(t),\quad
		\eps{\mathbf u_{\mathrm{nh}}(t)-\dot{\mathbf u}_{\mathrm{nh}}(t)},\quad
		\widehat{m_a}^{-1}\mathrm R\eps{\mathbf u_{\mathrm{nh}}(t)})
\]
and note that $F\in W^{1,1}_+(0,\infty;\mathcal H)$ and that
\begin{equation}\label{eq:9.5}
\|F^{(\ell)}(t)\|_{\mathcal H}
	\lesssim \sum_{k=\ell}^{\ell+1}\left(
		\|\bs\alpha^{(k)}(t)\|_{1/2,\Gamma_D}
		+\|\bs\beta^{(k-1)}(t)\|_{-1/2,\Gamma_N}\right)
		+\|\mathbf f^{(\ell-1)}(t)\|_\Omega\qquad \ell=0,1.
\end{equation}
We use this function to solve the non-homogeneous initial value problem
\begin{equation}\label{eq:9.6}
\dot U_0(t)=\mathcal AU_0(t)+F(t) \qquad U_0(0)=0,
\end{equation}
which has a unique solution $U_0=(u_0,\mathrm E_0,\mathrm S_0)\in \mathcal C^1([0,\infty);\mathcal H)$, admitting the bounds
\begin{equation}\label{eq:9.7}
\|U_0(t)\|_{\mathcal H}\lesssim \int_0^t \|F(\tau)\|_{\mathcal H}\mathrm d\tau,
	\qquad
\|\dot U_0(t)\|_{\mathcal H}\lesssim \int_0^t \|\dot F(\tau)\|_{\mathcal H}\mathrm d\tau.
\end{equation}
In the next step, we define the triple $(\mathbf u,\mathrm E,\mathrm S)\in \mathcal C^1([0,\infty);\mathcal H)$ by
\begin{subequations}\label{eq:9.8}
\begin{alignat}{6}
\mathbf u(t) &:=\mathbf u_0(t)+\mathbf u_{\mathrm{nh}}(t),\\
\mathrm E(t) &:=\mathrm E_0(t)+\eps{\mathbf u_{\mathrm{nh}}(t)},\\
\mathrm S(t) &:=\mathrm S_0(t),
\end{alignat}
\end{subequations}
Taking into account the definition of $F$, equations \eqref{eq:9.3} and \eqref{eq:9.6} show that $(\mathbf u,\mathrm E,\mathrm S)$ satisfy \eqref{eq:9.2}. 
\end{proof}

\begin{corollary}\label{cor:9.3}
If $\bs\alpha$, $\bs\beta$, and $\mathbf f$ satisfy the conditions of Theorem \ref{th:9.2}, $(\mathbf u,\mathrm E,\mathrm S)$ is the solution of \eqref{eq:9.2}, and we define
\begin{equation}\label{eq:9.9}
\bs\sigma(t):=\mathrm C_0\dot{\mathrm E}(t)+\widehat\Cdiff\dot{\mathrm S}(t),
\end{equation}
the pair $(\mathbf u,\bs\sigma)$ satisfies the bounds for all $t\ge 0$
\begin{alignat*}{6}
\|\mathbf u(t)\|_\Omega 
	\lesssim &  
	\sum_{k=0}^{1}\int_0^t \left(
		\|\bs\alpha^{(k)}(\tau)\|_{1/2,\Gamma_D}
		+\|\bs\beta^{(k-1)}(\tau)\|_{-1/2,\Gamma_N}\right)\mathrm d\tau
		\\
		& +\int_0^t \|\mathbf f^{(-1)}(\tau)\|_\Omega\mathrm d\tau,\\
\|\eps{\mathbf u(t)}\|_\Omega + \|\bs\sigma(t)\|_\Omega
	\lesssim &  
		\sum_{k=1}^{2}\int_0^t \left(
		\|\bs\alpha^{(k)}(\tau)\|_{1/2,\Gamma_D}
		+\|\bs\beta^{(k-1)}(\tau)\|_{-1/2,\Gamma_N}\right)\mathrm d\tau
		\\
		& +\int_0^t\|\mathbf f(\tau)\|_\Omega\mathrm d\tau.
\end{alignat*}
As a consequence $\bs\sigma(0)=0$.
\end{corollary}

\begin{proof}
Use the decomposition \eqref{eq:9.8} and the estimates \eqref{eq:9.4}, \eqref{eq:9.5}, and \eqref{eq:9.7}.
\end{proof}

\begin{corollary}\label{cor:9.4}
If $\bs\alpha$, $\bs\beta$, and $\mathbf f$ satisfy the conditions of Theorem \ref{th:9.2}, $(\mathbf u,\mathrm E,\mathrm S)$ is the solution of \eqref{eq:9.2}, and we define
\[
\bs\sigma(t):=\mathrm C_0\dot{\mathrm E}(t)+\widehat\Cdiff\dot{\mathrm S}(t),
\]
the pair
\begin{equation}\label{eq:9.119}
(\mathbf u,\bs\sigma)\in 
	\left(\mathcal C^1([0,\infty);\mathbf L^2(\Omega))	
		\cap \mathcal C([0,\infty);\mathbf H^1(\Omega))\right)
		\times
		\mathcal C([0,\infty);\Ltwo)
\end{equation}
satisfies the equations
\begin{subequations}\label{eq:9.11}
\begin{alignat}{6}
\label{eq:9.11a}
\rho\,\ddot{\mathbf u}(t) &=\mathrm{div}\,\bs\sigma(t)+\mathbf f(t) & \quad & 
	\mbox{a.e.}-t,\\
\label{eq:9.11b}
\bs\sigma(t)+a\,\dot{\bs\sigma}(t)
	&=\mathrm C_0\eps{\mathbf u(t)}+(a\mathrm C_0+\Cdiff)\eps{\dot{\mathbf u}(t)}
		& & \mbox{a.e.}-t,\\
\label{eq:9.11c}
\gamma_D\mathbf u(t) &=\bs\alpha(t) && t\ge 0,\\
\label{eq:9.11d}
\gamma_N\bs\sigma(t) &=\bs\beta(t) && \mbox{a.e.}-t,
\end{alignat}
with initial conditions
\begin{equation}
\mathbf u(0)=0, \quad \dot{\mathbf u}(0)=0, \quad \bs\sigma(0)=0.
\end{equation}
\end{subequations}
\end{corollary}

\begin{proof}
The key issue for this proof is regularity. Assuming the data regularity of Theorem \ref{th:9.2}, we have \eqref{eq:9.119}, as continuity of $\mathbf u$ as a function $[0,\infty)\to \mathbf H^1(\Omega)$ follows from \eqref{eq:9.2b}. Note that \eqref{eq:9.11c} is \eqref{eq:9.2d}. By \eqref{eq:9.2a}, we have that $\mathrm{div}(\mathrm C_0\mathrm E+\widehat\Cdiff\mathrm S)\in \mathcal C([0,\infty);\mathbf L^2(\Omega))$ and $\dot{\mathbf u}(0)=0$, which was the missing initial condition (recall Corollary \ref{cor:9.3}). We also have
\begin{equation}\label{eq:9.12}
\rho\dot{\mathbf u}(t)=\mathrm{div} \,\bs\sigma^{(-1)}(t) + \mathbf f^{(-1)}(t),
	\qquad
\gamma_N\bs\sigma^{(-1)}(t)=\bs\beta^{(-1)}(t),
\end{equation}
which are integrated versions of \eqref{eq:9.11a} and \eqref{eq:9.11d}.
Note next that by Proposition \ref{prop:8.2}(d) and \eqref{eq:propCdiff}, we have
\[
\widehat\Cdiff\mathrm S(t)+a\widehat\Cdiff\dot{\mathrm S}(t)
	=\widehat\Cdiff(\mathrm S(t)+\widehat{m_a}\dot{\mathrm S}(t))
	=\widehat\Cdiff\mathrm R\eps{\mathbf u(t)}
	=\Cdiff\eps{\mathbf u(t)}.
\]  
We also have 
\[
\mathrm C_0{\mathrm E}(t)+a\mathrm C_0\dot{\mathrm E}(t)
	=\mathrm C_0\eps{\mathbf u}^{(-1)}(t)+a\mathrm C_0\eps{{\mathbf u}}(t),
\]
and therefore
\begin{equation}\label{eq:9.13}
\bs\sigma^{(-1)}(t)+a\bs\sigma(t)=
\mathrm C_0\eps{\mathbf u}^{(-1)}(t)+(a\mathrm C_0+\Cdiff)\eps{\mathbf u(t)}.
\end{equation}
Equations \eqref{eq:9.12} and \eqref{eq:9.13} identify continuous functions of $t$ taking values in $\mathbf L^2(\Omega)$, $\mathbf H^{-1/2}(\Gamma_N)$, and $\Ltwo$ respectively. We can then differentiate them in the sense of vector-valued distributions of $t$ to obtain \eqref{eq:9.11a}, \eqref{eq:9.11d}, and \eqref{eq:9.11b}. Note that to be entirely precise, the additional regularity we obtain is
\begin{alignat*}{6}
\rho\ddot{\mathbf u}-\mathrm{div}\,\bs\sigma & \in L^1(0,\infty;\mathbf L^2(\Omega)),\\
\gamma_N\bs\sigma &\in L^1(0,\infty;\mathbf H^{-1/2}(\Gamma_N)),\\
a\dot{\bs\sigma}-(a\mathrm C_0+\Cdiff)\eps{\dot{\mathbf u}}
	&\in L^1(0,\infty;\Ltwo).
\end{alignat*}
This finishes the proof.
\end{proof}

\begin{corollary}
Let $\bs\alpha\in W^{3,1}_+(0,\infty;\mathbf H^{1/2}(\Gamma_D))$, $\bs\beta\in W^{2,1}_+(0,\infty;\mathbf H^{-1/2}(\Gamma_N))$, and $\mathbf f\in W^{1,1}(0,\infty;\mathbf L^2(\Omega))$, let $(\mathbf u,\mathrm E,\mathrm S)$ solve \eqref{eq:9.2} and $\bs\sigma$ be defined by \eqref{eq:9.9}. We have
\begin{alignat*}{6}
\mathbf u &\in 
	\mathcal C^2([0,\infty);\mathbf L^2(\Omega))	
		\cap \mathcal C^1([0,\infty);\mathbf H^1(\Omega)),\\
\bs\sigma & \in 
	\mathcal C^1([0,\infty);\Ltwo)
		\cap \mathcal C([0,\infty); \mathbb H(\mathrm{div},\Omega))
\end{alignat*}
and equations \eqref{eq:9.11} hold for all $t$ with all derivatives defined in the strong way in the appropriate spaces.
\end{corollary}

\begin{proof}
If we solve problem \eqref{eq:9.2} with $(\dot{\mathbf f},\dot{\bs\alpha},\dot{\bs\beta})$ as data, and we integrate from $0$ to $t$, we obtain the solution of \eqref{eq:9.2} which is therefore an element of $\mathcal C^2([0,\infty);\mathcal H)$. This is enough to prove everything else. 
\end{proof}

The estimates of Corollary \ref{cor:9.3} greatly improve those of Section \ref{sec:7} (see Theorem \ref{th:7.3}) in two aspects: less regularity required for the data, and constants independent of the time variable.
There are three particular cases included in the analysis of this section that are worth paying special attention to. 
\begin{itemize}
\item[(a)] If $\ker\Cdiff=\{0\}$, then $M=\Ltwo$ and $\mathrm R=\mathrm I$ is just the canonical inclusion of $\Ltwo$ into its completion with respect to the {\em norm} $(a\Cdiff\,\cdot\,,\,\cdot\,)_\Omega^{1/2}$.
\item[(b)] When there exists $c_{\mathrm{diff}}>0$ such that
\begin{equation}\label{eq:Cdiffplus}
(\Cdiff\,\mathrm M):\mathrm M\ge c_{\mathrm{diff}} \|\mathrm M\|^2 
\qquad\forall\mathrm M\in \matR_\symm,
\end{equation}
there is no need to use the completion process since $M=\widehat M=\Ltwo$ and then $\widehat\Cdiff=\Cdiff$, $\widehat{m_a}=m_a$, and $\mathrm R$ is the identity operator. The space 
\[
\mathcal H:=L^2(\Omega;\mathbb R^d)\times \Ltwo \times \Ltwo,
\]
is now endowed with the norm
\[
\|(\mathbf u,\mathrm E,\mathrm S)\|_{\mathcal H}^2
	:=(\rho\,\mathbf u,\mathbf u)_\Omega	
		+(\mathrm C_0\mathrm E,\mathrm E)_\Omega 
		+(a\,\Cdiff\mathrm S,\mathrm S)_\Omega,
\]
which is equivalent to the usual norm. This makes the analysis of this strictly diffusive viscoelastic problem much simpler.
\item[(c)] When $\Cdiff=0$, we have $M=\widehat M=\{0\}$ and the third equation and unknown do not play any role. In this case the operators $\pm\mathcal A$ are maximal dissipative and therefore, $\mathcal A$ is the generator of a group of isometries in $\mathcal H$, i.e., this model is conservative. This should not be a surprise, since in this case we recover a first order formulation
\[
\rho\,\dot{\mathbf u}=	\mathrm{div}\, \mathrm C_0 \mathrm E+\mathbf f^{(-1)},
	\qquad
\dot{\mathrm E}=\eps{\mathbf u}
\]
of the classical linear elastic wave equation. 
\end{itemize}

The introduction of non-zero initial conditions for the most general version of this model is not trivial. When the model is strictly diffusive (case (b) in the above discussion, i.e., when \eqref{eq:Cdiffplus} holds), we are allowed to impose initial conditions 
\[
\uu(0)=\uu_0, \qquad \dot{\uu}(0)=\mathbf v_0, \qquad \bs\sigma(0)=\bs\sigma_0,
\]
which would be the natural ones for the formulation \eqref{eq:9.11}. This is done by modifying the system \eqref{eq:9.2}, using initial conditions $\uu(0)=\uu_0$, $\mathrm E(0)=0$, $\mathrm S(0)=0$, and adding $\rho\,\mathbf v_0$ to the right-hand side of \eqref{eq:9.2a}
and $\mathrm G_0$, with
\[
\Cdiff\mathrm G_0=a\boldsymbol\sigma_0-\CC_1\eps{\uu_0}=
a\bs\sigma_0-(\Cdiff+a\CC_0)\eps{\uu_0},
\]
to the right-hand-side of \eqref{eq:9.2c}. The general case is much more complicated, given that the initial conditions for $\bs\sigma(0)$ must match $\CC_0\eps{\uu_0}$ in purely elastic subregions.

\section{More semigroup analysis}\label{sec:10}

We are now going to take advantage of the preparatory work of Section \ref{sec:8} to give a quick view of the formulations and estimates that can be obtained for the Maxwell and Voigt models. 

\subsection{Maxwell's model}

Maxwell's model can be understood as the particular case of Zener's model when $\mathrm C_0=0$ and $\Cdiff$ is strictly positive. However, its analysis is not included in the treatment given in Section \ref{sec:SEMI}, due to the fact that $\mathrm C_0$ was used to define the norm of the space $\mathcal H$. We thus start again, with a new space
\[
\mathcal H:=\mathbf L^2(\Omega)\times \Ltwo,
\]
endowed with the norm
\[
\|(\mathbf u,\mathrm S)\|_{\mathcal H}^2
	:=(\rho\mathbf u,\mathbf u)_\Omega
		+(a\,\Cdiff\mathrm S,\mathrm S)_\Omega,
\]
which is equivalent to the usual norm. The domain of the operator
\[
\mathcal A(\mathbf u,S):=
	(\rho^{-1}\mathrm{div}\, \Cdiff\mathrm S,
	m_a^{-1}(\eps{\mathbf u}-\mathrm S))
\]
is now
\[
D(\mathcal A):=\mathbf H^1_D(\Omega)
	\times
	\left\{ \mathrm S\in \Ltwo\,: 
		\mathrm{div}\,\Cdiff\mathrm S\in \mathbf L^2(\Omega),
		\gamma_N \Cdiff\mathrm S=0\right\}. 
\]
The operator $\mathcal A$ is maximal dissipative. The proof of surjectivity of $U\mapsto U-\mathcal AU$ starts with the solution of the coercive problem 
\begin{alignat*}{6}
& \mathbf u\in \mathbf H^1_D(\Omega),\\
& (\rho\,\mathbf u,\mathbf v)_\Omega
	+(\Cdiff (1+a)^{-1}\eps{\mathbf u},\eps{\mathbf v})_\Omega
		= && (\rho\,\mathbf f,\mathbf v)_\Omega\\
		 &&& -((a/(1+a)) \Cdiff  \mathrm G,\eps{\mathbf v})_\Omega
		\quad\forall\mathbf v\in \mathbf H^1_D(\Omega),
\end{alignat*}
for given $(\mathbf f,\mathrm G)\in \mathcal H$. (Note that the strict positivity of $\Cdiff$ is key for this argument to hold.) This is followed by the definition of
\[
\mathrm S=(1+a)^{-1}(\eps{\mathbf u}+a \mathrm G).
\]
Using the operator $\mathcal A$ in an equation of the form \eqref{eq:9.6}, we can prove that the hypotheses of Theorem \ref{th:9.2} are sufficient to provide a solution $(\mathbf u,\mathrm S)\in \mathcal C^1([0,\infty);\mathcal H)$ of the problem 
\begin{alignat*}{6}
\rho \dot{\mathbf u}(t)
	&=\mathrm{div} \Cdiff\mathrm S(t)
		+\mathbf f^{(-1)}(t) & \qquad & t\ge 0,\\
\mathrm S(t)+a\,\dot{\mathrm S}(t)
	&=\eps{\mathbf u(t)} && t\ge 0,\\
\gamma_D \mathbf u(t)
	&=\bs\alpha(t) &&t\ge 0,\\
\gamma_N \Cdiff\mathrm S(t)
	&=\bs\beta^{(-1)}(t) &&t\ge 0,
\end{alignat*}
with vanishing initial conditions. Introducing the stress tensor
\[
\bs\sigma(t):=\Cdiff\dot{\mathrm S}(t),
\]
we have a solution of 
\begin{alignat*}{6}
\rho\,\ddot{\mathbf u}(t) &=\mathrm{div}\,\bs\sigma(t)+\mathbf f(t) 
	& \quad & \mbox{a.e.}-t,\\
\bs\sigma(t)+a\,\dot{\bs\sigma}(t)
	&=\Cdiff\eps{\dot{\mathbf u}(t)}
		& & \mbox{a.e.}-t,\\
\gamma_D\mathbf u(t) &=\bs\alpha(t) && t\ge 0,\\
\gamma_N\bs\sigma(t) &=\bs\beta(t) && \mbox{a.e.}-t,
\end{alignat*}
with vanishing initial conditions. The estimates of Corollary \ref{cor:9.3} hold for this model as well. 

A combination of Zener's and Maxwell's models is also available. It requires an even more general framework so that $\mathrm C_0$ and $\Cdiff$ can vanish on separate parts of the domain as long as a certain combination stays strictly positive. (See the variational problem \eqref{eq:9.1} that is solved as a starting step to prove maximal dissipativity. As long as $\mathrm C_0+m_{1+a}^{-1}\Cdiff$ is a Hookean model, everything else will work.) The analysis would require a completion process with respect to the seminorm $(\mathrm C_0\cdot,\cdot)_\Omega^{1/2}$ and the corresponding restriction operator. This is a simple (while a little cumbersome) extension that the reader can do to prove their handle of the techniques developed above. 

\subsection{Voigt's model}\label{sec:10.2}

The analysis of Voigt's viscoelastic model (Zener's model with $a=0$, $\CC_0$ strictly positive and $\CC_\diff\ge 0$), including areas transitioning to classical linear elasticity, follows from a simple modification of the ideas of Section \ref{sec:SEMI}. In Voigt's model $\Cdiff=\CC_1$ plays the role of a dissipative term. The semigroup analysis of this model is slightly different in involving a second order differential operator. Like in Maxwell's model, there is no need to involve a completion process to handle transitions to classical linear elasticity. We now consider the following ingredients:
\begin{alignat*}{6}
\mathcal H &:= \mathbf L^2(\Omega)\times \mathbb L^2(\Omega),\\
\| (\mathbf u,\mathrm E)\|^2_{\mathcal H}
	&:=(\rho\mathbf u,\mathbf u)_\Omega+(\CC_0\mathrm E,\mathrm E)_\Omega,\\
D(\mathcal A) &:=
\left\{ (\mathbf u,\mathrm E)\in \mathbf H^1_D(\Omega)\times \mathbb L^2(\Omega)\,:\,
	\begin{array}{l}
	\mathrm{div}(\CC_0\mathrm E+\Cdiff \eps{\mathbf u})\in \mathbf L^2(\Omega),\\
	\gamma_N(\CC_0\mathrm E+\Cdiff \eps{\mathbf u})=0
	\end{array}\right\},\\
\mathcal A(\mathbf u,\mathrm E)
	&:=(\rho^{-1} \mathrm{div}(\CC_0\mathrm E+\Cdiff \eps{\mathbf u}),\eps{\mathbf u}).
\end{alignat*}
A simple argument shows that
\[
(\mathcal A(\mathbf u,\mathrm E),(\mathbf u,\mathrm E))_{\mathcal H}
=-(\Cdiff\eps{\mathbf u},\eps{\mathbf u})_\Omega\le 0
\qquad\forall (\mathbf u,\mathrm E)\in D(\mathcal A). 
\]
If we take $(\mathbf f,\mathrm F)\in \mathcal H$, solve the coercive problem
\begin{alignat*}{6}
& \mathbf u\in \mathbf H^1_D(\Omega),\\ 
& (\rho\,\mathbf u,\mathbf v)_\Omega
+((\mathrm C_0+\Cdiff )\eps{\mathbf u},\eps{\mathbf v})_\Omega
=&& (\rho\,\mathbf f,\mathbf v)_\Omega -(\mathrm C_0\mathrm F,\eps{\mathbf v})_\Omega \quad \forall\mathbf v\in \mathbf H^1_D(\Omega),
\end{alignat*}
and define $\mathrm E:=\eps{\mathbf u}+\mathrm F$, it is easy to prove that $(\mathbf u,\mathrm E)\in D(\mathcal A)$ and $(\mathbf u,\mathrm E)-\mathcal A(\mathbf u,\mathrm E)=(\mathbf f,\mathrm F)$. Therefore, $\mathcal A$ is maximal dissipative. 

Going carefully over the proof of Theorem \ref{th:9.2}, it is easy to see that with the same hypotheses on the data $\mathbf f$, $\bs\alpha$, and $\bs\beta$, we have a unique $(\mathbf u,\mathrm E)\in \mathcal C^1([0,\infty);\mathcal H)$, vanishing at zero, and solving 
\begin{alignat*}{6}
\rho \dot{\mathbf u}(t)
	&=\mathrm{div}(\mathrm C_0\mathrm E(t)+\Cdiff\eps{\uu(t)})
		+\mathbf f^{(-1)}(t) & \qquad & t\ge 0,\\
\dot{\mathrm E}(t)
	&=\eps{\mathbf u(t)} &&t\ge0,\\
\gamma_D \mathbf u(t)
	&=\bs\alpha(t) &&t\ge 0,\\
\gamma_N(\mathrm C_0\mathrm E(t)+\Cdiff\eps{\uu(t)})
	&=\bs\beta^{(-1)}(t) &&t\ge 0.
\end{alignat*}
The associated stress tensor is defined by 
\[
\bs\sigma(t):=\CC_0\dot{\mathrm E}(t)+\Cdiff\eps{\dot{\mathbf u}(t)}=
\CC_0\eps{\mathbf u(t)}+\Cdiff\eps{\dot{\mathbf u}(t)}
\]
and the resulting pair $(\mathbf u,\bs\sigma)$ is a solution to \eqref{eq:9.11} (with $a=0$) satisfying also the estimates of Corollary \ref{cor:9.3}.

\section{Some experiments}\label{sec:11}

We now present some numerical experiments of the various viscoelastic models that we described in Section \ref{sec:4}.  We use finite elements for space discretization and a trapezoidal rule-based convolution quadrature (TRCQ) for time discretization \cite{BaSc2012,HaSa2016,Lubich1994}.  The numerical experiments will be divided into three groups. First, we investigate 1D uniaxial wave propagation. Through these experiments, we observe how the viscoelastic behavior is dependent upon the choosing of parameters in the constitutive equations.  In the second group of experiments we compare the behaviors of 1D uniaxial wave propagation in elastic, classical viscoelastic, fractional viscoelastic, and heterogeneous models by plotting their 2D space-time contour graphs. In the heterogeneous model, we decompose the region into two subdomains with different viscoelastic models, where the reflection and refraction of waves can be observed at the transition interface.  Finally, we present the 3D simulation of a viscoelastic rod. Similar to the heterogeneous domain in the previous experiments, the rod is decomposed into two different subdomains.  The snapshots we present show how the simulation accurately captures the memory and relaxation effects of the rod under a sudden change in displacement. The numerical analysis of the discretization schemes employed in this section is the goal of future research. Tests have been performed in sufficiently refined space-and-time meshes to obtain some sort of eye-ball convergence to a solution.

\subsection{1D experiments} \label{Sec:NumEx1}

For simplicity, in the one-dimensional examples we will use traditional PDE notation, as opposed to the notation of evolutionary equations used throughout the paper.
We first present numerical experiments for different fractional models in one dimension
\begin{subequations} \label{eq:11.1}
\begin{alignat}{4}
\rho u_{tt} &= \sigma_x \qquad && x\in [0,1], \quad t\in [0,40],\\
\label{eq:11.1b}
u(0,t) &= g(t) && t\in [0,40],\\
\label{eq:11.1c}
\sigma(1,t) &= 0&& t\in [0,40],\\
u(x,0) &= u_t(x,0) = 0 &\qquad & x\in [0,1],
\end{alignat}
where the constitutive relation that determines the model is defined through $\sigma$ and $\rho$.  We use the window function displayed in the left of Figure \ref{fig:window} as Dirichlet boundary data at $x = 0$, while we take a homogeneous Neumann boundary condition at $x=1$. 
The strain-to-stress relation is given by a general formula
\begin{equation}
\sigma+ a\,\partial^\nu_t \sigma=\mathrm C_0 u_x+\mathrm C_1 \partial^\nu_t u_x,
\end{equation}
\end{subequations}
for parameters $\mathrm C_0, \mathrm C_1$, $a$ and $\nu$ to be determined. For the discretization in space we use $\mathcal P_4$ finite elements on a mesh with 513 subintervals of equal size. Discretization in time is carried out using a TRCQ with 10,240 time-steps of equal size in the interval $[0,40]$. 
\begin{figure}[H]
\centering
\includegraphics[width=0.4 \textwidth]{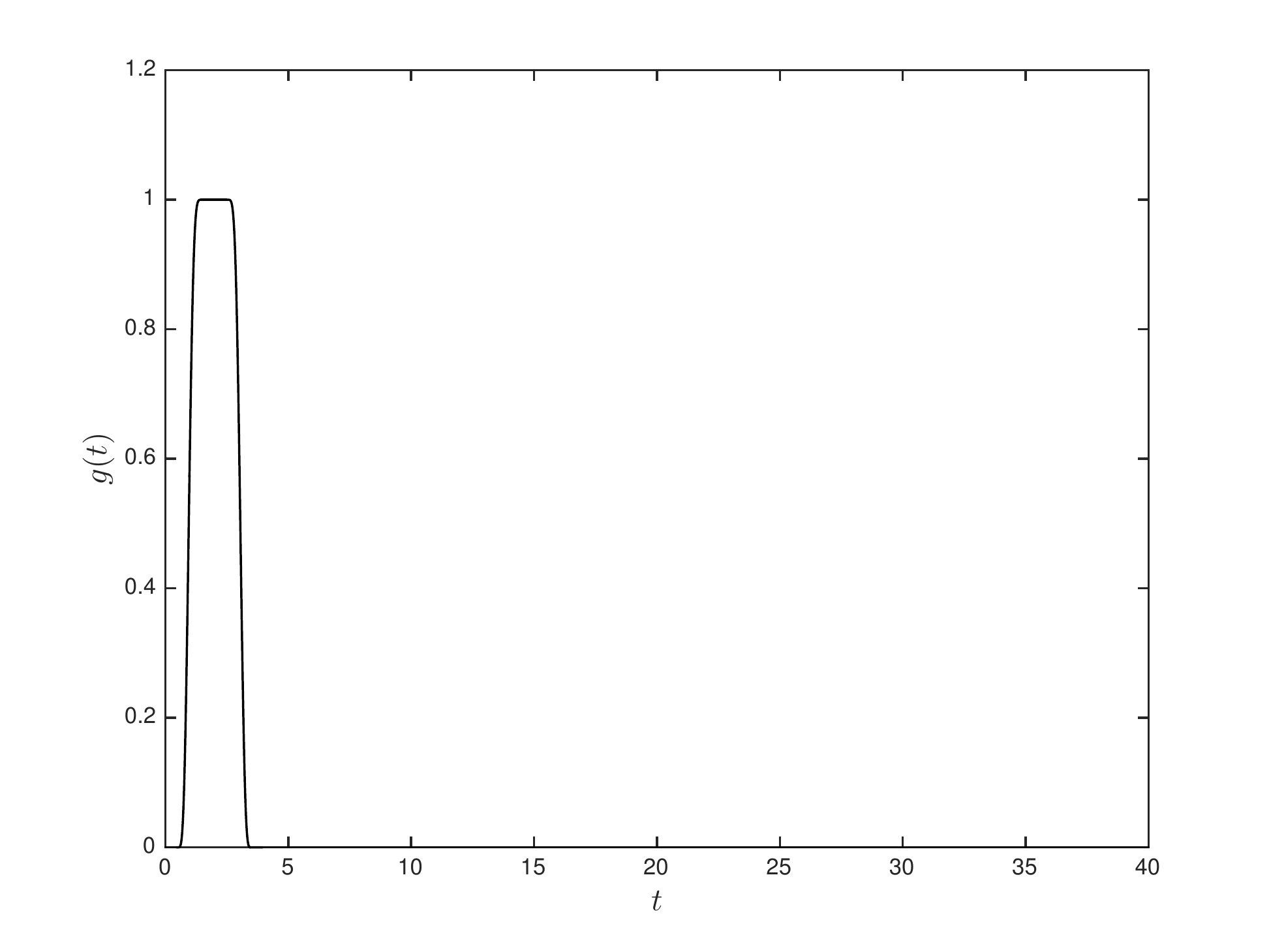} 
\includegraphics[width=0.4 \textwidth]{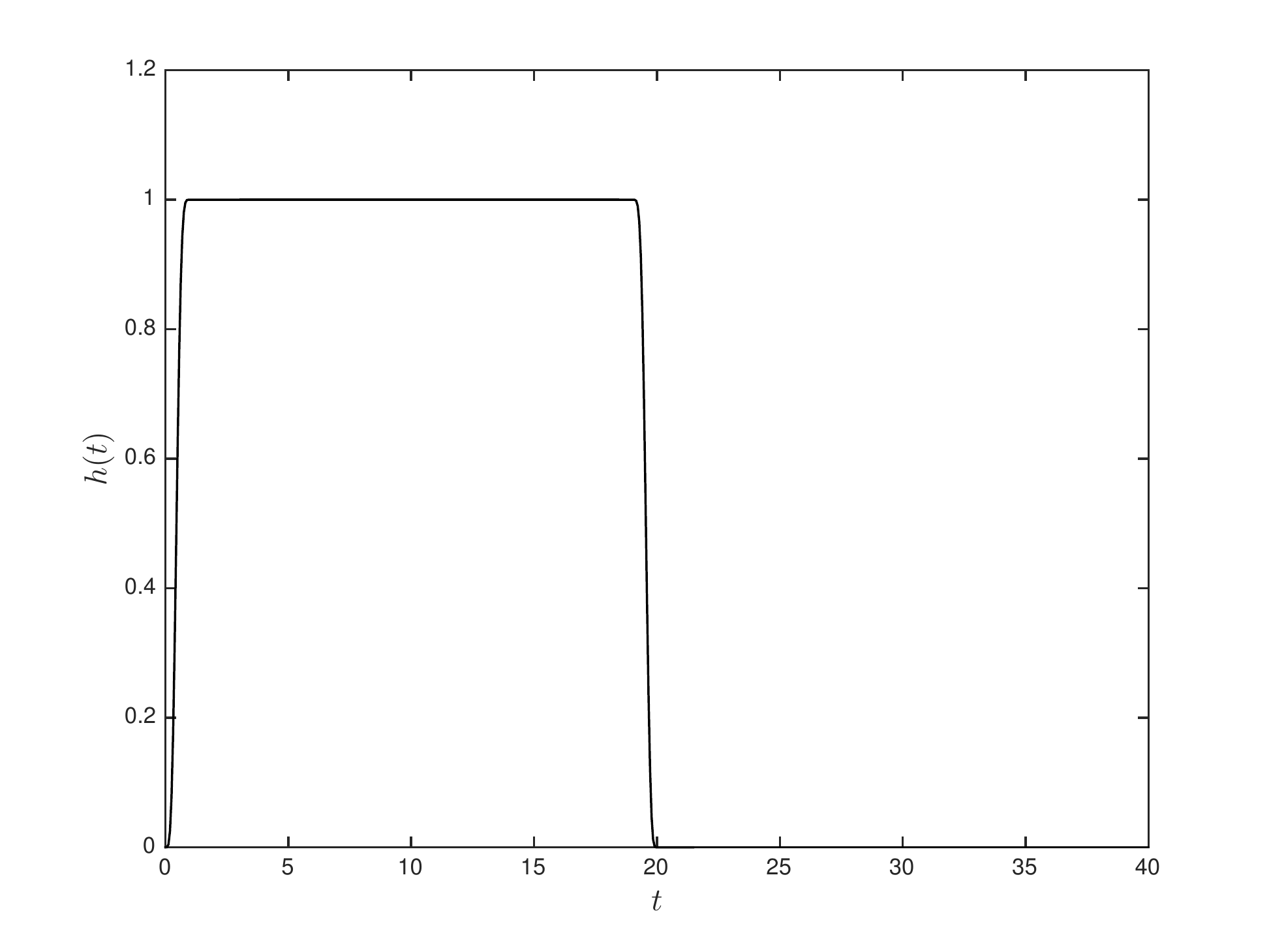}
\caption{{\footnotesize The window function $g(t)$ is smoothly changing in $(0.5,1.5)\cup(2.5,3.5)$, and constant on the rest of the domain. The function $h(t)$ has a similar shape, although the upper plateau (forced normal stress) is longer.}}\label{fig:window}
\end{figure}
For the first two sets of experiments, we implement the Dirichlet boundary condition $g(t)$ in Figure \ref{fig:window} at $x=0$ and observe the evolution of $u(1,t)$.  We use: (a) the values given in Table \ref{table00}, varying $\CC_1$, to create the results of Figure \ref{fig:plotC1all}, and (b) the values in Table \ref{table01}, varying $\nu$, for the results of Figure \ref{fig:plotnuall}.
%
%table for Fig1 and 2
\begin{table}[H] \center{
\begin{tabular}{|l|c|c|c|}
\hline 
& Zener  & Maxwell & Voigt \\
\hline
$\mathrm C_0$ & 1.5 & 0 & 1.5 \\
\hline
$\mathrm C_1$ & $0.75, 1, 2.75$ & $0.05, 0.25, 2$ & $0, 0.25, 2$  \\
\hline
$a$ & 0.5 & 0.5 & 0 \\
\hline
$\nu$ & $1$ & $1$ & $1$ \\
\hline
$\rho$ & $1$ & $1$ & $1$ \\
\hline
\end{tabular}
\caption{{\footnotesize The parameters used to create the plots given in Figure \ref{fig:plotC1all}. The first choice of $\mathrm C_1$ for the Zener and Voigt models reduce the model to linear elasticity.}}
\label{table00}}
\end{table}
%
%% showing change of C_1
\begin{figure}[H]
\centering
%trim left right bottom top 
\foreach \index in {1, 2, 3} {%Zener%Maxwell%Voigt
	\includegraphics[scale=0.3]{c1changeAll\index.pdf}
}
\caption{{\footnotesize Effect of changing $\mr C_1$ using the parameters given in Table \ref{table00}.}}\label{fig:plotC1all}
\end{figure}
In Table \ref{table00} the values for $\mr C_1$ are chosen so that the parameter $\mr C_{\mr{diff}}$, which controls the diffusion, takes same values for the three different models, except for Maxwell, where $\mr C_{\mr{diff}}=0$ results in a model that is identically zero. To avoid this while still getting comparable results, we use $\mr C_1 = 0.05$ for the first value in the Maxwell model.  In Figure \ref{fig:plotC1all}, as we increase $\mr C_1$, all three models show a faster energy dissipation. Compared to the Zener and Voigt models, the Maxwell model exhibits less oscillations as a response to the Dirichlet boundary condition.
%
%%
%% showing change of nu
%table for changin nu
\begin{table}[H] \center{
\begin{tabular}{|l|c|c|c|}
\hline 
& Zener  & Maxwell & Voigt \\
\hline
$\mathrm C_0$ & 1.5 & 0 & 1.5 \\
\hline
$\mathrm C_1$ & $1$ & $1$ & $1$  \\
\hline
$a$ & 0.5 & 0.5 & 0 \\
\hline
$\nu$ & $0.05, 0.5, 0.95$ & $0.05, 0.5, 0.95$ & $0.05, 0.5, 0.95$ \\
\hline
$\rho$ & $1$ & $1$ & $1$ \\
\hline
\end{tabular}
\caption{{\footnotesize The parameters used to create the plots given in Figure \ref{fig:plotnuall}.}}
\label{table01}}
\end{table}

\begin{figure}[H]
\centering
%trim left right bottom top 
\foreach \index in {1, 2, 3} {%Zener%Maxwell%Voigt
	\includegraphics[scale=0.3]{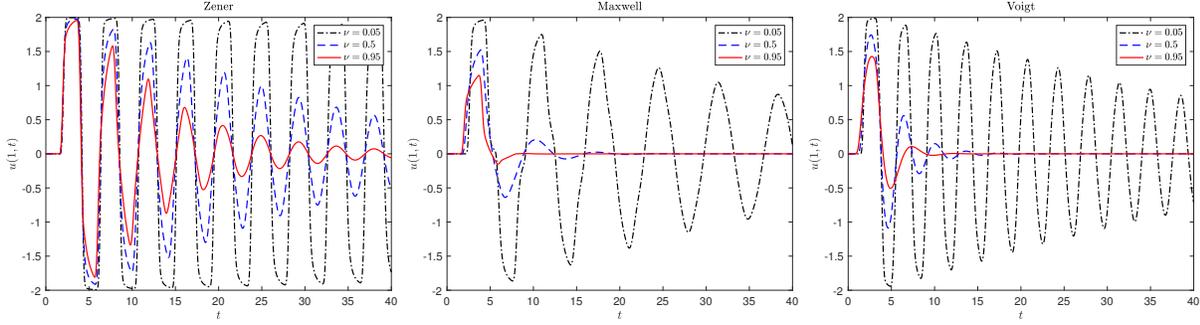}
}
\caption{{\footnotesize Effect of changing $\nu$ using the parameters given in Table \ref{table01}.}}\label{fig:plotnuall}
\end{figure}

In Figure \ref{fig:plotnuall}, we observe that the decreasing the fractional power $\nu$  leads to a slower rate of energy dissipation. We also notice that the change of $\nu$ does not have any obvious effect on the frequency of the oscillations.

We now change the boundary conditions \eqref{eq:11.1b} and \eqref{eq:11.1c} to
\[
u(0,t)=0 \qquad \sigma(1,t)=h(t),
\]
where $h$ is the function in the right of Figure \ref{fig:window}.
Once again, we vary $\nu$ to see the effects that the fractional order of the derivative has on the Zener, Maxwell, and Voigt models respectively.  The parameters we choose for this experiment is given in Table \ref{table03} and we again plot $u(1,t)$ in Figure \ref{fig:multiPlts}.
%table for sudden stess
\begin{table}[H] \center{
\begin{tabular}{|l|c|c|c|}
\hline 
& Zener  & Maxwell & Voigt \\
\hline
$\mathrm C_0$ & 1 & 0 & 1 \\
\hline
$\mathrm C_1$ & 1 & 1 & 1 \\
\hline
$a$ & 0.5 & 0.5 & 0 \\
\hline
$\nu$ & $0.25,0.5,0.75,1$ & $0.25,0.5,0.75,1$ & $0.25,0.5,0.75,1$ \\
\hline
$\rho$ & $1$ & $1$ & $1$ \\
\hline
\end{tabular}
\caption{{\footnotesize The parameters used to create the plots given in Figure \ref{fig:multiPlts}}}
\label{table03}}
\end{table}
%figures for sudden stress
\begin{figure}[H]
\centering
\includegraphics[scale=0.3]{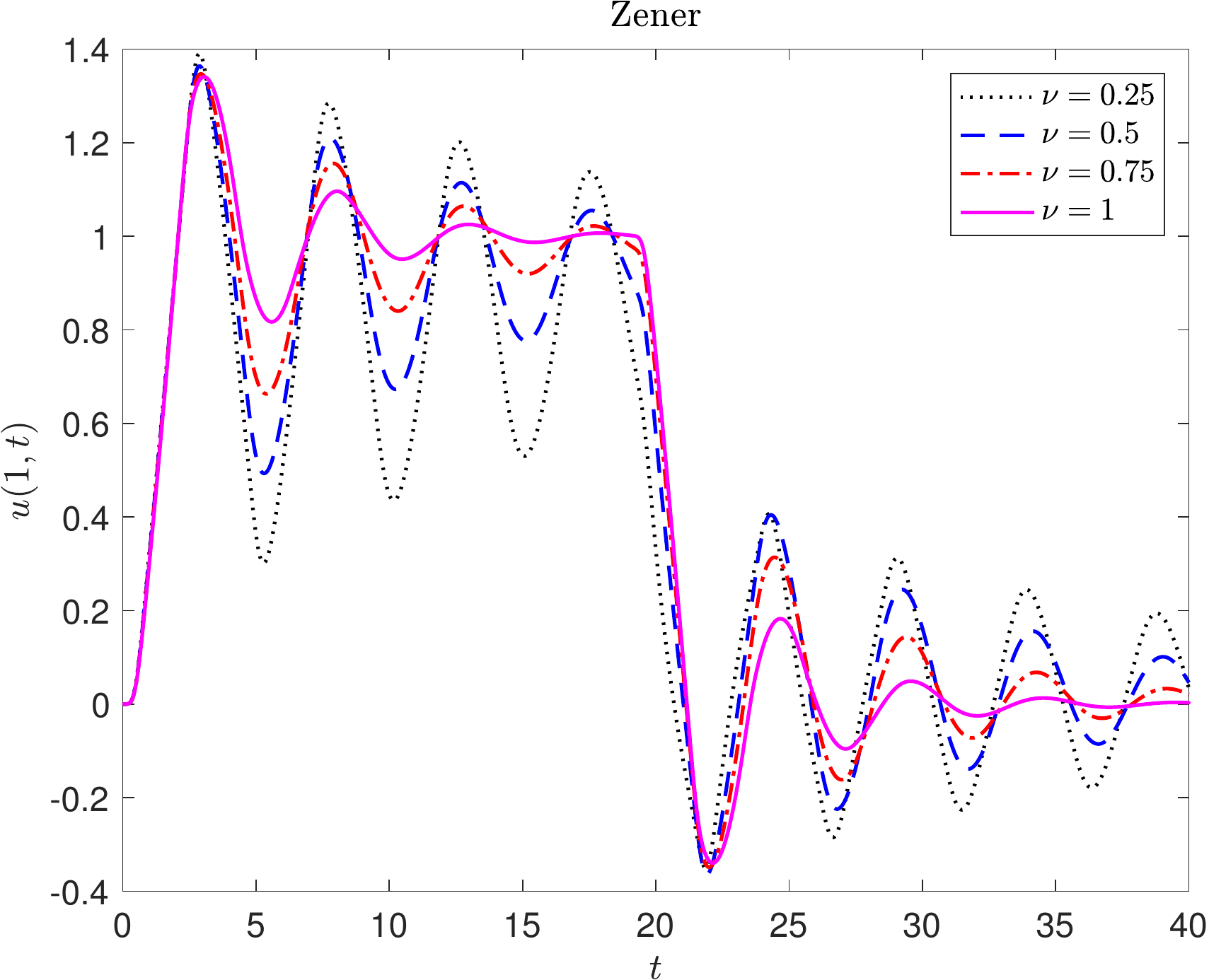}
\includegraphics[scale=0.3]{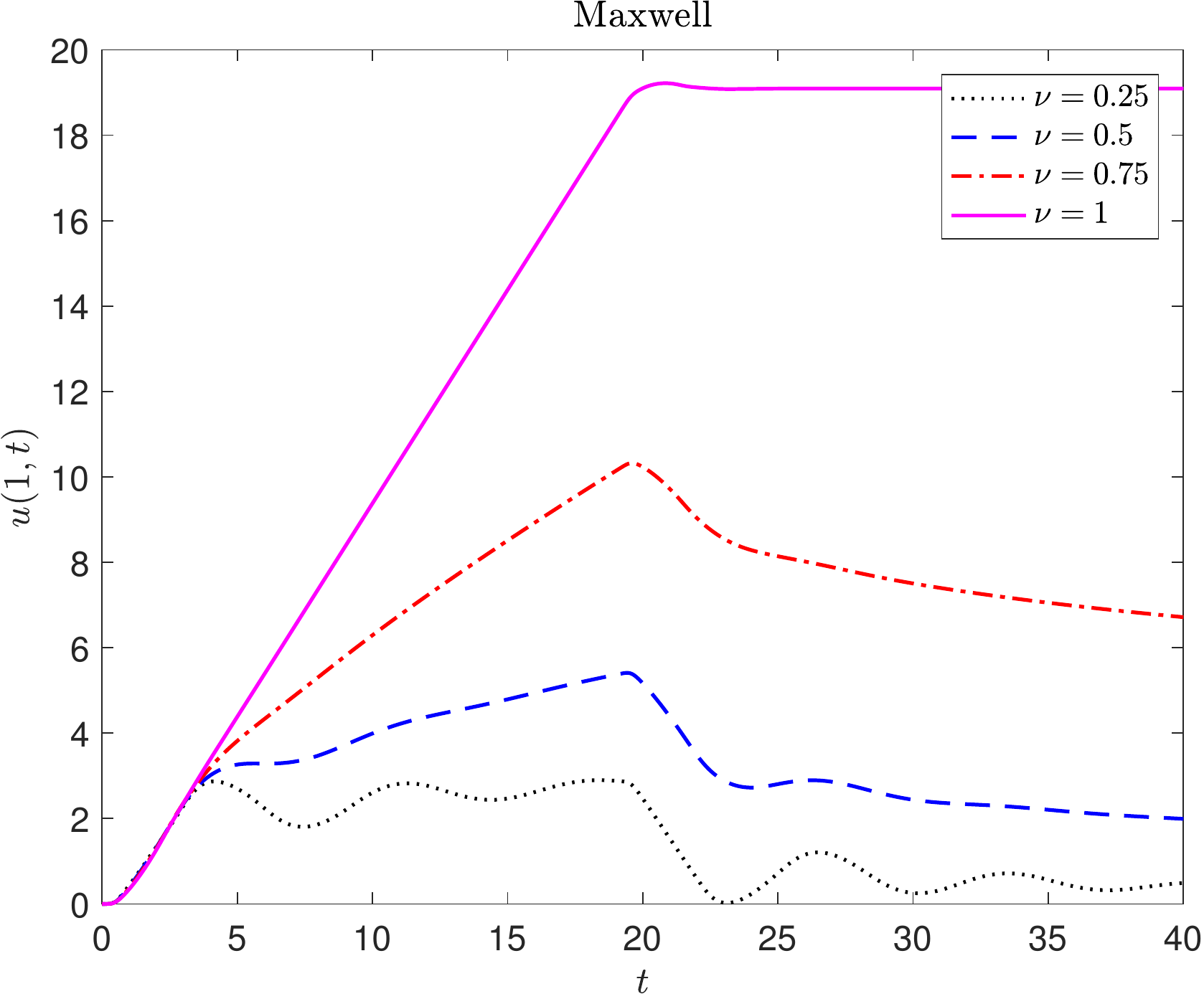}
\includegraphics[scale=0.3]{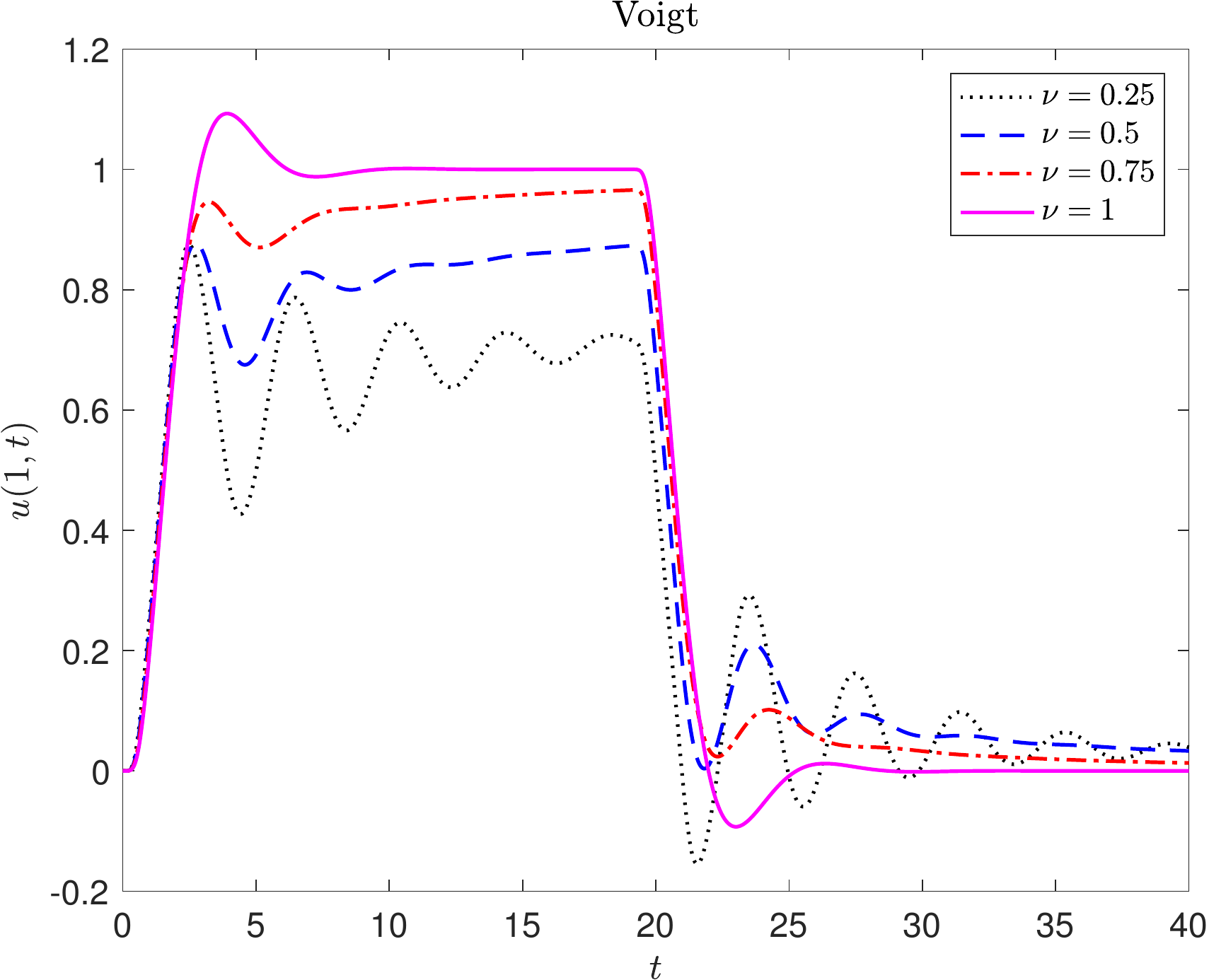}
\caption{{\footnotesize Each plot shows the displacement of a 1D viscoelastic rod evaluated at the right endpoint, $x=1$ for different values of $\nu$. From left to right are the Zener, Maxwell, and Voigt models.}} \label{fig:multiPlts}
\end{figure}

Here, both Zener and Voigt models show exponential rate response to the suddenly applied traction $h(t)$ and then converge to some equilibrium state. The Maxwell model, as expected, shows a linear rate of creep when the traction is constant in time. As the fractional power $\nu$ decreases, we observe that the amplitude of oscillations is larger for the Zener and Voigt models, on the other hand we see the creep rate is slower for the Maxwell model.

\subsection{Spacetime plots}\label{sec:11.2}
We now study space-time contour plots of the displacement solution of (\ref{eq:11.1}) with Zener, Maxwell and Voigt models. We show three simulations focusing on one of these models where in each experiment we compare it with its fractional version, and observe its behavior in a heterogenous domain coupled with an elastic model. When we work on a heterogenous domain we split the interval $[0,1]$ into $[0,1/2)$ and $[1/2,1]$, where the first half is elastic and the second half is one of the models we are comparing: Zener, Maxwell or Voigt. We show space-time plots of the displacement corresponding to different models side by side where elastic model is included in all cases for the sake of comparison. We implement two different signals as Dirichlet boundary condition: a single pulse and a train of pulses (see Figure \ref{fig:sigs}). For each experiment we display eight plots where the first four are the results of a single pulse while the last four are the results of the same experiment but for a train of pulses. 

\begin{figure}[H] \center{
\includegraphics[width = 0.45\textwidth]{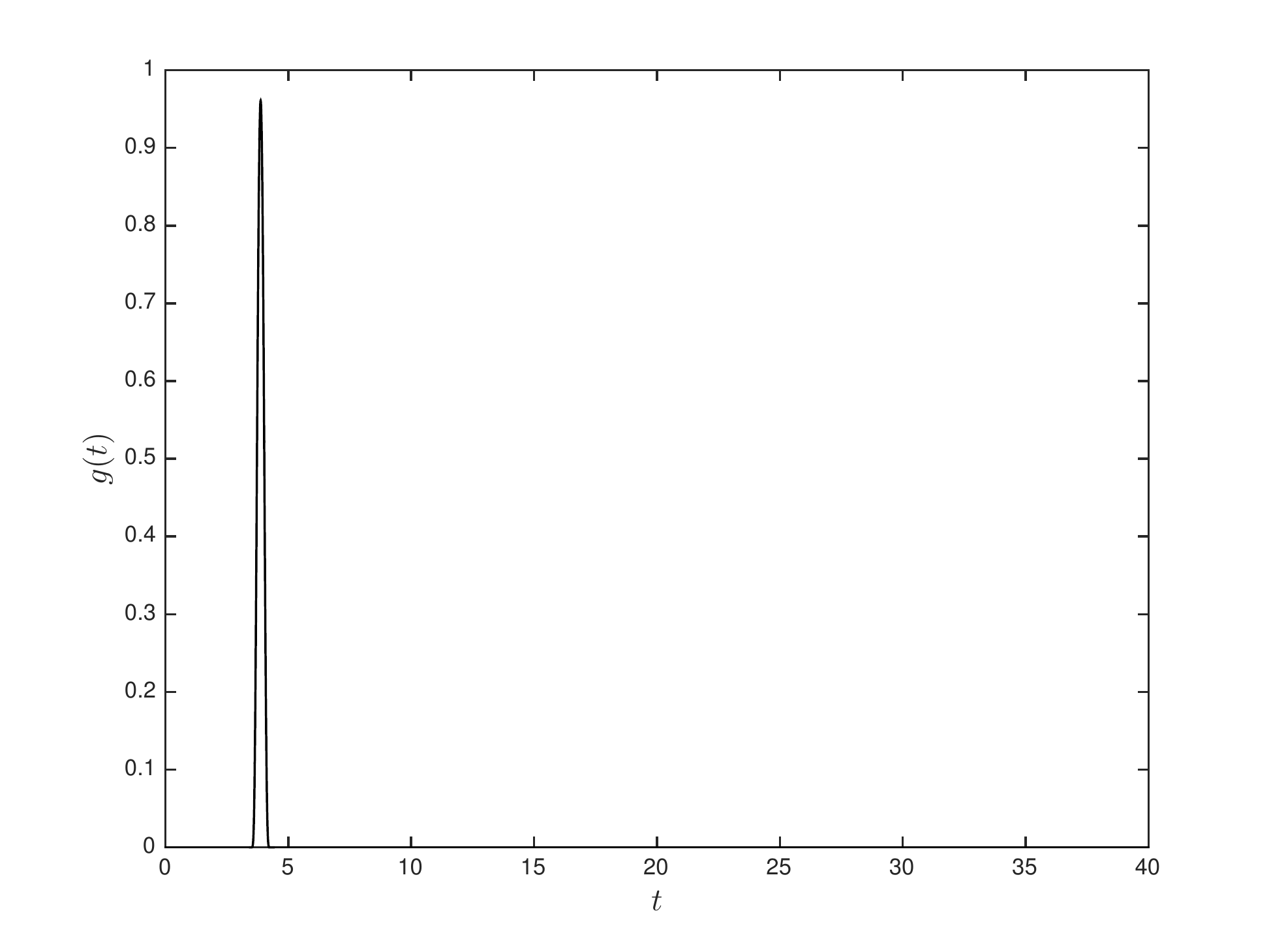}
\includegraphics[width = 0.45\textwidth]{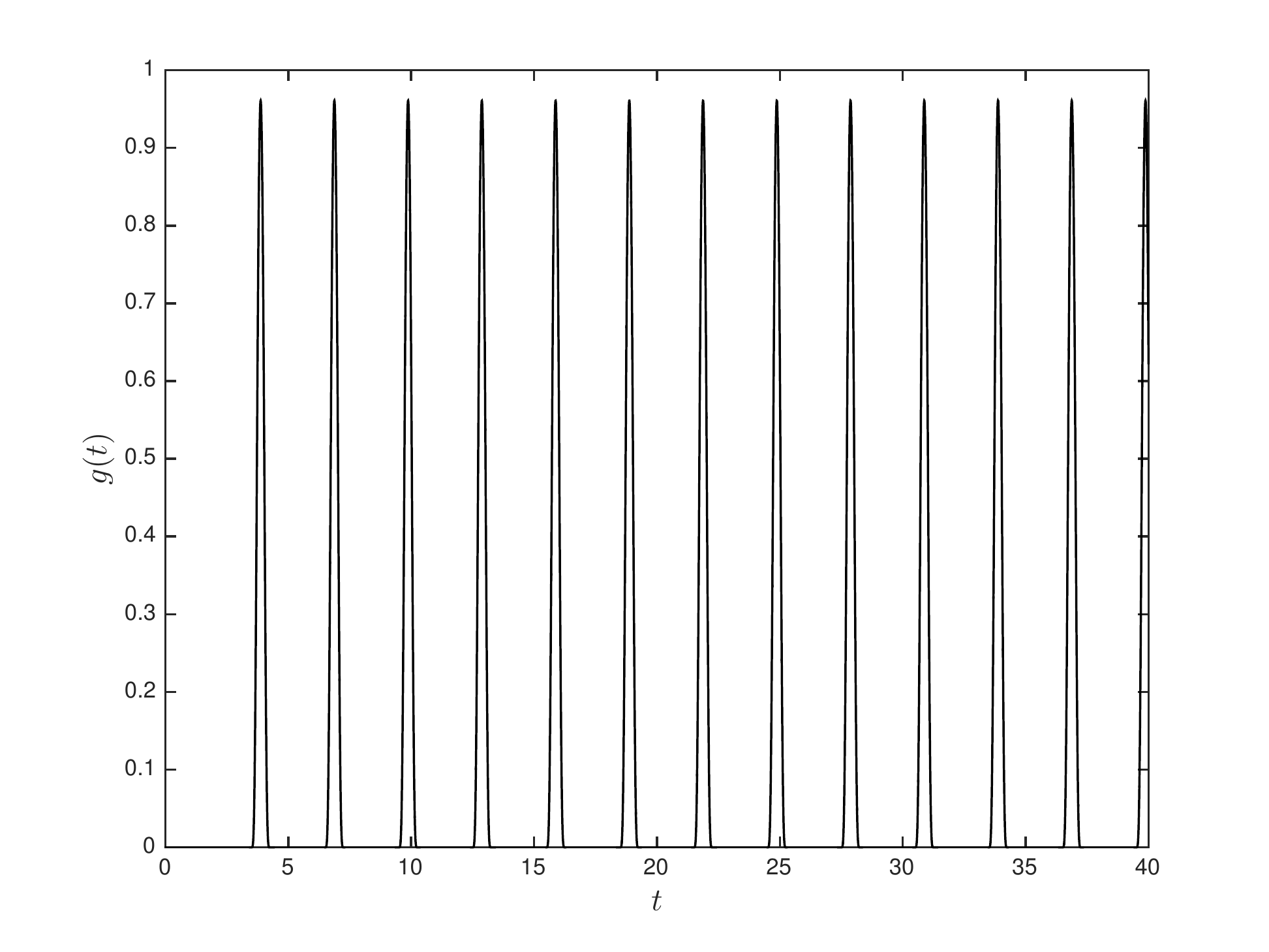}
\caption{{\footnotesize The Dirichlet boundary conditions used for the 2D spacetime simulations: single pulse and a periodic train of pulses, which will make the solution transition to time-harmonicity.}} \label{fig:sigs}}
\end{figure}

Our first test focuses on the Zener model where we utilize the parameters given in Table \ref{tab:Zener}, and Figure \ref{fig:2DZener} shows the outcome of this experiment. Here, the contrast between the energy conservative elastic model and the dissipative Zener model can easily be seen. We also notice that fractional Zener model displays slower dissipation than the classical model. In the heterogenous domain we observe a reflection and refraction of waves at the interface $x=1/2$. Although the elastic model is conservative, in the case of a heterogenous domain we see the dissipation of the Zener model affecting the coupled system with a loss in wave amplitude.  We also observe that in the results whose Dirichlet boundary condition is a train of pulses (the four panels on the right), the solution enters quickly into a time-harmonic regime.
%
%zener table
\begin{table}[H] \center{
\begin{tabular}{|l|c|c|c|c|}
\hline 
& Elastic & Zener & Fractional Zener & Heterogeneous Domain\\
\hline
$\mathrm C_0$ & 1.75 & 1.5 & 1.5& 1.75 $(x<1/2)$, \: 1.5 $(x \ge 1/2)$\\
\hline
$\mathrm C_1$ & 1.75 & 1.75 & 1.75 & 1.75\\
\hline
$a$ & 1 & 0.5 & 0.5 & 1 $(x<1/2)$, \: 0.5 $(x \ge 1/2)$\\
\hline
$\nu$ & 1 & 1 & 0.3 & 1\\
\hline
$\rho$ & 10 & 10 & 10 & 10 \\
\hline
\end{tabular}
\caption{{\footnotesize The parameters used in in the fractional Zener model simulations to create the 2D spacetime plots shown in Figure \ref{fig:2DZener}}}
\label{tab:Zener}}
\end{table}
% Zener simulations
\begin{figure}[H] \center{
\includegraphics[clip,trim={1.9cm 0cm 1cm 0cm}, scale=0.47]{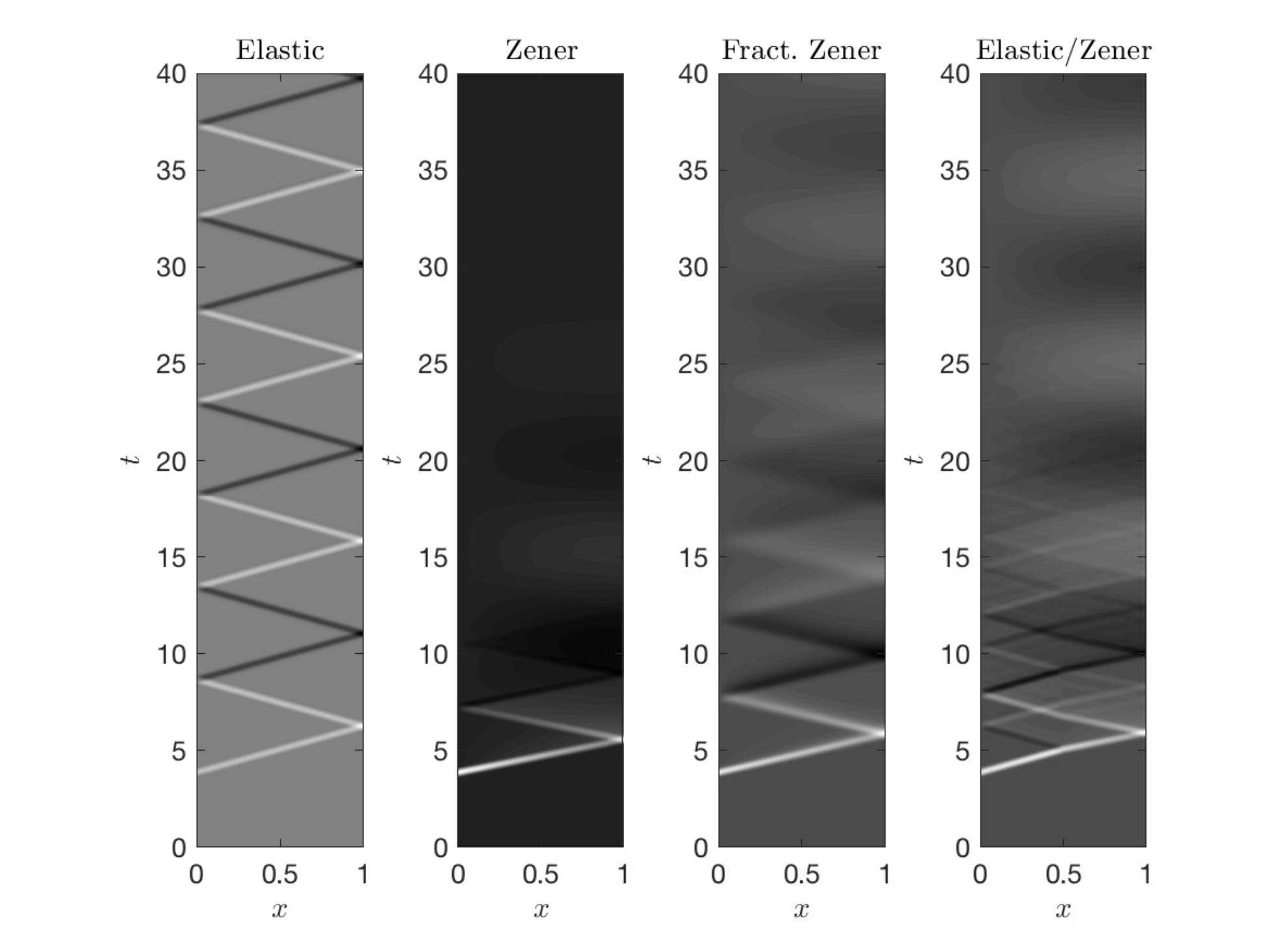}
\includegraphics[clip,trim={1.1cm 0cm 1.8cm 0cm}, scale=0.47]{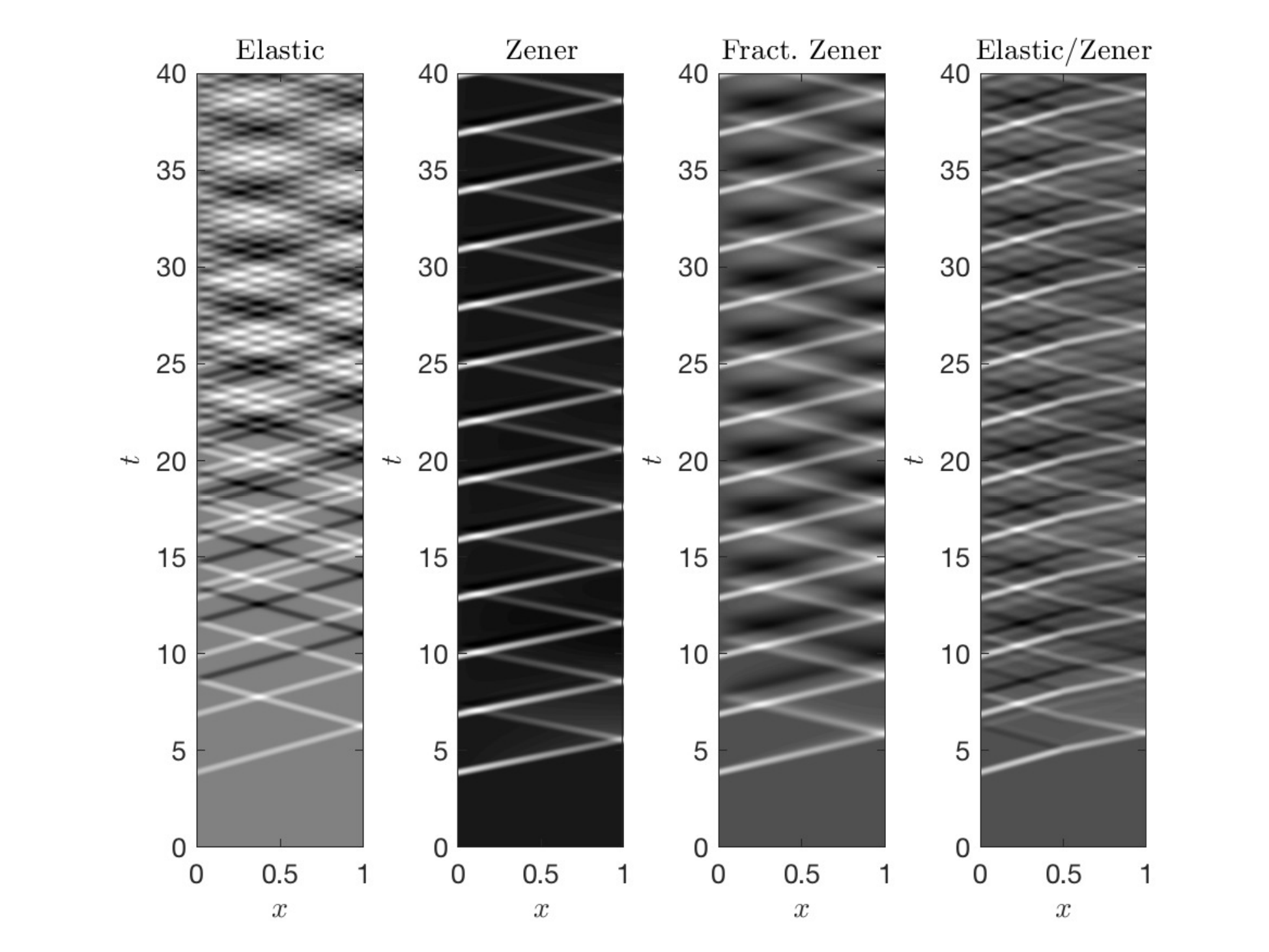}
\caption{{\footnotesize Space-time plots with parameters described in Table \ref{tab:Zener} where first and last four subplots correspond to the signals shown on the left and right of Figure \ref{fig:sigs}.}} \label{fig:2DZener} }
\end{figure}
We perform a similar comparison for the Maxwell model using the parameters given in Table \ref{tab:maxwell} and collecting the results in Figure \ref{fig:2DMaxwell}. When a single pulse is used, we notice that the waves in the Maxwell model exhibit dissipation. Moreover, the fractional Maxwell reveals less dissipation with a little dispersion. In the heterogenous domain the reflections at $x=1/2$ are less obvious for both of the input signals when comparing to the corresponding Zener simulations.
%
%maxwell table
\begin{table}[H]
\centering
\begin{tabular}{|l|c|c|c|c|}
\hline 
		& Elastic& Maxwell & Fractional Maxwell & Heterogeneous Domain\\
\hline
$\mathrm C_0$ & 1.75 & 0 & 0& $1.75(x<1/2)$, \: 0 $(x \ge 1/2)$\\
\hline
$\mathrm C_1$ & 1.75 & 1.75 & 1.75 & 1.75\\
\hline
$a$ & 1 & 1 & 1 & 1\\
\hline
$\nu$ & 1 & 1 & 0.3 & 1\\
\hline
$\rho$ & 10 & 10 & 10 & 10 \\
\hline
\end{tabular}
\caption{{\footnotesize The parameters used in the Maxwell model simulations to create the 2D spacetime plots shown in Figure \ref{fig:2DMaxwell}.}}\label{tab:maxwell}
\end{table}
%
%Maxwell Simulation
\begin{figure}[H] \center{
\includegraphics[clip,trim={1.9cm 0cm 1cm 0cm}, scale=0.47]{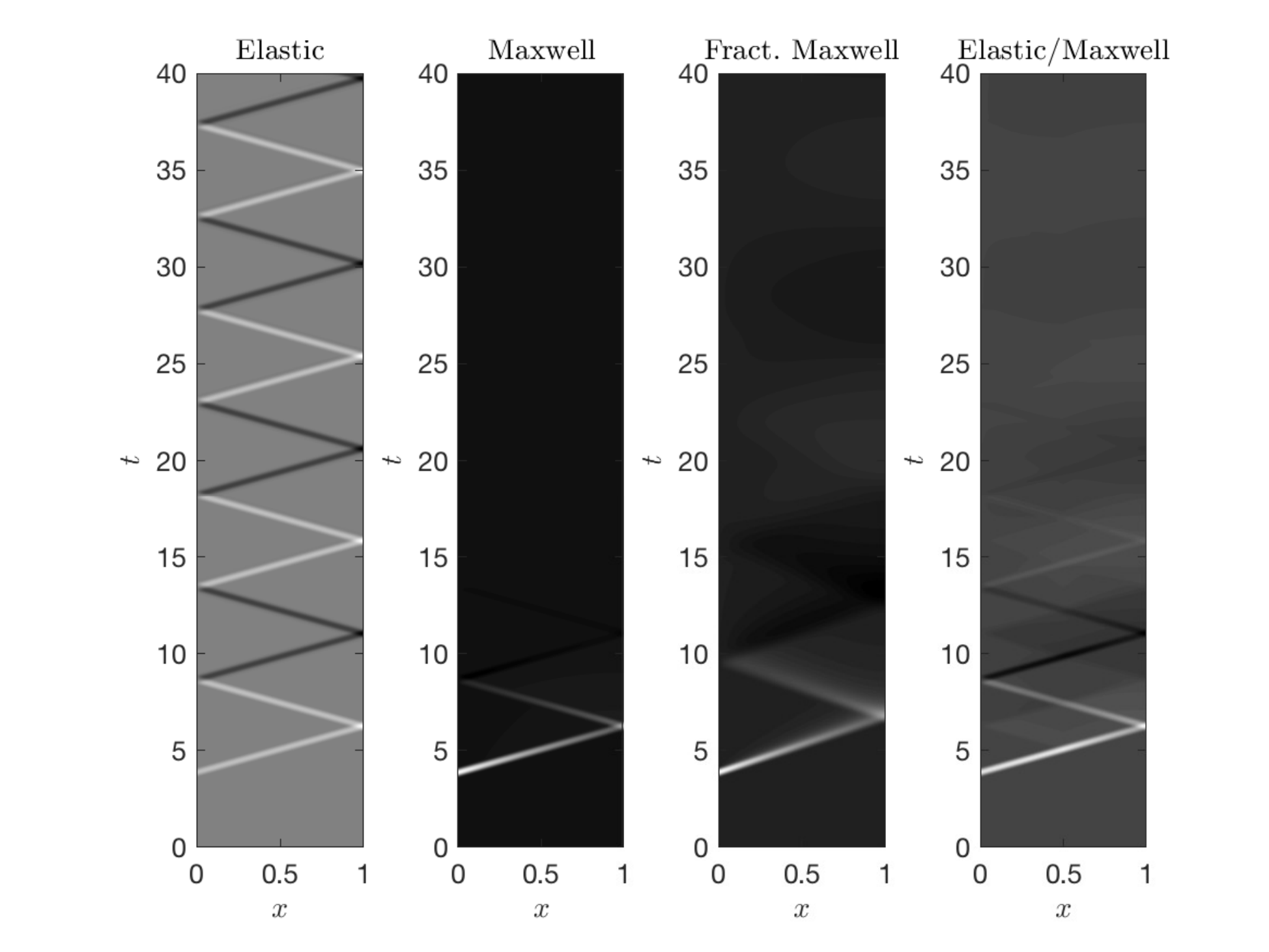}
\includegraphics[clip,trim={1.5cm 0cm 1.5cm 0cm}, scale=0.47]{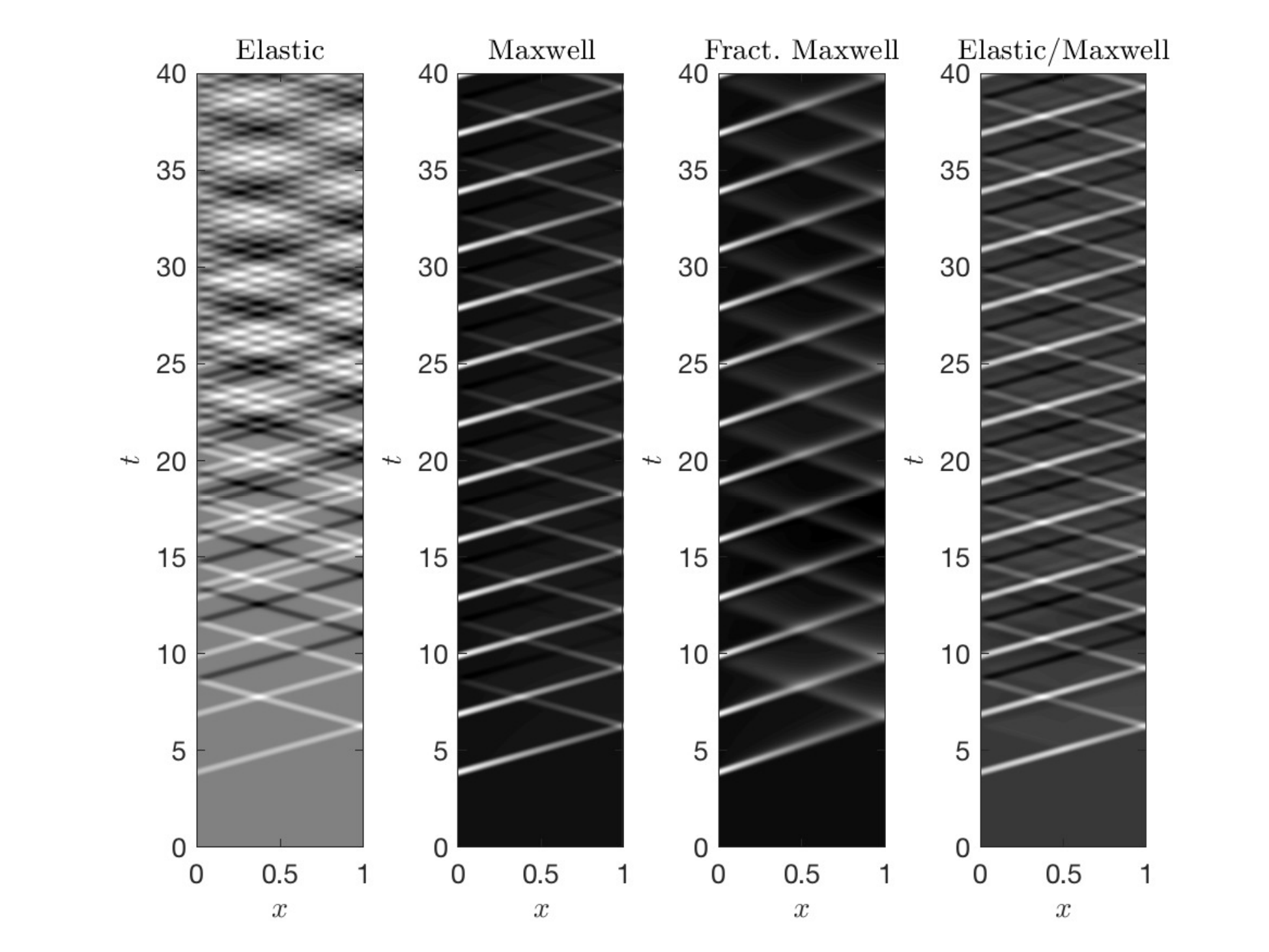}
\caption{{\footnotesize Space-time plots with parameters described in Table \ref{tab:maxwell} where first and last four subplots correspond to the signals shown on the left and right of Figure \ref{fig:sigs}.} }\label{fig:2DMaxwell} }
\end{figure}
Lastly, we demonstrate the space-time plots corresponding to the Voigt model in Figure \ref{fig:2DVoigt} using the parameters from Table \ref{tab:Voigt}. Comparing to Zener and Maxwell, we observe that this model displays more dispersion. In particular, this dispersion is on dramatic display in the heterogenous domain where the reflections are clearly seen in the elastic part, while the dispersion occurs in the Voigt subdomain.
%
%Table Voigt Simulation
\begin{table}[H] 
\centering
\begin{tabular}{|l|c|c|c|c|}
\hline 
		& Elastic& Voigt & Fractional Voigt & Heterogeneous Domain\\
\hline
$\mathrm C_0$ & 1.75 & 1.75 & 1.75& 1.75\\
\hline
$\mathrm C_1$ & 1.75 & 1.75 & 1.75 & 1.75\\
\hline
$a$ & 1 & 0 & 0 & 1 $(x<1/2)$, \: 0 $(x \ge 1/2)$\\
\hline
$\nu$ & 1 & 1 & 0.3 & 1\\
\hline
$\rho$ & 10 & 10 & 10 & 10 \\
\hline
\end{tabular}
\caption{{\footnotesize The parameters used in the fractional Voigt model simulations to create the 2D spacetime plots shown in Figure \ref{fig:2DVoigt}.}}\label{tab:Voigt}
\end{table}

%Voigt Simulation
\begin{figure}[H] \center{
\includegraphics[clip,trim={1.9cm 0cm 1cm 0cm}, scale=0.47]{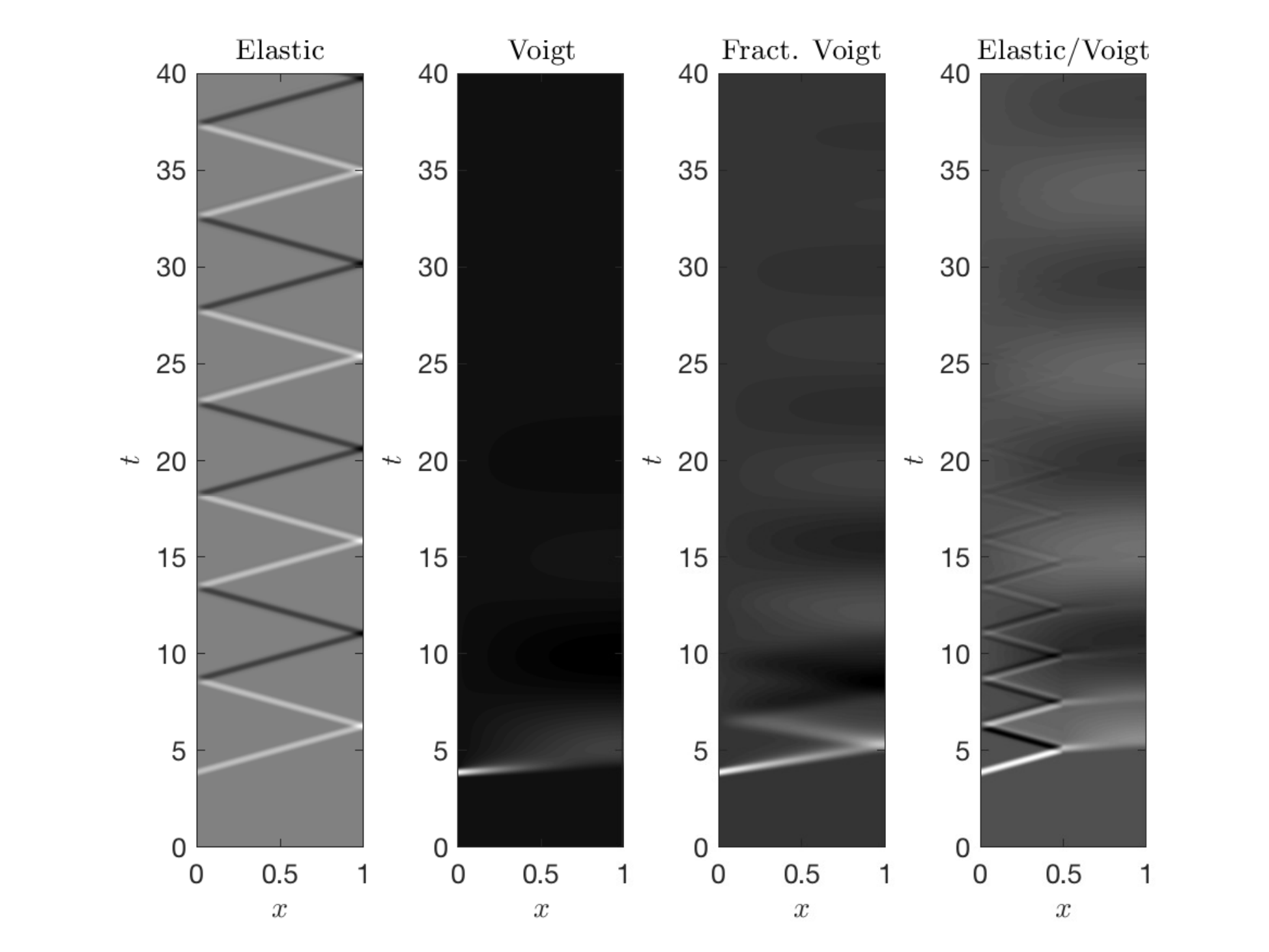}
\includegraphics[clip,trim={1.1cm 0cm 1.9cm 0cm}, scale=0.47]{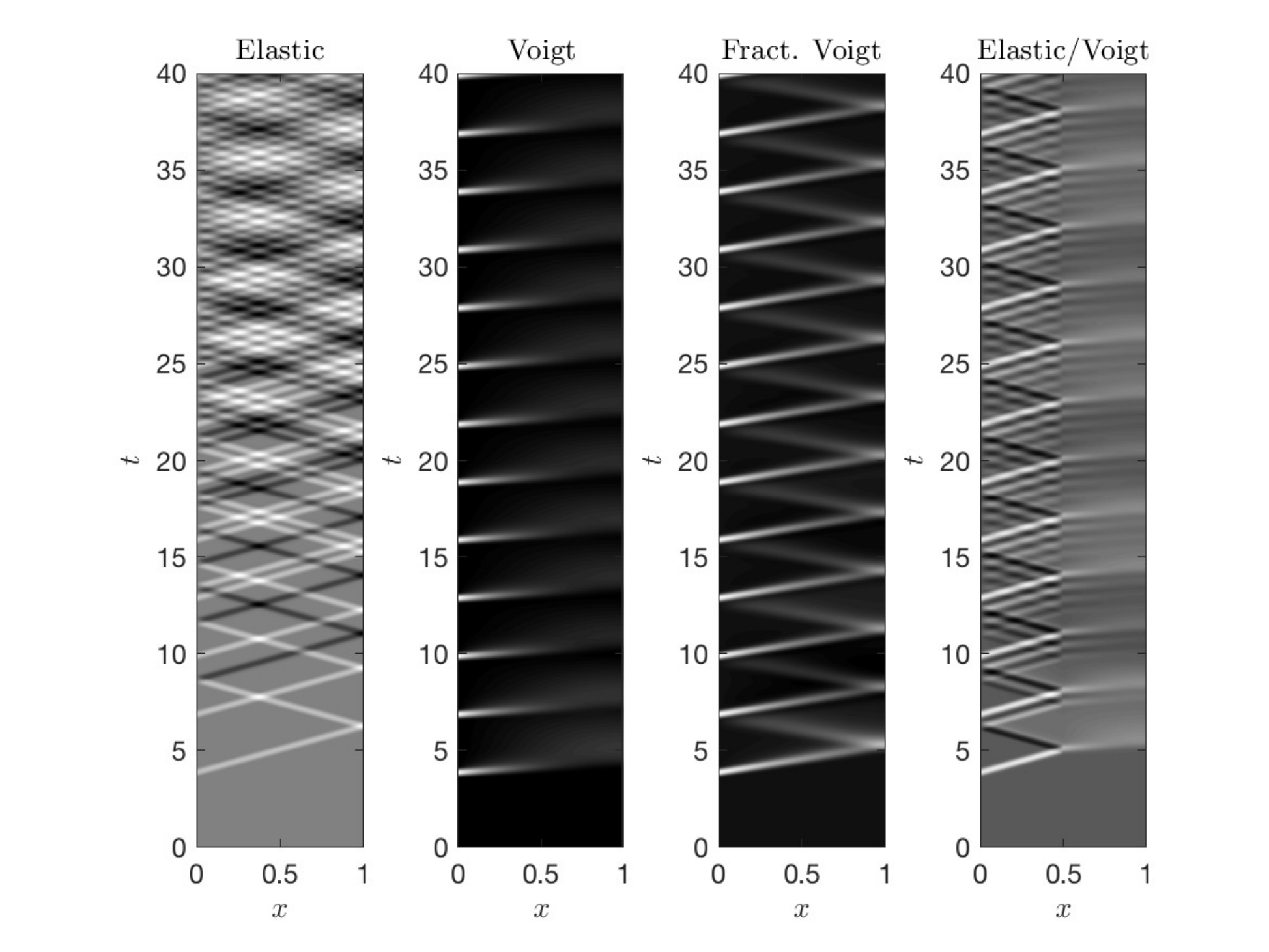}
\caption{{\footnotesize Space-time plots with parameters described in Table \ref{tab:maxwell} where first and last four subplots correspond to the signals shown on the left and right of Figure \ref{fig:sigs}.}} \label{fig:2DVoigt} }
\end{figure}

\subsection{3D numerical simulation} \label{sec:num_sim}

We now present a numerical simulation for viscoelastic waves propagating in the parallelepiped $\Omega=(0,1)\times (0,10)\times (0,1)$ with a Dirichlet boundary on one of the small faces $\Gamma_D:= (0,1)\times\{0\}\times (0,1)$.
The PDE we are simulating is
\begin{align*}
\ddot{\mb u}(t) &= \mathrm{div}\,\bs\sigma(t) 
&&\Omega\times[0,50],\\
\gamma_D\mb u(t) & = 0.25
\left( w(t), 0, 0 \right)^\top
&&\Gamma_D\times[0,50],\\
\gamma_N\bs \sigma(t) & = 0
&&\Gamma_N\times[0,50],
\end{align*}
where $w(t)$ represents an enforced displacement at the Dirichlet boundary that takes value $0$ in $[0,0.5]$, $1$ in $[1,50]$ and transitions smoothly from $0$ to $1$ in the interval $[0.5,1]$. Initial conditions are set to zero.
The material is isotropic and locally homogeneous, with a strain-stress law given 
by
\begin{align*}
\bs\sigma + \partial^\nu\bs\sigma=
2\eps{\uu}+(\nabla\cdot\uu) \mb I
+5 \left( 2\eps{\partial^\nu \uu}+(\nabla\cdot\partial^\nu\uu)\,\mathbf I\right),
\end{align*}
where 
\[
\nu := 
\left\{\begin{array}{c}
0,\quad y \in [0,5),\\[5pt]
1,\quad y\in [5,10].
\end{array}\right.
\]
Therefore when $y<5$, the model is purely elastic and when $y\ge5$ the model is a Zener viscoelastic model.    For the discretization in space we use $\mc P_2$ finite elements on a mesh of 30,720 tetrahedra obtained by partitioning a uniform quadrilateral mesh of $8\times 80\times 8$ elements. Discretization in time is done by TRCQ using 500 time-steps over the interval $[0,50]$. 

  Figures \ref{fig:with_stress} and \ref{fig:p2_refinement} show snapshots of the displacement from the simulation, where the coloring in the first one exhibits the norm of the stress (averaged on each tetrahedron). %with \textit{copper scaling}.
We remark that in both figures, while we choose the same snapshots, the snapshots are not uniform in time.  This is so we can highlight some of the more interesting aspects of the simulation which occur earlier in the time interval. 
\begin{figure}[H]
\centering
%trim left right bottom top 
\foreach \index in {1, ..., 8} {%
\includegraphics[clip,trim={1.4cm 0cm 2cm 0cm}, scale=0.23]{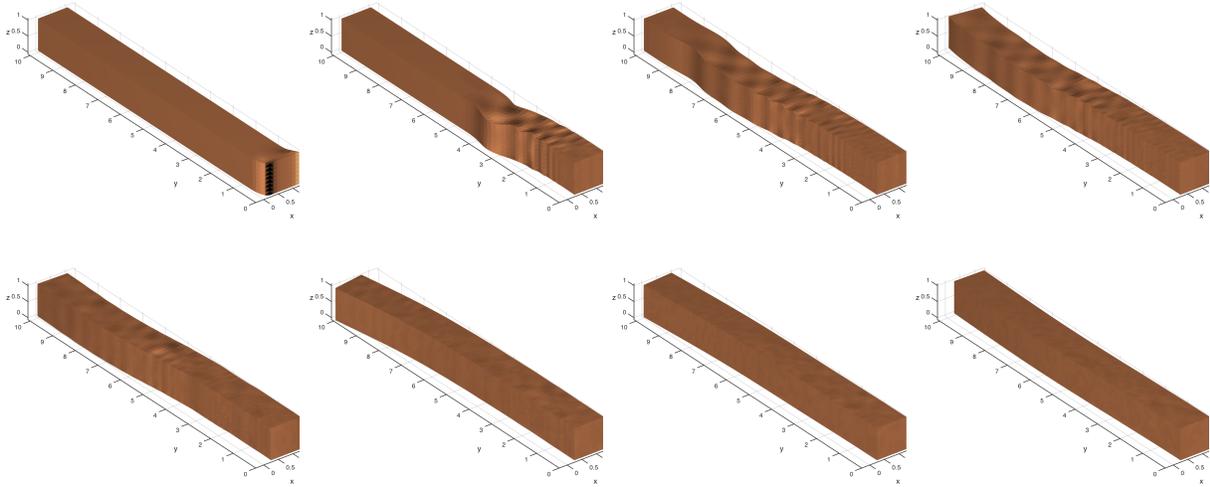}
}
\caption{{\footnotesize Snapshots for the 3D simulation showing the norm of the stress. 
From left to right, then from top to bottom, time-step = 9, 30, 55, 70, 100, 200, 350, 500.}}
\label{fig:with_stress}
\end{figure}

\begin{figure}[H]
\foreach \index in {1, ..., 8} {%
\includegraphics[clip,trim={1.4cm 0cm 2cm 0cm}, scale=0.23]{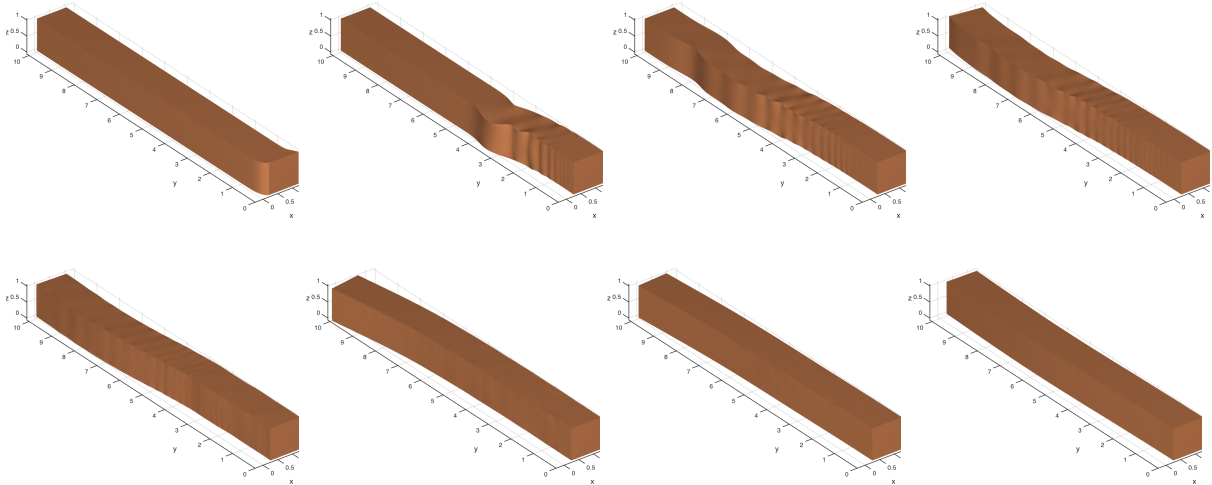}
}
\caption{{\footnotesize The results of the same simulation as Figure \ref{fig:with_stress}, without the colormap. From left to right, then from top to bottom, time-step = 9, 30, 55, 70, 100, 200, 350, 500.}}
\label{fig:p2_refinement}
\end{figure}

In the top rows of these figures, 
we observe the elastic waves, generated by the sudden deformation at $y=0$, propagating along the rod in the $y$ direction. We also note that the elastic part of the rod $(y<5)$ responds quickly to the sudden deformation and adjusts to the new displacement (enforced by the Dirichlet boundary condition) in about $70$ time-steps, while the viscoelastic part of the rod $(y>5)$ shows a much slower response to the change of the displacement taking  around $400$ time-steps to adjust. 

\bibliographystyle{Abbrv}
\bibliography{References}

\begin{thebibliography}{10}

\bibitem{AtPiStZo2014a}
T.~M. Atanackovi\'c, S.~Pilipovi\'c, B.~Stankovi\'c, and D.~s. Zorica.
\newblock {\em Fractional calculus with applications in mechanics}.
\newblock Mechanical Engineering and Solid Mechanics Series. ISTE, London; John
  Wiley \& Sons, Inc., Hoboken, NJ, 2014.
\newblock Wave propagation, impact and variational principles.

\bibitem{AtPiStZo2014b}
T.~M. Atanackovi\'c, S.~Pilipovi\'c, B.~Stankovi\'c, and D.~s. Zorica.
\newblock {\em Fractional calculus with applications in mechanics}.
\newblock Mechanical Engineering and Solid Mechanics Series. ISTE, London; John
  Wiley \& Sons, Inc., Hoboken, NJ, 2014.
\newblock Vibrations and diffusion processes.

\bibitem{BaHa1986}
A.~Bamberger and T.~H. Duong.
\newblock Formulation variationnelle espace-temps pour le calcul par potentiel
  retard\'e de la diffraction d'une onde acoustique. {I}.
\newblock {\em Math. Methods Appl. Sci.}, 8(3):405--435, 1986.

\bibitem{BaSc2012}
L.~Banjai and M.~Schanz.
\newblock Wave propagation problems treated with convolution quadrature and
  {BEM}.
\newblock In {\em Fast boundary element methods in engineering and industrial
  applications}, volume~63 of {\em Lect. Notes Appl. Comput. Mech.}, pages
  145--184. Springer, Heidelberg, 2012.

\bibitem{BeEzJo2004}
E.~B\'ecache, A.~Ezziani, and P.~Joly.
\newblock A mixed finite element approach for viscoelastic wave propagation.
\newblock {\em Comput. Geosci.}, 8(3):255--299, 2004.

\bibitem{CoGiMa2017a}
I.~Colombaro, A.~Giusti, and F.~Mainardi.
\newblock On the propagation of transient waves in a viscoelastic {B}essel
  medium.
\newblock {\em Z. Angew. Math. Phys.}, 68(3):Art. 62, 13, 2017.

\bibitem{CoGiMa2017b}
I.~Colombaro, A.~Giusti, and F.~Mainardi.
\newblock On transient waves in linear viscoelasticity.
\newblock {\em Wave Motion}, 74:191--212, 2017.

\bibitem{FaMo1992}
M.~Fabrizio and A.~Morro.
\newblock {\em Mathematical problems in linear viscoelasticity}, volume~12 of
  {\em SIAM Studies in Applied Mathematics}.
\newblock Society for Industrial and Applied Mathematics (SIAM), Philadelphia,
  PA, 1992.

\bibitem{GiCo2018}
A.~Giusti and I.~Colombaro.
\newblock Prabhakar-like fractional viscoelasticity.
\newblock {\em Commun. Nonlinear Sci. Numer. Simul.}, 56:138--143, 2018.

\bibitem{GuSt1962}
M.~E. Gurtin and E.~Sternberg.
\newblock On the linear theory of viscoelasticity.
\newblock {\em Arch. Rational Mech. Anal.}, 11:291--356, 1962.

\bibitem{HaSa2016}
M.~Hassell and F.-J. Sayas.
\newblock Convolution quadrature for wave simulations.
\newblock In {\em Numerical simulation in physics and engineering}, volume~9 of
  {\em SEMA SIMAI Springer Ser.}, pages 71--159. Springer, [Cham], 2016.

\bibitem{KeHe2014}
A.~Keramat and K.~Heidari~Shirazi.
\newblock Finite element based dynamic analysis of viscoelastic solids using
  the approximation of {V}olterra integrals.
\newblock {\em Finite Elem. Anal. Des.}, 86:89--100, 2014.

\bibitem{LaRaSa2015}
S.~Larsson, M.~Racheva, and F.~Saedpanah.
\newblock Discontinuous {G}alerkin method for an integro-differential equation
  modeling dynamic fractional order viscoelasticity.
\newblock {\em Comput. Methods Appl. Mech. Engrg.}, 283:196--209, 2015.

\bibitem{Lee2017}
J.~J. Lee.
\newblock Analysis of mixed finite element methods for the standard linear
  solid model in viscoelasticity.
\newblock {\em Calcolo}, 54(2):587--607, 2017.

\bibitem{LiZhLu2016}
H.~Li, Z.~Zhao, and Z.~Luo.
\newblock A space-time continuous finite element method for 2{D} viscoelastic
  wave equation.
\newblock {\em Bound. Value Probl.}, pages Paper No. 53, 17, 2016.

\bibitem{LuHa2004}
J.-F. Lu and A.~Hanyga.
\newblock Numerical modelling method for wave propagation in a linear
  viscoelastic medium with singular memory.
\newblock {\em Geophysical Journal International}, 159(2):688--702, 2004.

\bibitem{Lubich1994}
C.~Lubich.
\newblock On the multistep time discretization of linear initial-boundary value
  problems and their boundary integral equations.
\newblock {\em Numer. Math.}, 67(3):365--389, 1994.

\bibitem{Mainardi1982}
F.~Mainardi.
\newblock {\em Wave propagation in viscoelastic media}.
\newblock Pitman, 1982.

\bibitem{Mainardi2010}
F.~Mainardi.
\newblock {\em Fractional calculus and waves in linear viscoelasticity}.
\newblock Imperial College Press, London, 2010.
\newblock An introduction to mathematical models.

\bibitem{Mainardi2012}
F.~Mainardi.
\newblock An historical perspective on fractional calculus in linear
  viscoelasticity.
\newblock {\em Fract. Calc. Appl. Anal.}, 15(4):712--717, 2012.

\bibitem{Mainardi2018}
F.~Mainardi.
\newblock A note on the equivalence of fractional relaxation equations to
  differential equations with varying coefficients.
\newblock {\em Mathematics}, 6(1), 2018.

\bibitem{MaGo2000}
F.~Mainardi and R.~Gorenflo.
\newblock On {M}ittag-{L}effler-type functions in fractional evolution
  processes.
\newblock {\em J. Comput. Appl. Math.}, 118(1-2):283--299, 2000.
\newblock Higher transcendental functions and their applications.

\bibitem{MaCr2012}
S.~P. Marques and G.~J. Creus.
\newblock {\em Computational viscoelasticity}.
\newblock Springer Science \& Business Media, 2012.

\bibitem{MeCo2003}
A.~Mesquita and H.~Coda.
\newblock A two-dimensional bem/fem coupling applied to viscoelastic analysis
  of composite domains.
\newblock {\em International journal for numerical methods in engineering},
  57(2):251--270, 2003.

\bibitem{MuSaTo1999}
G.~A. Mu\~noz, Y.~Sarantopoulos, and A.~Tonge.
\newblock Complexifications of real {B}anach spaces, polynomials and
  multilinear maps.
\newblock {\em Studia Math.}, 134(1):1--33, 1999.

\bibitem{NaHo2013}
S.~P. N\"asholm and S.~Holm.
\newblock On a fractional {Z}ener elastic wave equation.
\newblock {\em Fract. Calc. Appl. Anal.}, 16(1):26--50, 2013.

\bibitem{Pazy1983}
A.~Pazy.
\newblock {\em Semigroups of linear operators and applications to partial
  differential equations}, volume~44 of {\em Applied Mathematical Sciences}.
\newblock Springer-Verlag, New York, 1983.

\bibitem{PeKa2014}
P.~Perdikaris and G.~E. Karniadakis.
\newblock Fractional-order viscoelasticity in one-dimensional blood flow
  models.
\newblock {\em Annals of biomedical engineering}, 42(5):1012--1023, 2014.

\bibitem{Phan2013}
N.~Phan-Thien.
\newblock {\em Understanding viscoelasticity}.
\newblock Graduate Texts in Physics. Springer-Verlag, Berlin, 2013.
\newblock An introduction to Rheology.

\bibitem{RiShWhWh2003}
B.~Rivi\`ere, S.~Shaw, M.~F. Wheeler, and J.~R. Whiteman.
\newblock Discontinuous {G}alerkin finite element methods for linear elasticity
  and quasistatic linear viscoelasticity.
\newblock {\em Numer. Math.}, 95(2):347--376, 2003.

\bibitem{Sayas2016}
F.-J. Sayas.
\newblock {\em Retarded potentials and time domain boundary integral equations:
  A road map}.
\newblock Springer, 2016.

\bibitem{ShNe2014}
V.~P. Shevchenko and R.~N. Neskorodev.
\newblock A numerical-analytical method for solving problems of linear
  viscoelasticity.
\newblock {\em Internat. Appl. Mech.}, 50(3):263--273, 2014.
\newblock Translation of Prikl. Mekh. {{\bf{5}}0} (2014), no. 3, 42--53.

\bibitem{Treves1967}
F.~Tr\`eves.
\newblock {\em Topological vector spaces, distributions and kernels}.
\newblock Academic Press, New York-London, 1967.

\bibitem{TrPe2014}
N.~Troyani and A.~P\'erez.
\newblock A comparison of a finite element only scheme and a {BEM}/{FEM} method
  to compute the elastic-viscoelastic response in composite media.
\newblock {\em Finite Elem. Anal. Des.}, 88:42--54, 2014.

\bibitem{YuPeKa2016}
Y.~Yu, P.~Perdikaris, and G.~E. Karniadakis.
\newblock Fractional modeling of viscoelasticity in 3{D} cerebral arteries and
  aneurysms.
\newblock {\em J. Comput. Phys.}, 323:219--242, 2016.

\end{thebibliography}

\end{document}